\documentclass[11pt,reqno]{amsart}

\usepackage{a4wide}
\usepackage{hyperref,amsmath,amsfonts,mathscinet,amssymb,amsthm}
\usepackage{accents}
\usepackage{enumitem}
\usepackage{constants}
\usepackage{xcolor}
\usepackage{tikz-cd}

\setlist[enumerate]{label={\rm(\roman*)}}

\theoremstyle{plain}
\newtheorem{theorem}{Theorem}[section]

\newtheorem{lemma}[theorem]{Lemma}

\theoremstyle{definition}
\newtheorem{definition}[theorem]{Definition}
\theoremstyle{remark}
\newtheorem{remark}[theorem]{Remark}
\numberwithin{equation}{section}

\expandafter\let\expandafter\oldproof\csname\string\proof\endcsname
\let\oldendproof\endproof
\renewenvironment{proof}[1][\proofname]{%
  \oldproof[\bf #1]%
}{\oldendproof}

\DeclareMathOperator*{\esup}{\sup}
\DeclareMathOperator{\Id}{Id}
\newcommand{\I}{\mathrm{I}}
\newcommand{\II}{\mathrm{II}}
\newcommand{\III}{\mathrm{III}}
\DeclareMathOperator{\RHS}{RHS}
\DeclareMathOperator{\LHS}{LHS}
\newcommand{\K}{\mathcal K}
\newcommand{\Z}{\mathbb Z}
\newcommand{\N}{\mathbb N}
\newcommand{\M}{\mathfrak M}
\newcommand{\Mpl}{\M^+}
\newcommand{\ksum}[3][k]{\sum_{{#1}={#2}}^{{#3}}}
\newcommand{\ksup}[3][k]{\sup_{{#2}\leq {#1} \leq {#3}}}
\newcommand{\LHSeq}[1]{\LHS\eqref{#1}}
\newcommand{\RHSeq}[1]{\RHS\eqref{#1}}

\def\ls{\lesssim}
\def\iq{\frac{1}{q}}

\def\ir{\frac{1}{r}}
\def\th{\widetilde{h}}

\allowdisplaybreaks

\begin{document}

\title{Embeddings Between Generalized Weighted Lorentz Spaces}

\author{Amiran Gogatishvili, Zden\v{e}k Mihula, Lubo\v s Pick, Hana Tur\v{c}inov\'{a} and Tu\u{g}\c{c}e \"{U}nver}

\email[A.~Gogatishvili]{gogatish@math.cas.cz}
\urladdr{ 0000-0003-3459-0355}
\email[Z~Mihula]{mihulzde@fel.cvut.cz}
\urladdr{0000-0001-6962-7635}
\email[L.~Pick]{pick@karlin.mff.cuni.cz}
\urladdr{0000-0002-3584-1454}
\email[H~Tur\v{c}inov\'{a}]{hana.turcinova@fel.cvut.cz}
\urladdr{0000-0002-5424-9413}
\email[T.~\"{U}nver]{tugceunver@kku.edu.tr}
\urladdr{0000-0003-0414-8400}

\address{ Institute of Mathematics of the Czech Academy of Sciences,  \v Zitn\'a~25,  115~67 Praha~1,  Czech Republic}

\address{Czech Technical University in Prague, Faculty of Electrical Engineering, Department of Mathematics, Technick\'a~2, 166~27 Praha~6, Czech Republic; AND Department of Mathematical Analysis, Faculty of Mathematics and Physics, Charles University, Sokolovsk\'a~83, 186~75 Praha~8, Czech Republic}

\address{Department of Mathematical Analysis, Faculty of Mathematics and Physics, Charles University, Sokolovsk\'a~83, 186~75 Praha~8, Czech Republic}

\address{Czech Technical University in Prague, Faculty of Electrical Engineering, Department of Mathematics, Technick\'a~2, 166~27 Praha~6, Czech Republic; AND Department of Mathematical Analysis, Faculty of Mathematics and Physics, Charles University, Sokolovsk\'a~83, 186~75 Praha~8, Czech Republic, ORCID 0000-0002-5424-9413}

\address{Department of Mathematics, Faculty of Engineering and Natural Sciences, Kirikkale University, 71450, Yahsihan, Kirikkale, T\"{u}rkiye}


\thanks{The research of  Amiran Gogatishvili was supported by the Czech Academy of Sciences RVO: 67985840, by the Czech Science Foundation (GA\v CR), grant no. 23-04720S, by the Shota Rustaveli National Science Foundation (SRNSF), grant no: FR22\_17770, and by the  grant Ministry of Education and Science of the Republic of Kazakhstan (project no. AP14869887). The research of Zden\v ek Mihula was supported by the project OPVVV CAAS CZ.02.1.01/0.0/0.0/16\_019/0000778 and by the grant P201/21-01976S of the Czech Science Foundation. The work of L.~Pick was supported in part by the Danube Region Grant no.~8X2043 and by the grant P201/21-01976S of the Czech Science Foundation. The research of Hana Tur\v cinov\'{a} was supported in part by the Grant Schemes at Charles University, reg.~no.~ \\CZ.02.2.69/0.0/0.0/19\_073/0016935, by the Primus research programme PRIMUS/21/SCI/002 of Charles University and Charles University Research program  No.~UNCE/SCI/023. The research of Tu\u{g}\c{c}e \"{U}nver was supported by the Grant of The Scientific and Technological Research Council of Turkey (TUBITAK), Grant No: 1059B192000075.}

\begin{abstract}
We give a new characterization of a continuous embedding between two function spaces of type $G\Gamma$. Such spaces are governed by functionals of type
\begin{equation*}
	\|f\|_{G\Gamma(r,q;w,\delta)} := \left(\int_{0}^{L} \left( \frac1{\Delta(t)} \int_0^t f^*(s)^r \delta(s) ds \right)^{\frac{q}{r}}
	w(t) dt \right)^\frac1{q},
\end{equation*}
in which $f^*$ is the nonincreasing rearrangement of $f$, $L\in(0,\infty]$, $r,q \in (0, \infty)$, $w, \delta$ are weights on $(0,L)$ and $\Delta(t)=\int_{0}^{t}\delta(s)\,ds$ for $t\in(0,L)$.
To characterize the embedding of such a space, say $G\Gamma(r_1,q_1;w_1,\delta_1)$, into another, $G\Gamma(r_2,q_2;w_2,\delta_2)$, means to find a balance condition on the four positive real parameters and the four weights in order that an appropriate inequality holds for every admissible function. We develop a new discretization technique which will enable us to get rid of restrictions on parameters imposed in earlier work such as the non-degeneracy conditions or certain relations between the $r$'s and $q$'s. Such restrictions were caused mainly by the use of duality techniques, which we avoid in this paper. On the other hand we consider here only the case when $q_1 \le q_2$, leaving the reverse case to future work.
\end{abstract}

\maketitle

\section{Introduction} 

Discretizing and antidiscretizing techniques have been successfully applied to solving several rather difficult problems in the function space theory that had looked almost impossible before. The method itself is technical and not very attractive, but it yields the desired results. Numerous dismal attempts to avoid it and to get equally strong results using different approaches have been tried heavily, most of them markedly unsuccessful.

In this paper, we have a different mission. Our aim is not to circumvent the discretization technique, but rather to enhance it, and to suggest a lateral point of view allowing one to overcome certain restrictions that have been littering it thus far. Roughly speaking, we are going to cleanse the
discretization method from several assumptions on weights involved that have been appearing regularly in earlier work, and which we now confute as unnecessary, a pivotal instance of these being the array of non-degeneracy conditions. As a result, we obtain a considerably stronger characterization of embeddings between $G\Gamma$-spaces, but the impact of the improvement is wider as it extends to natural applications of the embeddings obtained.

The earlier work \cite{Gog:03,Gog:14,Gog:17} made it clear that one of the main sources of the necessity for taking various restrictions was the use of duality as a crucial step in existing techniques. Our main achievement here is that the duality techniques are replaced by different ones based on exploiting the subtle interplay between discrete Hardy inequalities and the localization brought in by the discretization method, allowing us to obtain results of required generality and versatility.

Now is the time to be more precise. Let $(\mathcal{R},\mu)$ be a $\sigma$-finite nonatomic measure space such that $\mu(\mathcal{R})=L\in(0,\infty]$, and let $\mathfrak{M}(\mathcal{R},\mu)$ be the set of all $\mu$-measurable functions on $\mathcal{R}$ whose values belong to $[-\infty, \infty]$ and $\mathfrak{M}^+(\mathcal{R},\mu) = \{ f\in \mathfrak{M}(\mathcal{R},\mu) : f \geq 0\ \text{$\mu$-a.e.}\}$. By $f^*$ we denote the nonincreasing rearrangement of $f \in \mathfrak{M}(\mathcal{R},\mu)$ defined by
\begin{equation*}
	f^*(t) = \inf \{ \lambda \in [0,\infty)\colon \mu( \{ x\in \mathcal{R}\colon |f(x)| > \lambda \} )\leq t\}, \quad t\in (0,\infty).
\end{equation*}

If $X$ and $Y$ are (quasi-)Banach spaces of measurable functions on the same measure space and the identity operator $\Id$ is bounded from $X$ to $Y$ in the sense that there exists a positive constant $C$ such that $\| f \|_Y \leq C \|f\|_X$ for all $f \in X$, then we say that $X$ is \textit{embedded} into $Y$, a fact which we denote by $X \hookrightarrow Y$. The least such a constant $C$ is equal to $\|\Id\|_{X\rightarrow Y}$.

Let $r,q \in (0, \infty)$ and $w, \delta$ be \textit{weights} on $(0,L)$, that is, measurable functions on $(0,L)$ that are positive a.e.~on  $(0,L)$ and integrable near $0$. By integrable near $0$ we mean that $\int_0^t w(s) ds < \infty$ for every $t\in(0,L)$, and the same goes for $\delta$.
The \textit{generalized Gamma space} $G\Gamma(r,q;w,\delta)$ is the collection of all functions $f \in \mathfrak{M}(\mathcal{R},\mu)$ such that
\begin{equation*}
	\|f\|_{G\Gamma(r,q;w,\delta)} := \left(\int_{0}^{L} \left( \frac1{\Delta(t)} \int_0^t f^*(s)^r \delta(s) ds \right)^{\frac{q}{r}} w(t) dt \right)^\frac1{q} < \infty,
\end{equation*}
where we used the notation
\begin{equation*}
	\Delta(t) = \int_0^t \delta(s) ds\quad\text{for $t\in(0,L)$.}
\end{equation*}
We will use this convention throughout; for example, $\Delta_1(t)$ will denote $\int_0^t \delta_1(s) ds$,  $U(t)$ will denote $\int_0^t u(s) ds$, and so on.

The roots of generalized Gamma spaces reach the pivotal paper~\cite{Sa:90} by Sawyer, in which the spaces of type Gamma were first introduced in connection with duality questions for the so-called classical Lorentz spaces of type Lambda,  which had been introduced and studied earlier by Lorentz in~\cite{Lor:51}, and also in connection with action of classical integral operators of harmonic analysis on these spaces.  Sawyer's results unleashed a tsunami of papers, and it would be impossible to cite the whole lot of them here. Let us just recall certain important cornerstones of the theory. First, various weak versions of the Gamma-type spaces were studied in the early 1990s, see e.g.~\cite{Car:93,Car:93a}. In~\cite{Fio:08}, a simpler form of spaces $G\Gamma(r,q;w,\delta)$ (involving the outer weight but not the inner) was introduced, see also~\cite{Fio:09}. It did not go unnoticed that these spaces play a key role for the boundedness of Sobolev-type functions (in this connection see~\cite{Gog:14}), and, moreover, that they constitute a natural environment for seeking solutions to certain variation
inequalities. 

In~\cite{Fio:04}, it was observed that special cases of $G\Gamma$-spaces are equivalent to the so-called \textit{small Lebesgue spaces}, that had been defined in~\cite{Fio:00},  for $p\in(1,\infty)$, by
\begin{equation*}
	\|f\|_{L^{(p}}=\inf_{f=\sum f_k}\sum_{k=1}^{\infty}\inf_{0<\varepsilon<p'-1}\varepsilon^{-\frac{1}{p'-\varepsilon}}\|f_k\|_{(p'-\varepsilon)'},
\end{equation*}
where either it is assumed that the underlying domain is of measure 1 or a normalized norm is used, and $p'$ is the dual index of $p$. These spaces are naturally associated  with the
well-known \textit{grand Lebesgue spaces}, 
defined by
\begin{equation*}
	\|f\|_{L^{p)}}=\sup_{0<\varepsilon<p-1}\varepsilon^{\frac{1}{p-\varepsilon}}\|f\|_{p-\varepsilon},
\end{equation*}
introduced in~\cite{Iwa:92} in connection with the pointwise behavior of Jacobians and with classical discoveries of M\"uller~\cite{Mul:90} and Ball~\cite{Bal:76}. It was shown in~\cite{Cap:05}, that the small Lebesgue space $L^{(p'}$ is equivalent to the associate space (i.e., the K\"othe dual space) of the grand Lebesgue space $L^{p)}$. The main result of~\cite{Fio:04} tells us that one has, in fact,
\begin{equation*}
	\|f\|_{L^{(p}}\approx \int_{0}^{1}(1-\log t)^{-\frac1p}\left(\int_0^tf^*(s)^p\,ds\right)^{\frac1p}\frac{dt}{t},
\end{equation*}
and
\begin{equation*}
	\|f\|_{L^{p)}}\approx \sup_{0<t<1}(1-\log t)^{-\frac1p}\left(\int_t^1f^*(s)^p\,ds\right)^{\frac1p}\frac{dt}{t}.
\end{equation*}
So, the small Lebesgue space coincides with an appropriate particular case of a $G\Gamma$ space in the sense that they are equal in the set-theoretical sense, and their norms are equivalent. The statement of~\cite{Gog:14}*{Theorem~1.1} then shows a new characterization of the grand Lebesgue space in terms of various types of $G\Gamma$ spaces, depending on the parameters involved.

Direct applications of the $G\Gamma$-spaces to the study of the existence, uniqueness, and regularity of the so-called `very weak solutions' to Dirichlet and Neumann problems for the equation $-\Delta u = f$ in nonstandard function spaces can be found for example in~\cite{Rak:15}. The approach is related to the fact, proved in~\cite{Fio:05} and~\cite{Fio:08}, that a first-order Sobolev space built on the space $G\Gamma(p,m;w,1)(\Omega)$ is compactly embedded into $L^{\frac{np}{n-p}}(\Omega)$ if and only if $w\notin L^1(0,1)$, in which $p\in[1,n)$ and $\Omega$ is a sufficiently regular subdomain of the  ambient Euclidean space $\mathbb R^n$.

The spaces $G\Gamma$ play an interesting role in the interpolation theory, as pointed out in~\cite{Ahm:20,Fio:01,Fio:02}. In particular, \cite{Fio:02}*{Theorem~1.3} states that
\begin{equation*}
	(L^{p)},L^{(p})_{\theta,r} = G\Gamma(p, r; w_1, w_2), 
\end{equation*}
in which $w_1$ and $w_2$ are suitable power-logarithmic weights and $(\cdot,\cdot)_{\theta,r}$ denotes the standard $K$-method of real interpolation.

In~\cite{Gog:14}, K\"othe duals of simplified $G\Gamma$-spaces were studied, and the question of when they are Banach algebras was rounded off there. Some more connections, applications, and historical notes can be found in that paper, too.

Our aim here is to investigate embeddings between pairs of $G\Gamma$-spaces, that is,
\begin{equation*}
	G\Gamma(r_1, q_1; w_1, \delta_1) \hookrightarrow G\Gamma(r_2, q_2; w_2,  \delta_2).
\end{equation*}
This amounts to finding a balance condition that would characterize all parameters $r_i,q_i$, and weights $w_i,\delta_i$, $i=1,2$, for which there exists a positive constant $C$, depending possibly only on these parameters and weights, such that the inequality
\begin{equation}\label{E:main}
	\begin{split}
		&\left( \int_0^L \left( \frac1{\Delta_2(t)} \int_0^t f^*(s)^{r_2} \delta_2(s) ds \right)^{\frac{q_2}{r_2}} w_2(t) dt \right)^\frac1{q_2}
		\\
		&\qquad \leq C \left( \int_0^L \left( \frac1{\Delta_1(t)} \int_0^t f^*(s)^{r_1} \delta_1(s) ds \right)^{\frac{q_1}{r_1}} w_1(t) dt \right)^\frac1{q_1}
	\end{split}
\end{equation}
holds for every $\mu$-measurable function $f$. We begin the analysis by ridding of one of the parameters and expressing the inequality in an equivalent but  slightly simpler form. By a standard rescaling argument based on replacing $(f^{*})^{r_1}$ with $f^*$, and then denoting $r= r_2/r_1$, $q= q_2/r_1$, $p= q_1/r_1$, $u = \delta_1 $, $ \delta = \delta_2$, $v = w_1$ and $w = w_2$, we easily observe that~\eqref{E:main} is equivalent to
\begin{equation}\label{main1}
	\begin{split}
		&\left(\int_0^L \left( \frac1{\Delta(t)} \int_0^t f^*(s)^r \delta(s) ds \right)^{\frac{q}{r}} w(t) dt \right)^\frac1{q}
		\\
		&\qquad\leq C \left( \int_0^L \left( \frac1{U(t)} \int_0^t f^*(s) u(s) ds \right)^p v(t) dt \right)^\frac1{p},
	\end{split}
\end{equation}
again with $C$ universal for any $f$.

Let us concentrate on the inequality \eqref{main1}. The first step relaxed a little the number of dangers to worry about, but we still face a more serious problem which consists in the fact that the inequality is formulated for symmetrized versions of functions. Put another way, it constitutes a~weighted inequality restricted to nonincreasing functions. Since such a restriction makes inequalities notoriously hard to manage, our next step will be a reduction of~\eqref{main1} to an unrestricted equivalent inequality. However, we will pay for the reduction by the appearance of one more integral operator.

Assume that $0 < p,q,r < \infty$ and let $v,w, u, \delta$ be weights on $(0,L)$. Then the inequality \eqref{main1} holds if and only if there exists a positive constant $C$ such that the inequality
\begin{equation}\label{main2}
	\begin{split}
		&\left( \int_0^L \left( \frac1{\Delta(t)} \int_0^t \left( \int_s^L h  \right)^r \delta(s) ds \right)^{\frac{q}{r}} w(t) dt \right)^{\frac1{q}}
		\\
		&\leq C \left( \int_0^L \left( \frac1{U(t)} \int_0^t \left( \int_s^L h  \right) u(s) ds \right)^p v(t) dt \right)^{\frac1{p}}
	\end{split}
\end{equation}
holds for all $h \in \mathfrak{M}^+(0,L)$. The equivalence between~\eqref{main1} and~\eqref{main2} is quite standard. Indeed, the fact that \eqref{main1} implies \eqref{main2} amounts to
finding, to a given nonnegative function $h\colon(0,L)\to[0,\infty)$, a function $f\colon \mathcal R\to[0,\infty)$ such that $f^*(s)=\int_s^L h$ for almost every $s\in(0,L)$, which is possible owing to the classical Sierpi\'nski theorem (see~\cite{BS}*{Chapter~2, Corollary~7.8}). Conversely, assume that \eqref{main2} holds.  For every $f \in \mathfrak{M}(\mathcal{R},
\mu)$, there exists a sequence $\{g_n\}_{n=1}^\infty$ of nonnegative measurable functions whose support is bounded and such that the sequence
$\left\{\int_t^{\infty} g_n(s) ds\right\}_{n=1}^\infty$ is nondecreasing in $n$ for every fixed $t>0$, and  $\lim_{n\rightarrow \infty} \int_t^{\infty} g_n(s) ds = f^*(t) $ for almost all $t >0$ (\cite{GogStep}*{Proposition~2.1}). Then, using the monotone convergence theorem, we get \eqref{main1}.

So now it is \eqref{main2} which we have to worry about. From now on, we shall denote by $C$ the optimal (smallest) constant in \eqref{main2}. This can be formally written as
\begin{equation}\label{E:C}
	C = \sup_{h\in\Mpl(0,L)} \frac {\left( \int_0^L \left( \frac1{\Delta(t)} \int_0^t \left( \int_s^L h  \right)^r \delta(s) ds \right)^{\frac{q}{r}} w(t) dt \right)^{\frac1{q}}} {\left( \int_0^L \left( \frac1{U(t)} \int_0^t \left( \int_s^L h  \right) u(s) ds \right)^p v(t) dt \right)^{\frac1{p}}}.
\end{equation}
The ultimate task is to establish two-sided estimates of $C$ in terms of quantities defined in an easily computable way and dependent solely on parameters and weights. In fact, it is quite remarkable that something like that is possible at all. As always in the theory of weighted inequalities, the form of the characterizing expressions will heavily depend on the comparison of the parameters $p,q,r$, inevitably forcing us to split the result into several cases. In this paper, we handle the `convex' variant of the inequality, that is, the case $p\le q$ (which corresponds to the relation $q_1\le q_2$ in~\eqref{E:main}, mentioned in the abstract).

We finally introduce auxiliary nonnegative functions $\varphi$ and $\sigma$ by setting
\begin{equation}\label{varphi}
	\varphi(t) = \int_0^L \min\{U(t)^p, U(s)^p\} \frac{v(s)}{U(s)^p} ds,\ t\in(0,L),
\end{equation}
and
\begin{equation}\label{E:sigma}
	\sigma(t)=\varphi(t)^{-\frac{r}{p-r}-2}V(t)\left(\int_t^L U^{-p}(s)v(s) ds\right)U(t)^{p-1}u(t),\ t\in(0,L),
\end{equation}
in order to simplify the statement of our main result.

The key part of the statement of the main result will be formulated in the form of a two-sided estimate between two constants, one of which being $C$ (the best constant in~\eqref{main2}, formally defined by~\eqref{E:C}), and the other, say $B$, being some other quantity, for which an explicit formula using the weights $u,v,w$ and $\delta$ will be given, usually in the form of a combination of certain integrals and suprema. Before stating the main result let us explain the meaning of such a two-sided estimate. By $C\lesssim B$, we mean that $C \le \lambda B$ with some positive constant $\lambda$ independent of
appropriate quantities. If $C\lesssim B$ and $B\lesssim C$, we write $C \approx B$ and say that $C$ and $B$ are equivalent. 

Now we are in a position to state the principal result of this paper. 

\begin{theorem}\label{T:main}
	Let $0<p\leq q<\infty$, $0<r<\infty$ and $u, \delta, v, w$ be weights on $(0,L)$. Assume that there is $t_0\in(0,L)$ such that $0 < \varphi(t_0) < \infty$. Then $C$ defined by~\eqref{E:C} satisfies the following relations.
	
	\rm{(i)} If $p\leq q$, $p\leq r$, $1\leq q$, $1\leq r$, then
	\begin{equation*}
		C \approx B_1 + B_2,
	\end{equation*}
	where
	\begin{align*}
		B_1 &:= \esup_{t\in(0,L)} W(t)^{\frac1q}\varphi(t)^{-\frac1p},\\
		B_2 &:= \esup_{t\in(0,L)} \Delta(t)^{\frac1r}\varphi(t)^{-\frac1p}\left(\int_t^L\Delta^{-\frac{q}{r}}w\right)^{\frac{1}{q}}.\\
	\end{align*}
	
	\rm{(ii)} If $p \leq r<1\leq q$, then
	\begin{equation*}
		C \approx B_1 + B_2 + B_3,
	\end{equation*}
	where
	\begin{align*}
		B_3
		&:=
		\esup_{t\in(0,L)} \left(\int_t^L\Delta^{-\frac{q}{r}}w\right)^{\frac{1}{q}} \esup_{s\in(0,t)}U(s)\varphi(s)^{-\frac1p}\left(\int_s^t\Delta^{\frac{r}{1-r}}\delta U^{-\frac{r}{1-r}}\right)^{\frac{1-r}{r}}.
	\end{align*}
	
	\rm{(iii)}
	If $1\le r<p\le q$, then
	\begin{equation*}
		C \approx B_1 + B_2 + B_4,
	\end{equation*}
	where
	\begin{align*}
		B_4
		&:=
		\esup_{t\in(0,L)} \left(\int_t^L\Delta^{-\frac{q}{r}}w\right)^{\frac{1}{q}}
		\left(\int_{0}^{t}\sigma(s)U(s)^{\frac{pr}{p-r}}\esup_{\tau\in (s,t)}\Delta(\tau)^{\frac{p}{p-r}}U(\tau)^{-\frac{pr}{p-r}}\, ds\right)^{\frac{p-r}{pr}}.
	\end{align*}
	
	\rm{(iv)} If $r<p\le q$, $r<1\le q$, then
	\begin{equation*}
		C \approx B_1 + B_2 + B_3 + B_5,
	\end{equation*}
	where
	\begin{align*}
		B_5 &:= \esup_{t\in(0,L)} \left(\int_t^L\Delta^{-\frac{q}{r}}w\right)^{\frac{1}{q}} 
		\left(\int_{0}^{t}\sigma(s)\left(\int_{0}^{t}\Delta(\tau)^{\frac{1}{1-r}}\delta(\tau) U(\tau)^{-\frac{r}{1-r}}\right.\right.\\
		&\qquad\times
		\left.\left.\min\left\{U(s)^{\frac{r}{1-r}},U(\tau)^{\frac{r}{1-r}}\right\}\,d\tau\right)^{\frac{p(1-r)}{p-r}}\, ds\right)^{\frac{p-r}{pr}}.
	\end{align*}
	
	\rm{(v)} If $p\leq q<1\le r$, then
	\begin{equation*}
		C \approx B_1 + B_2 + B_6 + B_7,
	\end{equation*}
	where
	\begin{align*}
		B_6
		&:=
		\esup_{t\in(0,L)}         U(t)\varphi(t)^{-\frac1p}\left(\int_t^L\left(\int_s^L\Delta^{-\frac{q}{r}}w\right)^{\frac{q}{1-q}}\Delta(s)^{-\frac{q}{r}}w(s) \right.\\
		&\hspace{4cm} \left. \times \esup_{\tau\in(t,s)}\Delta(\tau)^{\frac{q}{r(1-q)}}U(\tau)^{-\frac{q}{1-q}}\,ds\right)^{\frac{1-q}{q}}, \\
		B_7
		&:=
		\esup_{t\in(0,L)} U(t)\varphi(t)^{-\frac1p}\left(\int_t^L W^{\frac{q}{1-q}}wU^{-\frac{q}{1-q}}\right)^{\frac{1-q}{q}}.
	\end{align*}
	
	\rm{(vi)} If $p\leq q<1$, $p\le r<1$, then
	\begin{equation*}
		C \approx B_1 + B_2 + B_3 + B_7 + B_8,
	\end{equation*}
	where
	\begin{align*}
		B_8
		&:=
		\esup_{t\in(0,L)}         U(t)\varphi(t)^{-\frac1p}\left(\int_t^L\left(\int_s^L\Delta^{-\frac{q}{r}}w\right)^{\frac{q}{1-q}}\Delta(s)^{-\frac{q}{r}}w(s) \right.\\
		&\hspace{4cm} \left. \times \left(\int_t^s\Delta^{\frac{r}{1-r}}\delta U^{-\frac{r}{1-r}}\right)^{\frac{q(1-r)}{r(1-q)}}\,ds\right)^{\frac{1-q}{q}}.
	\end{align*}
	
	\rm{(vii)} If $r<p\leq q<1$, then
	\begin{equation*}
		C \approx B_1 + B_2 + B_3 + B_5 + B_7 + B_8.
	\end{equation*}
	
	Moreover, the multiplicative constants in all the equivalences above depend only on $p,q,r$.
\end{theorem}

The first particular result in this direction was obtained in~\cite{Gog:17} under the restriction $q_2\ge r_2$ in~\eqref{E:main}, which translates to $q\ge r$ in~\eqref{E:C}. It is also stated there that the solution in the converse case is left as an open problem. In this paper, we solve this problem, at least for the convex variant of the inequality.

Let us summarize the content of the following sections. Elements of the discretization technique are collected in Section~\ref{S:preliminaries}. Fine analysis of indispensable discrete inequalities is carried out in Section~\ref{S:dicrete-inequalities}. The converse process of antidiscretization is the content of Section~\ref{S:antidiscretization}. Finally, in the last Section~\ref{S:proof}, we prove Theorem~\ref{T:main}.

\section{Preliminaries}\label{S:preliminaries}

In this section, we shall fix the notation and recall preliminary results. It is essentially borrowed from \cite{KMT}*{Section~2}, which draws from \cite{EGO-book}, and we include it to make this paper as self-contained as possible.

Throughout the entire paper, $L\in(0,\infty]$ is fixed. We say that a positive function defined on $(0,L)$ is \emph{admissible} if it is increasing and continuous. In this section, we shall assume that $\varrho$ is an admissible function. A function $h\colon(0,L)\to[0,\infty)$ is said to be \emph{$\varrho$-quasiconcave} if $h$ is nondecreasing on $(0,L)$ and the function $\frac{h}{\varrho}$ is nonincreasing on $(0,L)$. If this is the case, we write $h\in Q_\varrho(0,L)$. Let $h$ denote a function from $Q_\varrho(0,L)$ in the rest of this section. Thanks to the monotonicity properties of $\varrho$-quasiconcave functions, $h$ does not vanish identically on $(0,L)$ if and only if $h(t)\neq0$ for every $t\in(0,L)$.
Note that $h^p$ is a $\varrho^p$-quasiconcave function for every $p>0$, and so is $\frac{\varrho}{h}$ provided that $h\not\equiv0$. A nonnegative linear combination of $\varrho$-quasiconcave functions is a $\varrho$-quasiconcave function. Furthermore, if $k\in\N$ and $h_j\in Q_{\varrho_j}(0,L)$ for $j=1,2,\dots, k$, where each $\varrho_j$ is admissible, then the product $h_1h_2\cdots h_k$ is a $(\varrho_1\varrho_2\cdots \varrho_k)$-quasiconcave function.

\begin{definition}
	Let $M,N\in\Z \cup \{-\infty, \infty\}$ such that $-\infty\le N\le0\le M\le\infty$, $a\in(1,\infty)$ and  $h\in Q_{\rho}(0,L)$ such that $h\not\equiv0$. An increasing sequence $\{x_k\}_{k=N}^{M}\subseteq[0,L]$ is called a \emph{covering sequence for $h$, $\varrho$ and $a$} if it satisfies the following six properties.
	\begin{enumerate}
		
		\item $M=\infty$ if and only if
		\begin{equation*}
			\lim_{t\to L^-}h(t)=\infty\quad\text{and}\quad\lim_{t\to L^-}\frac{\varrho(t)}{h(t)}=\infty.
		\end{equation*}
		If $M=\infty$, then $\lim_{k\to\infty}x_k=L$. Otherwise, $x_M=L$.
		
		\item $N=-\infty$ if and only if
		\begin{equation*}
			\lim_{t\to 0^+}h(t)=0\quad\text{and}\quad\lim_{t\to 0^+}\frac{\varrho(t)}{h(t)}=0.
		\end{equation*}
		If $N=-\infty$, then $\lim_{k\to-\infty}x_k=0$. Otherwise, $x_{N}=0$.
		
		\item For every $k\in\Z$ such that $N+2 \le k \le M - 1$, one has
		\begin{equation*}
			a h(x_{k-1})\le h(x_k)\quad\text{and}\quad a\frac{\varrho(x_{k-1})}{h(x_{k-1})}\le\frac{\varrho(x_k)}{h(x_k)}.
		\end{equation*}
		
		\item For every $k\in\Z$ such that $N+2 \le k \le M - 1$, one has
		\begin{align*}
			\frac1{a}h(x_k)\le h(t)\le h(x_k)\quad&\text{for each $t\in[x_{k-1},x_k]$}\\
			\intertext{or}
			\frac1{a}\frac{\varrho(x_k)}{h(x_k)}\le\frac{\varrho(t)}{h(t)}\le\frac{\varrho(x_k)}{h(x_k)}\quad&\text{for each $t\in[x_{k-1},x_k]$}.
		\end{align*}
		
		\item If $M < \infty$, then
		\begin{align*}
			h(x_{M - 1})\le h(t)\le a h(x_{M - 1})\quad&\text{for each $t\in[x_{M - 1},L)$}\\
			\intertext{or}
			\frac{\varrho(x_{M - 1})}{h(x_{M - 1})}\le \frac{\varrho(t)}{h(t)}\le a \frac{\varrho(x_{M - 1})}{h(x_{M - 1})}\quad&\text{for each $t\in[x_{M - 1},L)$}.
		\end{align*}
		
		\item If $N>-\infty$, then
		\begin{align*}
			\frac1{a}h(x_{N+1})\le h(t)\le h(x_{N+1})\quad&\text{for each $t\in(0, x_{N+1}]$}\\
			\intertext{or}
			\frac1{a}\frac{\varrho(x_{N+1})}{h(x_{N+1})}\le \frac{\varrho(t)}{h(t)}\le\frac{\varrho(x_{N+1})}{h(x_{N+1})}\quad&\text{for each $t\in(0, x_{N+1}]$}.
		\end{align*}
	\end{enumerate}
	We denote the set of all covering sequences for $h,$ $\varrho$ and $a$ by $CS(h,\varrho,a)$.
\end{definition}

Note that all sequences in $CS(h,\varrho,a)$ share the same values of $N$ and $M$. Moreover, it is independent of the parameter $a$ whether $N$ and $M$ are finite or infinite.
If $\{x_k\}_{k=N}^{M}\in CS(h,\varrho,a)$, then $\{x_k\}_{k=N}^{M}\in CS(\frac{\varrho}{h},\varrho,a)$, and $\{x_k\}_{k=N}^{M}\in CS(h^p,\varrho^p,a^p)$ for every $p\in(0,\infty)$. Furthermore, it follows from the properties of covering sequences that
\begin{equation*}\label{prel:CSunion}
	(0,L)\subseteq\bigcup_{k=N+1}^M(x_{k-1},x_k]\subseteq(0,L];
\end{equation*}
moreover, the first inclusion is strict if and only if $M \neq \infty$.

\begin{lemma}[\cite{EGO-book}*{Lemma~3.2.5}] \label{Lem:Decomp}
	Let $M,N\in\Z\cup\{-\infty, \infty\}$ such that $-\infty\le N\le0\le M\le\infty$, $a\in(1,\infty)$ and  $h\in Q_{\rho}(0,L)$ such that $h\not\equiv0$ and $\{x_k\}_{k=N}^{M}\in CS(h,\varrho,a)$. The index set $\K^+=\{k\in\Z\colon N+1\leq k \leq M\}$ can be decomposed into $\K^+ = \mathcal{Z}_1\cup\mathcal{Z}_2$, where $\mathcal{Z}_1\cap\mathcal{Z}_1=\emptyset$, in such a way that
	\begin{align*}
		h(t)\approx h(x_k)\quad&\text{for all $t\in[x_{k-1},x_k]$ and every $k\in\mathcal{Z}_1$}, \\
		\intertext{and}
		\frac{\varrho(t)}{h(t)}\approx \frac{\varrho(x_k)}{h(x_k)}\quad&\text{for all $t\in[x_{k-1},x_k]$ and every $k\in\mathcal{Z}_2$}, 
	\end{align*}
	in which the equivalence constants depend only on the parameter $a$. 
\end{lemma}

The interested reader can find the construction of covering sequences and proofs of their properties in \cite{EGO-book}*{Chapter~3}.

We shall conclude this section by recalling a result which, in a way, bridges the divide between the discrete world and the continuous one. Let $p>0$ and $\widetilde w\in\M^+(0,L)$. Set
\begin{equation*}
	\widetilde{\varphi}(t)= \int_0^L \min\{\varrho(t), \varrho (s)\} \widetilde w(s) ds,\ t\in(0,L),
\end{equation*}
and assume that there is $t_0\in(0,L)$ such that $0<\widetilde\varphi(t_0)<\infty$. It is easy to see that $\widetilde\varphi\in Q_\varrho(0,L)$. Let $\{x_k\}_{k=N}^M\in CS(\widetilde\varphi,\varrho,a)$ with $a$ being large enough, namely $a>108$. It follows from \cite{KMT}*{Lemma~3.4 with $\alpha=\beta=0$, see also (4.3)} that
\begin{equation}\label{prel:dis-antidis}
	\begin{aligned}
		\int_0^L \left(\int_0^L \frac{\varrho(t)^\frac1{p} g(s)}{\varrho(t)^\frac1{p} + \varrho(s)^\frac1{p}} ds\right)^p \widetilde w(t) dt &\approx \sum_{k=N}^{M} \widetilde\varphi(x_k) \left(\int_0^L \frac{g(t)}{\varrho(x_k)^\frac1{p} + \varrho(t)^\frac1{p}} dt\right)^p \\
		&\approx \sum_{k = N + 1}^{M} \left(\int_{x_{k-1}}^{x_k} \frac{\widetilde\varphi(t)^\frac1{p}}{\varrho(t)^\frac1{p}} g(t) dt\right)^p
	\end{aligned}
\end{equation}
for every $g\in\M^+(0,L)$, in which the multiplicative constants depend only on $a$ and $p$.

Moreover, it follows from \cite{KMT}*{Lemma~3.5} that
\begin{equation}\label{KMT-Lemma3.5}
	\begin{aligned}
		\sup_{t\in(0,L)} \widetilde\varphi(t) \left(\int_0^L \frac{g(s)}{\varrho(t)^\frac1{p} + \varrho(s)^\frac1{p}} ds\right)^p &\approx \sup_{N\leq k \leq M} \widetilde\varphi(x_k) \left(\int_0^L \frac{g(t)}{\varrho(x_k)^\frac1{p} + \varrho(t)^\frac1{p}} dt\right)^p \\
		&\approx \sup_{N + 1\leq k \leq M} \left(\int_{x_{k-1}}^{x_k} \frac{\widetilde\varphi(t)^\frac1{p}}{\varrho(t)^\frac1{p}} g(t) dt\right)^p
	\end{aligned}
\end{equation}
for every $g\in\M^+(0,L)$, in which the multiplicative constants depend only on $a$ and $p$. The assumption on $a$, which is dictated by the assumptions of \cite{KMT}*{Lemmas~3.4-3.5}, is merely technical and not restrictive at all.

Let $N,M \in \Z\cup\{-\infty, \infty\}$, $N<M$, and $\{\varrho_k\}_{k=N}^M$ be a sequence of positive numbers. We say that $\{\varrho_k\}_{k=N}^M$ is \emph{strongly increasing} or \emph{strongly decreasing} if
\begin{equation}
	\inf\bigg\{ \frac{\varrho_{k+1}}{\varrho_k}\colon N\leq k < M\bigg\} > 1 \label{EQ:strongly_increasing_parameter}
\end{equation}
or
\begin{equation}
	\sup\bigg\{ \frac{\varrho_{k+1}}{\varrho_k}\colon N\leq k < M\bigg\} < 1, \label{EQ:strongly_decreasing_parameter}
\end{equation}
respectively. We shall frequently use the following equivalences involving strongly monotone sequences. Let $\{a_k\}_{k=N}^M$ be a sequence of nonnegative numbers and $p>0$. If $\{\varrho_k\}_{k=N}^M$ is strongly increasing, then
\begin{align}
	\ksum{N}{M}{\varrho_k \left(\ksum[i]{k}{M} a_i\right)^p} &\approx \ksum{N}{M}{\varrho_k a_k^p}, \label{EQ:strongly_increasing_sum_sum}\\
	\ksum{N}{M}{\varrho_k \left(\sup_{k\leq i\leq M} a_i\right)^p} &\approx \ksum{N}{M}{\varrho_k a_k^p} \label{EQ:strongly_increasing_sum_sup}\\
	\intertext{and}
	\sup_{N\leq k\leq M}\varrho_k \left(\ksum[i]{k}{M}{a_i}\right)^p &\approx \sup_{N\leq k\leq M}\varrho_k a_k^p. \label{EQ:strongly_increasing_sup_sum}
\end{align}
If $\{\varrho_k\}_{k=N}^M$ is strongly decreasing, then
\begin{align}
	\ksum{N}{M}{\varrho_k \left(\ksum[i]{N}{k} a_i\right)^p} &\approx \ksum{N}{M}{\varrho_k a_k^p} \label{EQ:strongly_decreasing_sum_sum}\\
	\intertext{and}
	\sup_{N\leq k\leq M}\varrho_k \left(\ksum[i]{N}{k}{a_i}\right)^p &\approx \sup_{N\leq k\leq M}\varrho_k a_k^p. \label{EQ:strongly_decreasing_sup_sum}
\end{align}
Moreover, all the equivalence constants depend only on the value of \eqref{EQ:strongly_increasing_parameter} or \eqref{EQ:strongly_decreasing_parameter} and $p$. Such inequalities involving strongly monotone sequences are classical; e.g., see \cite{GHS}*{Proposition~2.1} (cf.~\cite{L:76, L:93}).

Let $0 < p\le q < \infty$, $\{b_k\}_{k=N}^M$ be sequence of nonnegative numbers, $N,M\in\Z\cup\{-\infty, \infty\}$, $N<M$. By Landau theorem
(e.g., \cite{EGO-book}*{Lemma~1.4.1}),
\begin{equation} \label{eq:landau}
	\sup_{\{a_k\}_{k=N}^M}
	\frac{\left( \ksum{N}{M} a_k^q b_k^q \right)^{\frac{1}{q}}}{ \left( \ksum{N}{M} a_k^p \right)^{\frac1{p}}}=\sup_{N\leq k \leq M}  b_k, 
\end{equation}
where the supremum extends over all sequences ${\{a_k\}_{k=N}^M}$ of nonnegative numbers.

Finally, we shall also make use of the following equivalent expression for optimal constants in discrete Hardy inequalities with weights. Let $0 < p,q,r < \infty$, $\{d_k\}_{k=N}^M$ and $\{b_k\}_{k=N}^M$ be sequences of nonnegative numbers, $N,M\in\Z\cup\{-\infty, \infty\}$, $N<M$. Set
\begin{equation*}
	D = \sup_{\{a_k\}_{k=N}^M} \frac{\left( \ksum{N}{M} \left( \ksum[i]{N}{k} a_i^r b_i \right)^\frac{q}{r} d_k \right)^\frac1{q}}{\left( \ksum{N}{M} a_k^p \right)^\frac1{p}},
\end{equation*}
where the supremum extends over all sequences $\{a_k\}_{k=N}^M$ of nonnegative numbers. Owing to \cite{Be:91}, we have
\begin{equation}\label{eq:Bennett_discr_Hardy}
	D \approx \begin{cases}
		\sup_{N\leq k \leq M} \left(\ksum[i]{k}{M} d_i\right)^\frac1{q} b_k^\frac1{r} \qquad &\text{if $p\leq\min\{r,q\}$},\\
		\sup_{N\leq k \leq M} \left( \ksum[i]{k}{M} d_i \right)^\frac1{q} \left( \ksum[i]{N}{k} b_i^\frac{p}{p-r} \right)^\frac{p-r}{pr} \qquad &\text{if $r < p \leq q$},
	\end{cases}
\end{equation}
in which the equivalence constants depend only on $p,q$ and $r$.

\section{Equivalent Discrete Inequalities}\label{S:dicrete-inequalities}

We use the abbreviation $\LHS(*)$ and $\RHS(*)$ for the left-hand side and right-hand side of the inequality numbered by $(*)$, respectively.

We start with an auxiliary lemma.
\begin{lemma}\label{lem:disc_main}
	Let $0< p,q,r < \infty$ and $u, \delta, v, w$ be weights on $(0,L)$. Let $\varphi$ be the function from~\eqref{varphi}.
	Assume that there is $t_0\in(0,L)$ such that $0 < \varphi(t_0) < \infty$. Let $\{x_k\}_{k=N}^M \in CS(\varphi, U^p, a)$ with $a>108$. Denote by $\widetilde{C_i}$, $i=1,2,3,4$, the optimal constants in the inequalities:
	\begin{align}
		&\left( \ksum{N+1}{M} \int_{x_{k-1}}^{x_k} \left( \frac1{\Delta(t)} \int_{x_{k-1}}^t \left( \int_s^{x_k} h\right)^r \delta(s) ds \right)^{\frac{q}{r}} w(t) dt \right)^{\frac1{q}}
		\nonumber \\
		& \hspace{4cm} 	\leq \widetilde{C_1} \left( \ksum{N+1}{M} \left(\int_{x_{k-1}}^{x_k} \varphi^{\frac1{p}} h\right)^p \right)^{\frac1{p}} \label{M1}\\
		& \text{for every} \,\,  h\in\Mpl(0,L); \notag\\
		&\left( \ksum{N+1}{M-1}  \left( \ksum[i]{N+1}{k} \int_{x_{i-1}}^{x_i} \left( \int_s^{x_i} h\right)^r \delta(s) ds \right)^{\frac{q}{r}} \left( \int_{x_k}^{x_{k+1}} \Delta^{-\frac{q}{r}}w \right) \right)^{\frac1{q}} \nonumber\\
		& \hspace{4cm} 	\leq \widetilde{C_2} \left( \ksum{N+1}{M-1} \left(\int_{x_{k-1}}^{x_k} \varphi^{\frac1{p}} h\right)^p \right)^{\frac1{p}} \label{M2}\\
		& \text{for every} \,\,  h\in\Mpl(0,x_{M-1}); \notag
		\\
		&\left( \ksum{N+1}{M-1}  \left( \ksum[i]{N+1}{k}  \left( \ksum[j]{i}{k} \int_{x_{j}}^{x_{j+1}} h\right)^r \left( \int_{x_{i-1}}^{x_i}\delta \right) \right)^{\frac{q}{r}} \left( \int_{x_k}^{x_{k+1}} \Delta^{-\frac{q}{r}}w \right) \right)^{\frac1{q}}\nonumber \\
		& \hspace{4cm} 	\leq \widetilde{C_3} \left( \ksum{N+1}{M-1} \left(\int_{x_{k}}^{x_{k+1}} \varphi^{\frac1{p}} h\right)^p \right)^{\frac1{p}} \label{M3}\\
		&\text{for every} \,\,  h\in\Mpl(x_{N+1},L); \notag
		\\
		&\left( \ksum{N+1}{M-1} \left( \ksum[i]{k}{M-1}  \int_{x_i}^{x_{i+1}} h  \right)^q \left( \int_{x_{k-1}}^{x_k} w \right) \right)^{\frac1{q}}\leq \widetilde{C_4} \left( \ksum{N+1}{M-1} \left(\int_{x_{k}}^{x_{k+1}} \varphi^{\frac1{p}} h\right)^p \right)^{\frac1{p}} \label{M4}\\
		& \text{for every} \,\,  h\in\Mpl(x_{N+1},L). \notag
	\end{align}
	Then the $C$ defined by~\eqref{E:C} satisfies
	\begin{equation*}
		C\approx \widetilde{C_1} + \widetilde{C_2} + \widetilde{C_3} + \widetilde{C_4},
	\end{equation*}
	in which the equivalence constants depend only on the parameters $p,q,r$ and $a$.
\end{lemma}

\begin{proof}
	First, since
	\begin{equation}\label{min-equiv}
		\frac{1}{U(s)+U(t)} \approx \min\left\{\frac1{U(t)}, \frac1{U(s)}\right\} \quad \text{for every $s,t\in(0,L)$,}
	\end{equation}
	we have
	\begin{equation*}
		\int_0^t \left( \int_s^L h \right) u(s) ds \approx \int_0^L \frac{U(s)U(t)}{U(s) + U(t)}h(s) ds \quad \text{for every $t\in(0,L)$}.
	\end{equation*}
	Then
	\begin{equation*}
		\RHSeq{main2} \approx  \left( \int_0^L \left( \int_0^L \frac{U(s)U(t)}{U(s) + U(t)}  h(s)ds \right)^p  \frac{v(t)}{U(t)^p} dt \right)^\frac1{p}
	\end{equation*}
	and, applying~\eqref{prel:dis-antidis} to $\varrho=U^p$, $g = hU$, $\widetilde w = \frac{v}{U^p}$, and $\widetilde\varphi=\varphi$, we obtain
	\begin{equation}\label{lem:disc_main:discRHS}
		\RHSeq{main2} \approx  \left( \ksum{N+1}{M} \left( \int_{x_{k-1}}^{x_k} \varphi^{\frac1{p}} h \right)^p \right)^{\frac1{p}}.
	\end{equation}
	Next,
	\begin{align}\label{lem:disc_main:discLHS}
		\LHSeq{main2} &= \left( \ksum{N+1}{M} \int_{x_{k-1}}^{x_k} \left( \frac1{\Delta(t)} \int_0^t \left( \int_s^{x_k} h + \int_{x_k}^L h  \right)^r \delta(s) ds \right)^{\frac{q}{r}} w(t) dt \right)^{\frac1{q}} \nonumber
		\\
		&\approx \left( \ksum{N+1}{M} \int_{x_{k-1}}^{x_k} \left( \frac1{\Delta(t)} \int_0^t \left( \int_s^{x_k} h  \right)^r \delta(s) ds \right)^{\frac{q}{r}} w(t) dt \right)^{\frac1{q}} \nonumber
		\\
		&\quad+ \left( \ksum{N+1}{M-1} \left( \int_{x_k}^L h  \right)^q \int_{x_{k-1}}^{x_k} \left( \frac1{\Delta(t)} \int_0^t \delta(s) ds \right)^{\frac{q}{r}} w(t) dt \right)^{\frac1{q}}\nonumber\\
		&=: \I + \II.
	\end{align}
	We shall first deal with $\I$. Decomposing the integral $\int_{0}^{t}$ into the sum $\int_{0}^{x_{k-1}}+\int_{x_{k-1}}^{t}$ and using the fact that $x_{k-1}=0$ if $k=N+1$, which is possible if and only if $N>-\infty$, we obtain
	\begin{align}\label{lem:disc_main:discLHS1}
		\I&\approx \left( \ksum{N+2}{M}  \left(\int_0^{x_{k-1}} \left( \int_s^{x_k} h  \right)^r \delta(s) ds \right)^{\frac{q}{r}} \left( \int_{x_{k-1}}^{x_k} \Delta^{-\frac{q}{r}}w \right) \right)^{\frac1{q}}\nonumber
		\\
		&\quad+ \left( \ksum{N+1}{M} \int_{x_{k-1}}^{x_k} \left( \frac1{\Delta(t)} \int_{x_{k-1}}^t \left( \int_s^{x_k} h  \right)^r \delta(s) ds \right)^{\frac{q}{r}} w(t) dt \right)^{\frac1{q}}\nonumber\\
		&=: \I_1+\I_2.
	\end{align}
	Note that $\I_1$ can be written as
	\begin{align*}
		\I_1&=\left( \ksum{N+2}{M}  \left( \ksum[i]{N+1}{k-1} \int_{x_{i-1}}^{x_i} \left( \int_s^{x_k} h  \right)^r \delta(s) ds \right)^{\frac{q}{r}} \left(\int_{x_{k-1}}^{x_k} \Delta^{-\frac{q}{r}}w\right) \right)^{\frac1{q}} 
		\\
		&\approx \left( \ksum{N+2}{M}  \left( \ksum[i]{N+1}{k-1} \int_{x_{i-1}}^{x_i} \left( \int_s^{x_i} h  \right)^r \delta(s) ds \right)^{\frac{q}{r}} \left(\int_{x_{k-1}}^{x_k} \Delta^{-\frac{q}{r}} w \right)  \right)^{\frac1{q}}
		\\
		&\quad+ \left( \ksum{N+2}{M}  \left( \ksum[i]{N+1}{k-1} \left( \int_{x_i}^{x_k} h  \right)^r \int_{x_{i-1}}^{x_i}  \delta(s) ds \right)^{\frac{q}{r}}\left(\int_{x_{k-1}}^{x_k} \Delta^{-\frac{q}{r}}w\right) \right)^{\frac1{q}}
		\\
		&= \left( \ksum{N+2}{M}  \left( \ksum[i]{N+1}{k-1} \int_{x_{i-1}}^{x_i} \left( \int_s^{x_i} h  \right)^r \delta(s) ds \right)^{\frac{q}{r}}\left(\int_{x_{k-1}}^{x_k} \Delta^{-\frac{q}{r}}w\right)\right)^{\frac1{q}}
		\\
		&\quad+ \left( \ksum{N+2}{M}  \left( \ksum[i]{N+1}{k-1} \left( \ksum[j]{i}{k-1} \int_{x_j}^{x_{j+1}} h  \right)^r \left(\int_{x_{i-1}}^{x_i}  \delta\right) \right)^{\frac{q}{r}}\left(\int_{x_{k-1}}^{x_k} \Delta^{-\frac{q}{r}}w\right) \right)^{\frac1{q}}.
	\end{align*}
	Reindexing $(k-1)\mapsto k$, we obtain
	\begin{align}\label{lem:disc_main:discLHS2}
		\I_1 
		&\approx \left( \ksum{N+1}{M-1}  \left( \ksum[i]{N+1}{k} \int_{x_{i-1}}^{x_i} \left( \int_s^{x_i} h  \right)^r \delta(s) ds \right)^{\frac{q}{r}}\left(\int_{x_k}^{x_{k+1}} \Delta^{-\frac{q}{r}}w\right)\right)^{\frac1{q}}\nonumber
		\\
		&\quad+ \left( \ksum{N+1}{M-1}  \left( \ksum[i]{N+1}{k} \left( \ksum[j]{i}{k} \int_{x_j}^{x_{j+1}} h  \right)^r \left(\int_{x_{i-1}}^{x_i}  \delta\right) \right)^{\frac{q}{r}}\left(\int_{x_{k}}^{x_{k+1}} \Delta^{-\frac{q}{r}}w\right) \right)^{\frac1{q}}\nonumber\\
		& =: \I_{1,1} + \I_{1,2}.
	\end{align}
	
	Now we shall deal with $\II$. The very definition of $\Delta$ yields
	\begin{align}\label{lem:disc_main:discLHSII}
		\II = \left( \ksum{N+1}{M-1} \left( \ksum[i]{k}{M-1} \int_{x_i}^{x_{i+1}} h  \right)^q \left( \int_{x_{k-1}}^{x_k} w \right) \right)^{\frac1{q}}.
	\end{align}
	Then, \eqref{lem:disc_main:discLHS}, \eqref{lem:disc_main:discLHS1}, \eqref{lem:disc_main:discLHS2} and \eqref{lem:disc_main:discLHSII} altogether yields
	\begin{align}\label{lem:disc_main:discLHS3}
		\LHSeq{main2} & \approx \I_{1,1} + \I_{1,2} + \I_2 + \II \nonumber\\
		&= \left( \ksum{N+1}{M-1}  \left( \ksum[i]{N+1}{k} \int_{x_{i-1}}^{x_i} \left( \int_s^{x_i} h  \right)^r \delta(s) ds \right)^{\frac{q}{r}} \left( \int_{x_k}^{x_{k+1}} \Delta^{-\frac{q}{r}}w \right) \right)^{\frac1{q}} \nonumber
		\\
		&\quad+ \left( \ksum{N+1}{M-1}  \left( \ksum[i]{N+1}{k} \left( \ksum[j]{i}{k} \int_{x_j}^{x_{j+1}} h  \right)^r \left( \int_{x_{i-1}}^{x_i}  \delta \right) \right)^{\frac{q}{r}} \left( \int_{x_k}^{x_{k+1}} \Delta^{-\frac{q}{r}}w \right) \right)^{\frac1{q}} \nonumber
		\\
		&\quad+ \left( \ksum{N+1}{M} \int_{x_{k-1}}^{x_k} \left( \frac1{\Delta(t)} \int_{x_{k-1}}^t \left( \int_s^{x_k} h  \right)^r \delta(s) ds \right)^{\frac{q}{r}} w(t) dt \right)^{\frac1{q}}\nonumber
		\\
		&\quad+ \left( \ksum{N+1}{M-1} \left( \ksum[i]{k}{M-1} \int_{x_i}^{x_{i+1}} h  \right)^q \left( \int_{x_{k-1}}^{x_k} w \right) \right)^{\frac1{q}}.
	\end{align}
	From the validity of inequality \eqref{main2} for every $h\in\Mpl(0,L)$, together with \eqref{lem:disc_main:discRHS} and \eqref{lem:disc_main:discLHS3}, the following four inequalities can be obtained:
	\begin{align}
		& \left( \ksum{N+1}{M} \int_{x_{k-1}}^{x_k} \left( \frac1{\Delta(t)} \int_{x_{k-1}}^t \left( \int_s^{x_k} h  \right)^r \delta(s) ds \right)^{\frac{q}{r}} w(t) dt \right)^{\frac1{q}} \notag\\
		&\hspace{5cm}  \lesssim\left( \ksum{N+1}{M} \left( \int_{x_{k-1}}^{x_k} \varphi^{\frac1{p}} h \right)^p \right)^{\frac1{p}},\label{3.11}\\
		&\left( \ksum{N+1}{M-1}  \left( \ksum[i]{N+1}{k} \int_{x_{i-1}}^{x_i} \left( \int_s^{x_i} h  \right)^r \delta(s) ds \right)^{\frac{q}{r}} \left( \int_{x_k}^{x_{k+1}} \Delta^{-\frac{q}{r}}w \right) \right)^{\frac1{q}} \notag\\
		& \hspace{5cm} \lesssim\left( \ksum{N+1}{M} \left( \int_{x_{k-1}}^{x_k} \varphi^{\frac1{p}} h \right)^p \right)^{\frac1{p}},
		\label{3.12} \\
		&\left( \ksum{N+1}{M-1}  \left( \ksum[i]{N+1}{k} \left( \ksum[j]{i}{k} \int_{x_j}^{x_{j+1}} h  \right)^r \left( \int_{x_{i-1}}^{x_i}  \delta \right) \right)^{\frac{q}{r}} \left( \int_{x_k}^{x_{k+1}} \Delta^{-\frac{q}{r}}w \right) \right)^{\frac1{q}}\notag\\
		& \hspace{5cm} \lesssim \left( \ksum{N+1}{M} \left( \int_{x_{k-1}}^{x_k} \varphi^{\frac1{p}} h \right)^p \right)^{\frac1{p}}, \label{3.13}  \\ &\left( \ksum{N+1}{M-1} \left( \ksum[i]{k}{M-1} \int_{x_i}^{x_{i+1}} h  \right)^q \left( \int_{x_{k-1}}^{x_k} w \right) \right)^{\frac1{q}} \notag\\
		& \hspace{5cm}  \lesssim\left( \ksum{N+1}{M} \left( \int_{x_{k-1}}^{x_k} \varphi^{\frac1{p}} h \right)^p \right)^{\frac1{p}}.\label{3.14}
	\end{align}
	It remains to show that inequalities \eqref{3.11}, \eqref{3.12}, \eqref{3.13} and \eqref{3.14} are equivalent to inequalities \eqref{M1}, \eqref{M2}, \eqref{M3} and \eqref{M4}, respectively.  
	
	Firstly, note that \eqref{M1} and \eqref{3.11} are identical. 
	
	Next, assume that \eqref{M2} holds for all $h \in \Mpl(0,x_{M-1})$. Then, for every $h \in \Mpl(0,L)$, the following is true:
	\begin{align*}
		\LHS\eqref{3.12} = \LHSeq{M2} \leq \widetilde{C_2} \RHSeq{M2} \lesssim \RHSeq{3.12}.
	\end{align*}
	Conversely, assume that \eqref{3.12} holds for every $h \in \Mpl(0,L)$. Then for $ g \in \Mpl(0, x_{M-1})$,  choosing 
	\begin{equation*}
		h(x) = \begin{cases}
			g(x) &\text{if $x\in(0, x_{M-1})$},\\
			0 &\text{if $x\in [x_{M-1}, L)$},
		\end{cases}
	\end{equation*} 
	we have by the validity of inequality \eqref{3.12} that
	\begin{align*}
		\LHSeq{M2}  = \LHSeq{3.12} &\lesssim  \left( \ksum{N+1}{M} \left( \int_{x_{k-1}}^{x_k} \varphi^{\frac1{p}} h \right)^p \right)^{\frac1{p}} \\
		&= \left( \ksum{N+1}{M-1} \left( \int_{x_{k-1}}^{x_k} \varphi^{\frac1{p}}  g \right)^p \right)^{\frac1{p}} = \RHSeq{M2}.
	\end{align*} 
	Therefore, inequality
	\eqref{3.12} is equivalent to  \eqref{M2}. 
	
	Now, assume that \eqref{M3} holds for all $h\in \Mpl(x_{N+1}, L)$. Then, for every $h \in \Mpl(0,L)$, validity of \eqref{M3} yields the following chain of relations:
	\begin{align*}
		\LHSeq{3.13} &= \LHSeq{M3} \\
		&=\left( \ksum{N+1}{M-1}  \left( \ksum[i]{N+1}{k} \left( \ksum[j]{i}{k} \int_{x_j}^{x_{j+1}} h  \right)^r \left( \int_{x_{i-1}}^{x_i}  \delta \right) \right)^{\frac{q}{r}} \left( \int_{x_k}^{x_{k+1}} \Delta^{-\frac{q}{r}}w \right) \right)^{\frac1{q}}\\
		& \leq \widetilde{C_3} \left( \ksum{N+1}{M-1} \left(\int_{x_{k}}^{x_{k+1}} \varphi^{\frac1{p}} h\right)^p \right)^{\frac1{p}}\\
		& \approx \left( \ksum{N+2}{M} \left( \int_{x_{k-1}}^{x_{k}} \varphi^{\frac1{p}} h \right)^p \right)^{\frac1{p}} \leq \RHSeq{3.13}.
	\end{align*}
	Conversely, assume that \eqref{3.13} holds for every $h \in \Mpl(0,L)$. Then for $f \in \Mpl(x_{N+1},L)$,  choosing 
	\begin{equation*}
		h(x) = \begin{cases}
			0 &\text{if $x\in(0, x_{N+1})$},\\
			f(x) &\text{if $x\in [x_{N+1}, L)$},
		\end{cases}
	\end{equation*} 
	we have by the validity of inequality \eqref{3.13} that
	\begin{align*}
		\LHSeq{M3} & = \LHSeq{3.13} \\
		&= \left( \ksum{N+1}{M-1}  \left( \ksum[i]{N+1}{k} \left( \ksum[j]{i}{k} \int_{x_j}^{x_{j+1}} h  \right)^r \left( \int_{x_{i-1}}^{x_i}  \delta \right) \right)^{\frac{q}{r}} \left( \int_{x_k}^{x_{k+1}} \Delta^{-\frac{q}{r}}w \right) \right)^{\frac1{q}}\\
		& \lesssim\left( \ksum{N+1}{M} \left( \int_{x_{k-1}}^{x_{k}} \varphi^{\frac1{p}} h \right)^p \right)^{\frac1{p}}\\
		& = \left( \ksum{N+2}{M} \left( \int_{x_{k-1}}^{x_{k}} \varphi^{\frac1{p}} f \right)^p \right)^{\frac1{p}}\\
		& = \left( \ksum{N+1}{M-1} \left( \int_{x_{k-1}}^{x_{k}} \varphi^{\frac1{p}} f \right)^p \right)^{\frac1{p}}\\
		& = \RHSeq{M3}.
	\end{align*}
	
	The equivalencies of \eqref{3.14} and \eqref{M4} can be proved in the same way.
\end{proof}

\begin{remark}\label{rem:value_of_a_is_immaterial}
	The assumption $a>108$ is merely technical, as already noted below \eqref{prel:dis-antidis}.
\end{remark}

We are now in a position to prove a discrete characterization of \eqref{main2}.

\begin{theorem}\label{thm:discr:resonanceform}
	Let $0< p,q,r < \infty$ and $u, \delta, v, w$ be weights on $(0,L)$. Let $\varphi$ be given by \eqref{varphi}. Assume that there is $t_0\in(0,L)$ such that $0 < \varphi(t_0) < \infty$. Let $\{x_k\}_{k=N}^M \in CS(\varphi, U^p, a)$ with $a>108$. Set
	\begin{align}
		A(x_{k-1}, x_k) &= \sup_{h \in \Mpl(0,L)} \dfrac{\left( \int_{x_{k-1}}^{x_k} \left( \frac1{\Delta(t)} \int_{x_{k-1}}^t \left( \int_s^{x_k} h \right)^r \delta(s) ds \right)^{\frac{q}{r}} w(t) dt \right)^\frac1{q}}{\int_{x_{k-1}}^{x_k}  h \varphi^\frac1{p}} \label{thm:discr:Ak}\\
		\intertext{and}
		B(x_{k-1}, x_k) &= \sup_{h \in \Mpl(0,L)} \dfrac{\left( \int_{x_{k-1}}^{x_k} \left( \int_s^{x_k} h \right)^r \delta(s) ds \right)^\frac1{r}}{\int_{x_{k-1}}^{x_k}  h \varphi^\frac1{p}} \label{thm:discr:Bk}
	\end{align}
	for $k\in\Z$, $N+1\leq k\leq M$. Denote by $C_i$, $i=1,2,3,4$, the optimal constants in the inequalities:
	\begin{align}
		&\left(\ksum{N+1}{M} a_k^q A(x_{k-1},x_k)^q\right)^{\frac1{q}}
		\leq C_1 \left(\ksum{N+1}{M} a_k^p\right)^\frac1{p}; \label{D1}
		\\
		&\left(\ksum{N+1}{M-1} \left( \ksum[i]{N+1}{k} a_i^r B(x_{i-1}, x_i)^r \right)^\frac{q}{r} \left( \int_{x_k}^{x_{k+1}} \Delta^{-\frac{q}{r}}w \right) \right)^{\frac1{q}}
		\leq C_2 \left(\ksum{N+1}{M-1} a_k^p\right)^\frac1{p}; \label{D2}
		\\
		&\left(\ksum{N+1}{M-1} \left( \ksum[i]{N+1}{k} a_i^r \varphi(x_i)^{-\frac{r}{p}} \left( \int_{x_{i-1}}^{x_i} \delta \right) \right)^\frac{q}{r} \left( \int_{x_k}^{x_{k+1}} \Delta^{-\frac{q}{r}}w \right) \right)^{\frac1{q}}
		\leq C_3 \left(\ksum{N+1}{M-1} a_k^p\right)^\frac1{p}; \label{D3}
		\\
		\intertext{and}
		&\left(\ksum{N+1}{M-1} a_k^q \varphi(x_k)^{-\frac{q}{p}} \left( \int_{x_{k-1}}^{x_k} w \right) \right)^{\frac1{q}}
		\leq C_4 \left(\ksum{N+1}{M-1} a_k^p\right)^\frac1{p} \label{D4}
	\end{align}
	for every sequence $\{a_k\}_{k=N+1}^M$ of nonnegative numbers. Then the $C$ defined by~\eqref{E:C} satisfies
	\begin{equation*}
		C \approx C_1 + C_2 + C_3 + C_4,
	\end{equation*}
	in which the equivalence constants depend only on $p,q,r$ and $a$.
\end{theorem}

\begin{proof}
	In view of Lemma~\ref{lem:disc_main}, it is sufficient to show that $\widetilde{C_i}\approx C_i$, $i=1,2,3,4$, with the equivalence constants depending only on the parameters $p,q,r$ and $a$.
	First, we shall show that $\widetilde{C_1}\approx C_1$. Assume that $\widetilde{C_1}<\infty$. Consequently, $A(x_{k-1}, x_k) < \widetilde{C_1} < \infty$ for every $k\in\Z$, $N+1\leq k\leq M$. Hence there are functions $h_k\in\Mpl(0,L)$, $k\in\Z$, $N+1\leq k\leq M$, supported in $[x_{k-1}, x_k]$ such that
	\begin{equation}\label{thm:discr:unitint}
		\int_{x_{k-1}}^{x_k} h_k \varphi^{\frac1{p}} = 1
	\end{equation}
	and
	\begin{equation*}
		\left( \int_{x_{k-1}}^{x_k} \left( \frac1{\Delta(t)} \int_{x_{k-1}}^t \left( \int_s^{x_k} h_k \right)^r \delta(s) ds \right)^{\frac{q}{r}} w(t) dt \right)^\frac1{q} \geq \frac1{2} A(x_{k-1}, x_k).
	\end{equation*}
	Plugging $h = \ksum[i]{N+1}{M} a_ih_i$, where $\{a_i\}_{i=N+1}^M$ is a sequence of nonnegative numbers, in \eqref{M1}, we obtain
	\begin{align*}
		\LHSeq{M1} &= \left( \ksum{N+1}{M} \int_{x_{k-1}}^{x_k} \left( \frac1{\Delta(t)} \int_{x_{k-1}}^t \left( \int_s^{x_k} \ksum[i]{N+1}{M} a_ih_i\right)^r \delta(s) ds \right)^{\frac{q}{r}} w(t) dt \right)^{\frac1{q}}\\
		&= \left( \ksum{N+1}{M} a_k^q \int_{x_{k-1}}^{x_k} \left( \frac1{\Delta(t)} \int_{x_{k-1}}^t \left( \int_s^{x_k} h_k\right)^r \delta(s) ds \right)^{\frac{q}{r}} w(t) dt \right)^{\frac1{q}} \\
		&\gtrsim \left( \ksum{N+1}{M} a_k^q A(x_{k-1}, x_k)^q \right)^\frac1{q},
	\end{align*}
	and in view of \eqref{thm:discr:unitint}
	\begin{equation*}
		\RHSeq{M1} = \widetilde{C_1} \left( \ksum{N+1}{M} \left( \int_{x_{k-1}}^{x_k} \ksum[i]{N+1}{M} a_i h_i \varphi^{\frac1{p}}\right)^p \right)^{\frac1{p}} = \widetilde{C_1} \left( \ksum{N+1}{M} a_k^p  \right)^{\frac1{p}}.
	\end{equation*}
	Therefore
	\begin{equation*}
		\left( \ksum{N+1}{M} a_k^q A(x_{k-1}, x_k)^q \right)^\frac1{q} \lesssim \widetilde{C_1} \left( \ksum{N+1}{M} a_k^p  \right)^{\frac1{p}};
	\end{equation*}
	hence $C_1\lesssim \widetilde{C_1}$. On the other hand, assume that $C_1<\infty$. Let $h\in\Mpl(0,L)$. Using \eqref{thm:discr:Ak} and \eqref{D1}
	we obtain
	\begin{align*}
		\LHSeq{M1} & = \left( \ksum{N+1}{M}  \left( \int_{x_{k-1}}^{x_k} \left( \frac1{\Delta(t)} \int_{x_{k-1}}^t \left( \int_s^{x_k} h\right)^r \delta(s) ds \right)^{\frac{q}{r}} w(t) dt \right) \times \right. \\
		& \hspace{1cm} \times \left.
		\left( \int_{x_{k-1}}^{x_k} h \varphi^\frac1{p} \right)^{q}\left( \int_{x_{k-1}}^{x_k} h \varphi^\frac1{p} \right)^{-q} \right)^{\frac1{q}} \\
		&\leq \left( \ksum{N+1}{M}  \left( \int_{x_{k-1}}^{x_k} h \varphi^\frac1{p} \right)^{q} A(x_{k-1}, x_k)^q \right)^{\frac1{q}}  \\
		&\leq C_1 \left(\ksum{N+1}{M} \left( \int_{x_{k-1}}^{x_k} h \varphi^\frac1{p} \right)^p\right)^\frac1{p}.
	\end{align*}
	The last inequality follows by applying \eqref{D1} with $\left\{\int_{x_{k-1}}^{x_k} h \varphi^\frac1{p}\right\}_{k={N+1}}^M$. Hence $\widetilde{C_1}\leq C_1$.
	
	Second, we shall show that $\widetilde{C_2}\approx C_2$. Assume that $\widetilde{C_2}<\infty$. Consequently, for every $k\in\Z$, $N+1\leq k\leq M-1$,
	$$
	B(x_{k-1}, x_k) < \widetilde{C_2} \left( \int_{x_k}^{x_{k+1}} \Delta^{-\frac{q}{r}} w \right)^{-\frac{1}{q}} < \infty.
	$$
	Hence there are functions $h_k\in\Mpl(0,L)$, $k\in\Z$, $N+1\leq k\leq M-1$, supported in $[x_{k-1}, x_k]$ and satisfying \eqref{thm:discr:unitint} such that
	\begin{equation*}
		\left( \int_{x_{k-1}}^{x_k} \left( \int_s^{x_k} h \right)^r \delta(s) ds \right)^\frac1{r} \geq \frac1{2} B(x_{k-1}, x_k).
	\end{equation*}
	Testing \eqref{M2} with $h = \ksum[i]{N+1}{M-1} a_ih_i$, where $\{a_i\}_{i=N+1}^{M}$ is a sequence of nonnegative numbers, we get
	\begin{align*}
		\LHSeq{M2} &= \left( \ksum{N+1}{M-1} \left( \ksum[i]{N+1}{k} \int_{x_{i-1}}^{x_i} \left(  \int_s^{x_i}\ksum[j]{N+1}{M-1} a_j h_j\right)^r \delta(s) ds \right)^{\frac{q}{r}} \left( \int_{x_k}^{x_{k+1}} \Delta^{-\frac{q}{r}} w \right) \right)^{\frac1{q}} \\
		&= \left( \ksum{N+1}{M-1} \left( \ksum[i]{N+1}{k} a_i^r \int_{x_{i-1}}^{x_i} \left( \int_s^{x_i} h_i \right)^r \delta(s) ds \right)^{\frac{q}{r}} \left( \int_{x_k}^{x_{k+1}} \Delta^{-\frac{q}{r}} w \right) \right)^{\frac1{q}} \\
		&\gtrsim \left( \ksum{N+1}{M-1} \left( \ksum[i]{N+1}{k} a_i^r B(x_{i-1}, x_i)^r \right)^{\frac{q}{r}} \left( \int_{x_k}^{x_{k+1}} \Delta^{-\frac{q}{r}} w \right) \right)^{\frac1{q}}.
	\end{align*}
	Plainly,
	\begin{equation*}
		\RHSeq{M2} = \widetilde{C_2} \left( \ksum{N+1}{M-1} \left(\int_{x_{k-1}}^{x_k} \ksum[i]{N+1}{M-1} a_i  h_i \varphi^{\frac1{p}}\right)^p \right)^{\frac1{p}} = \widetilde{C_2} \left( \ksum{N+1}{M-1} a_k^p \right)^\frac1{p}.
	\end{equation*}
	Therefore
	\begin{equation*}
		\left( \ksum{N+1}{M-1} \left( \ksum[i]{N+1}{k} a_i^r B(x_{i-1}, x_i)^r \right)^{\frac{q}{r}} \left( \int_{x_k}^{x_{k+1}} \Delta^{-\frac{q}{r}} w \right) \right)^{\frac1{q}} \lesssim \widetilde{C_2} \left( \ksum{N+1}{M-1} a_k^p \right)^\frac1{p},
	\end{equation*}
	which implies $C_2 \lesssim \widetilde{C_2}$. Assume now that $C_2 < \infty$. Thanks to \eqref{thm:discr:Bk} and  \eqref{D2}, we have
	\begin{align*}
		&\LHSeq{M2}\\
		& \hspace{0.5cm} =\left(\ksum{N+1}{M-1} \left( \ksum[i]{N+1}{k} \left( \int_{x_{i-1}}^{x_i} h \varphi^\frac1{p} \right)^{r} \left( \int_{x_{i-1}}^{x_i} \left( \int_s^{x_i} h \right)^r \delta(s) ds \right) \left( \int_{x_{i-1}}^{x_i} h \varphi^\frac1{p} \right)^{-r} \right)^\frac{q}{r} \right. \\
		& \hspace{2cm} \times \left. \left( \int_{x_k}^{x_{k+1}} \Delta^{-\frac{q}{r}} w  \right) \right)^{\frac1{q}} \\
		&\hspace{0.5cm}  \leq \left(\ksum{N+1}{M-1} \left( \ksum[i]{N+1}{k} \left( \int_{x_{i-1}}^{x_i} h \varphi^\frac1{p} \right)^{r} B(x_{i-1}, x_i)^r \right)^\frac{q}{r} \left( \int_{x_k}^{x_{k+1}}\Delta^{-\frac{q}{r}} w  \right) \right)^{\frac1{q}}  \\
		&\hspace{0.5cm}  \leq C_2 \left( \ksum{N+1}{M-1} \left(\int_{x_{k-1}}^{x_k}  h \varphi^{\frac1{p}} \right)^p \right)^{\frac1{p}}
	\end{align*}
	for every $h\in\Mpl(0,L)$ where the last inequality is followed by applying \eqref{D2} with $\left\{\int_{x_{k-1}}^{x_k} h \varphi^\frac1{p}\right\}_{k={N+1}}^{M-1}$. Thus, $\widetilde{C_2}\leq C_2$.
	
	Next, we turn our attention to the equivalence $\widetilde{C_3}\approx C_3$. Assume that $\widetilde{C_3} < \infty$. Note that
	\begin{equation}\label{thm:discr:holder_saturation}
		\sup_{h\in\Mpl(0,L)} \frac{\int_{x_k}^{x_{k+1}} h}{\int_{x_k}^{x_{k+1}} h \varphi^\frac1{p}} = \esup_{t\in(x_k, x_{k+1})} \varphi(t)^{-\frac1{p}} = \varphi(x_k)^{-\frac1{p}}
	\end{equation}
	for every $k\in\Z$, $N+1 \leq k \leq M-1$, owing to the saturation of H\"older's  inequality and the monotonicity of $\varphi$. Consequently, there are functions $h_k\in\Mpl(0,L)$, $k\in\Z$, $N+1 \leq k \leq M-1$, supported in $[x_k, x_{k+1}]$ such that
	\begin{align}
		\int_{x_k}^{x_{k+1}} h_k \varphi^{\frac1{p}} &= 1 \label{thm:discr:unitint2}\\
		\intertext{and}
		\int_{x_k}^{x_{k+1}} h_k \geq \frac1{2} \varphi(x_k)^{-\frac1{p}} \label{thm:discr:almost_extremal}.
	\end{align}
	By plugging $h=\ksum[n]{N+1}{M-1} a_n h_n$, where $\{a_n\}_{n=N+1}^{M-1}$ is a sequence of nonnegative numbers, in \eqref{M3}, we obtain
	\begin{align*}
		\LHSeq{M3} &= \left( \ksum{N+1}{M-1}  \left( \ksum[i]{N+1}{k}  \left( \ksum[j]{i}{k}  \int_{x_{j}}^{x_{j+1}} \ksum[n]{N+1}{M-1} a_n h_n\right)^r \left( \int_{x_{i-1}}^{x_i}\delta \right) \right)^{\frac{q}{r}} \right.\\
		&\hspace{2cm} \left. \times \left( \int_{x_k}^{x_{k+1}} \Delta^{-\frac{q}{r}} w \right) \right)^{\frac1{q}} \\
		&\gtrsim \left( \ksum{N+1}{M-1}  \left( \ksum[i]{N+1}{k}  \left( \ksum[j]{i}{k} a_j \varphi(x_j)^{-\frac1{p}} \right)^r \left( \int_{x_{i-1}}^{x_i}\delta \right) \right)^{\frac{q}{r}} \left( \int_{x_k}^{x_{k+1}} \Delta^{-\frac{q}{r}} w \right) \right)^{\frac1{q}} \\
		&\geq \left( \ksum{N+1}{M-1}  \left( \ksum[i]{N+1}{k}  a_i^r \varphi(x_i)^{-\frac{r}{p}} \left( \int_{x_{i-1}}^{x_i}\delta \right) \right)^{\frac{q}{r}} \left( \int_{x_k}^{x_{k+1}} \Delta^{-\frac{q}{r}} w \right) \right)^{\frac1{q}}
	\end{align*}
	and
	\begin{align*}
		\RHSeq{M3} &= \widetilde{C_3} \left( \ksum{N+1}{M-1} \left( \int_{x_{k}}^{x_{k+1}} \ksum[n]{N+1}{M-1} a_n   h_n \varphi^{\frac1{p}}\right)^p \right)^{\frac1{p}} \\
		&= \widetilde{C_3} \left( \ksum{N+1}{M-1}  a_k^p \right)^{\frac1{p}}.
	\end{align*}
	Hence
	\begin{equation*}
		\left( \ksum{N+1}{M-1}  \left( \ksum[i]{N+1}{k}  a_i^r \varphi(x_i)^{-\frac{r}{p}} \left( \int_{x_{i-1}}^{x_i}\delta \right) \right)^{\frac{q}{r}} \left( \int_{x_k}^{x_{k+1}}  \Delta^{-\frac{q}{r}} w \right) \right)^{\frac1{q}} \lesssim \widetilde{C_3} \left( \ksum{N+1}{M-1}  a_k^p \right)^{\frac1{p}},
	\end{equation*}
	and so $C_3 \lesssim \widetilde{C_3}$. Assume now that $C_3 < \infty$. Let $h\in\Mpl(0,L)$, and test \eqref{D3} with $\{a_k\}_{k=N+1}^{M-1}$ defined as
	\begin{equation*}
		a_k = \varphi(x_k)^\frac1{p} \ksum[j]{k}{M-1} b_j \varphi(x_j)^{-\frac1{p}},
	\end{equation*}
	where
	\begin{equation}\label{thm:discr:bj}
		b_j = \int_{x_j}^{x_{j+1}} h \varphi^\frac1{p}, \,\, N+1 \leq j \leq M-1.
	\end{equation}
	We have
	\begin{align*}
		\LHSeq{D3} &= \left(\ksum{N+1}{M-1} \left( \ksum[i]{N+1}{k} \left( \varphi(x_i)^\frac1{p} \ksum[j]{i}{M-1} b_j \varphi(x_j)^{-\frac1{p}} \right)^r \varphi(x_i)^{-\frac{r}{p}} \left( \int_{x_{i-1}}^{x_i} \delta \right) \right)^\frac{q}{r} \right. \\
		&\hspace{4cm} \left. \times \left( \int_{x_k}^{x_{k+1}} \Delta^{-\frac{q}{r}} w \right) \right)^{\frac1{q}} \\
		&\geq \left( \ksum{N+1}{M-1} \left( \ksum[i]{N+1}{k} \left( \ksum[j]{i}{k} b_j \varphi(x_j)^{-\frac1{p}} \right)^r \left( \int_{x_{i-1}}^{x_i} \delta \right) \right)^\frac{q}{r} \left( \int_{x_k}^{x_{k+1}} \Delta^{-\frac{q}{r}} w \right) \right)^{\frac1{q}}.
	\end{align*}
	Next, using \eqref{thm:discr:holder_saturation}, we obtain
	\begin{align*}
		\LHSeq{D3}
		&\geq \left( \ksum{N+1}{M-1} \left( \ksum[i]{N+1}{k} \left( \ksum[j]{i}{k} b_j \left( \int_{x_j}^{x_{j+1}} h \right) \left( \int_{x_j}^{x_{j+1}} h \varphi^\frac1{p} \right)^{-1} \right)^r \left( \int_{x_{i-1}}^{x_i} \delta \right) \right)^\frac{q}{r} \right. \\
		&\hspace{4cm} \left. \times \left( \int_{x_k}^{x_{k+1}} \Delta^{-\frac{q}{r}} w \right) \right)^{\frac1{q}} \\
		&= \left( \ksum{N+1}{M-1} \left( \ksum[i]{N+1}{k} \left( \ksum[j]{i}{k} \int_{x_j}^{x_{j+1}} h \right)^r \left( \int_{x_{i-1}}^{x_i} \delta \right) \right)^\frac{q}{r} \left( \int_{x_k}^{x_{k+1}} \Delta^{-\frac{q}{r}} w \right) \right)^{\frac1{q}}
	\end{align*}
	and
	\begin{equation}\label{thm:discr:varphixk_strongly_incr}
		\RHSeq{D3} = C_3 \left(\ksum{N+1}{M-1} \varphi(x_k) \left( \ksum[j]{k}{M-1} b_j \varphi(x_j)^{-\frac1{p}} \right)^p\right)^\frac1{p} \approx C_3 \left(\ksum{N+1}{M-1} b_k^p \right)^\frac1{p},
	\end{equation}
	in which we used \eqref{EQ:strongly_increasing_sum_sum} with $\{\varrho_k\}_{k=N+1}^{M-1} = \{\varphi(x_k)\}_{k=N+1}^{M-1}$; moreover, the equivalence constants depend only on $p$ and $a$. It follows from the validity of \eqref{D3} and the definition of $b_k$ in \eqref{thm:discr:bj} that
	\begin{align*}
		&\left( \ksum{N+1}{M-1} \left( \ksum[i]{N+1}{k} \left( \ksum[j]{i}{k} \int_{x_j}^{x_{j+1}} h \right)^r \left( \int_{x_{i-1}}^{x_i} \delta \right) \right)^\frac{q}{r} \left( \int_{x_k}^{x_{k+1}} \Delta^{-\frac{q}{r}} w \right) \right)^{\frac1{q}} \\
		&\qquad \lesssim C_3 \left(\ksum{N+1}{M-1} b_k^p \right)^\frac1{p} = C_3 \left(\ksum{N+1}{M-1} \left( \int_{x_k}^{x_{k+1}} h \varphi^\frac1{p} \right)^p \right)^\frac1{p};
	\end{align*}
	hence $\widetilde{C_3} \lesssim C_3$.
	
	Last, we shall show that $\widetilde{C_4} \approx C_4$. Assume that $\widetilde{C_4} < \infty$. Thanks to \eqref{thm:discr:holder_saturation} again, there are functions $h_k\in\Mpl(0,L)$, $k\in\Z$, $N+1 \leq k \leq M-1$, supported in $[x_k, x_{k+1}]$ and satisfying \eqref{thm:discr:unitint2} and \eqref{thm:discr:almost_extremal}. Let $\{a_k\}_{k=N+1}^{M-1}$ be a sequence of nonnegative numbers. Inserting $h=\ksum[j]{N+1}{M-1} a_j h_j$ in \eqref{M4}, we obtain
	\begin{align*}
		\LHSeq{M4} &= \left( \ksum{N+1}{M-1} \left( \ksum[i]{k}{M-1} \int_{x_i}^{x_{i+1}} \ksum[j]{N+1}{M-1}  a_j h_j \right)^q \int_{x_{k-1}}^{x_k} w \right)^{\frac1{q}} \\
		&\gtrsim \left( \ksum{N+1}{M-1} \left( \ksum[i]{k}{M-1} a_i \varphi(x_i)^{-\frac1{p}} \right)^q \int_{x_{k-1}}^{x_k} w \right)^{\frac1{q}} \\
		&\geq \left( \ksum{N+1}{M-1} a_k^q \varphi(x_k)^{-\frac{q}{p}} \int_{x_{k-1}}^{x_k} w \right)^{\frac1{q}},
	\end{align*}
	and
	\begin{equation*}
		\RHSeq{M4} = \widetilde{C_4} \left( \ksum{N+1}{M-1} \left(\int_{x_{k}}^{x_{k+1}} \ksum[i]{N+1}{M-1} a_i   h_i \varphi^{\frac1{p}}\right)^p  \right)^{\frac1{p}} = \widetilde{C_4} \left( \ksum{N+1}{M-1} a_k^p \right)^{\frac1{p}}.
	\end{equation*}
	Hence
	\begin{equation*}
		\left( \ksum{N+1}{M-1} a_k^q \varphi(x_k)^{-\frac{q}{p}} \int_{x_{k-1}}^{x_k} w \right)^{\frac1{q}} \lesssim \widetilde{C_4} \left( \ksum{N+1}{M-1} a_k^p \right)^{\frac1{p}},
	\end{equation*}
	and so $C_4 \lesssim \widetilde{C_4}$. Now, the proof will be finished once we show that $\widetilde{C_4} \lesssim C_4$. Assume that $C_4 < \infty$. Let $h\in\Mpl(0,L)$, and consider the sequence $\{\varphi(x_k)^\frac1{p} \ksum[j]{k}{M-1} b_j \varphi(x_j)^{-\frac1{p}}\}_{j=N+1}^{M-1}$, where $\{b_j\}_{j=N+1}^{M-1}$ is defined by \eqref{thm:discr:bj}. Plugging it in \eqref{D4} and using \eqref{thm:discr:holder_saturation} we get
	\begin{align*}
		\LHSeq{D4} &= \left(\ksum{N+1}{M-1} \left( \ksum[j]{k}{M-1} b_j \varphi(x_j)^{-\frac1{p}} \right)^q \int_{x_{k-1}}^{x_k} w \right)^{\frac1{q}} \\
		&\geq \left(\ksum{N+1}{M-1} \left( \ksum[j]{k}{M-1} b_j \left( \int_{x_j}^{x_{j+1}} h \right) \left( \int_{x_j}^{x_{j+1}} h\varphi^\frac1{p} \right)^{-1} \right)^q \int_{x_{k-1}}^{x_k} w \right)^{\frac1{q}} \\
		&= \left(\ksum{N+1}{M-1} \left( \ksum[j]{k}{M-1} \int_{x_j}^{x_{j+1}} h  \right)^q \int_{x_{k-1}}^{x_k} w \right)^{\frac1{q}},
	\end{align*}
	and
	\begin{equation*}
		\RHSeq{D4} \approx C_4 \left( \ksum{N+1}{M-1} b_k^p \right)^\frac1{p} = C_4 \left( \ksum{N+1}{M-1} \left( \int_{x_k}^{x_{k+1}} h \varphi^\frac1{p} \right)^p \right)^\frac1{p},
	\end{equation*}
	in which we used the same argument as in \eqref{thm:discr:varphixk_strongly_incr}. It follows that
	\begin{align*}
		\left(\ksum{N+1}{M-1} \left( \ksum[j]{k}{M-1}  \int_{x_j}^{x_{j+1}} h  \right)^q \int_{x_{k-1}}^{x_k} w \right)^{\frac1{q}} &\lesssim C_4 \left( \ksum{N+1}{M-1} \left( \int_{x_k}^{x_{k+1}} h \varphi^\frac1{p}
		\right)^p \right)^\frac1{p}
	\end{align*}
	which finishes the proof.
\end{proof}

\begin{remark}\label{rem:Cts_Hardy_Char}
	For future reference, note that, thanks to the following equivalent expression for optimal constants in (continuous) Hardy inequalities with weights (see \cite{B:78} for $r\geq1$ and \cite{SS:96} for $r<1$), we have
	\begin{align*}
		B(x_{k-1}, x_k) &= \sup_{h \in \Mpl(0,L)} \dfrac{\left( \int_{x_{k-1}}^{x_k} \left( \int_s^{x_k} h \right)^r \delta(s) ds \right)^\frac1{r}}{\int_{x_{k-1}}^{x_k}  h \varphi^\frac1{p}}\\
		&\approx \begin{cases}
			\esup_{t\in(x_{k-1}, x_k)} \left( \int_{x_{k-1}}^t \delta \right)^\frac1{r} \varphi(t)^{-\frac1{p}} &\qquad \text{if $r \geq 1$},\\
			\left( \int_{x_{k-1}}^{x_k} \left( \int_{x_{k-1}}^t \delta \right)^\frac{r}{1-r} \delta(t) \varphi(t)^{-\frac{r}{p(1-r)}} dt \right)^\frac{1-r}{r} &\qquad\text{if $r < 1$},
		\end{cases}
	\end{align*}
	for every $k\in\Z$, $N+1\leq k\leq M$, in which the equivalence constants depend only on $r$.
\end{remark}

\begin{theorem}\label{thm:main_discretization}
	Let $0<p\leq q<\infty$, $0<r<\infty$ and $u, \delta, v, w$ be weights on $(0,L)$. Let $\varphi$ be the function defined by \eqref{varphi}. Assume that there is $t_0\in(0,L)$ such that $0 < \varphi(t_0) < \infty$. Let $\{x_k\}_{k=N}^M \in CS(\varphi, U^p, a)$ with $a>108$. Let $C$ be given by~\eqref{E:C}.
	
	\rm{(i)} If $p \leq q$, $p\leq r$, $1\leq q$, $1\leq r$, then $C\approx C_{1,1}+C_{1,2}+ C_{3,1} +C_{4,1}$, where
	\begin{align*}
		&C_{1,1}:=  \sup_{N+1\leq k \leq M} \esup_{t\in(x_{k-1}, x_k)} \left( \int_t^{x_k} \Delta^{-\frac{q}{r}} w \right)^\frac1{q} \esup_{s\in(x_{k-1}, t)} \left( \int_{x_{k-1}}^s \delta \right)^\frac1{r} \varphi(s)^{-\frac1{p}},
		\\
		&C_{1,2} := \sup_{N+1\leq k \leq M} \esup_{t\in(x_{k-1}, x_k)} \bigg(\int_{x_{k-1}}^t \Delta(s)^{-\frac{q}{r}}w(s) \bigg(\int_{x_{k-1}}^s \delta\bigg)^{\frac{q}{r}} ds\bigg)^{\frac{1}{q}} \varphi(t)^{-\frac{1}{p}},
		\\
		&C_{3,1} :=  \sup_{N+1\leq k \leq M-1} \bigg(\int_{x_k}^L \Delta^{-\frac{q}{r}} w \bigg)^{\frac{1}{q}} \esup_{t\in(x_{k-1}, x_k)} \bigg( \int_{x_{k-1}}^t \delta\bigg)^{\frac{1}{r}}
		\varphi(t)^{-\frac{1}{p}},
		\intertext{and}
		&C_{4,1} :=  \sup_{N+1\leq k \leq M-1} \bigg(\int_{x_{k-1}}^{x_k} w \bigg)^{\frac{1}{q}} \varphi(x_k)^{-\frac{1}{p}}.
	\end{align*}
	
	\rm{(ii)} If $p\leq r < 1 \leq q$, then $C \approx C_{1,2} + C_{1,3} + C_{3,2} + C_{4,1}$, where
	\begin{align*}
		&C_{1,3}:= \sup_{N+1\leq k \leq M} \esup_{t\in(x_{k-1}, x_k)} \left( \int_t^{x_k} \Delta^{-\frac{q}{r}} w \right)^\frac1{q} \\
		&\hspace{4cm}\times \left( \int_{x_{k-1}}^t \left( \int_{x_{k-1}}^s \delta \right)^\frac{r}{1-r} \delta(s)
		\varphi(s)^{-\frac{r}{p(1-r)}} ds \right)^{\frac{1-r}{r}},
		\intertext{and}
		&C_{3,2}:= \sup_{N+1\leq k \leq M-1} \bigg(\int_{x_k}^L \Delta^{-\frac{q}{r}} w \bigg)^{\frac{1}{q}}  \bigg(\int_{x_{k-1}}^{x_k} \bigg(\int_{x_{k-1}}^t \delta\bigg)^{\frac{r}{1-r}} \delta(t) \varphi(t)^{-\frac{r}{p(1-r)}}dt \bigg)^{\frac{1-r}{r}}.
	\end{align*}
	
	\rm{(iii)} If $1\leq r<p\leq q$, then $C \approx C_{1,1} + C_{1,2} + C_{3,3} + C_{4,1}$, where
	\begin{equation*}
		C_{3,3} :=  \sup_{N+1\leq k \leq M-1} \bigg(\int_{x_k}^L \Delta^{-\frac{q}{r}} w \bigg)^{\frac{1}{q}} \bigg(\sum_{i=N+1}^k  \esup_{t\in(x_{i-1}, x_i)} \bigg( \int_{x_{i-1}}^t \delta\bigg)^{\frac{p}{p-r}}
		\varphi(t)^{-\frac{r}{p-r}}  \bigg)^{\frac{p-r}{pr}}.
	\end{equation*}
	
	\rm{(iv)} If $r<p\leq q$, $r<1\leq q$, then $C\approx C_{1,2}+C_{1,3} + C_{3,4}+C_{4,1}$, where
	\begin{align*}
		C_{3,4} &:=  \sup_{N+1\leq k \leq M-1} \bigg(\int_{x_k}^L \Delta^{-\frac{q}{r}} w \bigg)^{\frac{1}{q}} \\
		&\hspace{2cm}\times\bigg(\sum_{i=N+1}^k \bigg(\int_{x_{i-1}}^{x_i} \bigg(\int_{x_{i-1}}^t \delta\bigg)^{\frac{r}{1-r}} \delta(t) \varphi(t)^{-\frac{r}{p(1-r)}}dt \bigg)^{\frac{p(1-r)}{p-r}}  \bigg)^{\frac{p-r}{pr}}.
	\end{align*}
	
	\rm{(v)} If $p\leq q<1\leq r$, then $C\approx C_{1,4}+C_{1,5}+ C_{3,1} + C_{4,1}$, where
	\begin{align*}
		C_{1,4} &:= \sup_{N+1\leq k \leq M} \left( \int_{x_{k-1}}^{x_k} \left( \int_t^{x_k} \Delta^{-\frac{q}{r}} w \right)^\frac{q}{1-q} w(t) \Delta(t)^{-\frac{q}{r}}  \right.\\
		&\hspace{2cm} \times\left.\esup_{s\in(x_{k-1},t)} \left( \int_{x_{k-1}}^s \delta \right)^\frac{q}{r(1-q)} \varphi(s)^{-\frac{q}{p(1-q)}} dt \right)^{\frac{1-q}{q}},
		\intertext{and}
		C_{1,5}&:=  \sup_{N+1\leq k \leq M} \left( \int_{x_{k-1}}^{x_k} \left( \int_{x_{k-1}}^t \Delta(s)^{-\frac{q}{r}} w(s) \left( \int_{x_{k-1}}^s \delta \right)^{\frac{q}{r}} ds \right)^\frac{q}{1-q}  \right.\\
		&\hspace{2cm} \times\left. \Delta(t)^{-\frac{q}{r}} w(t) \left( \int_{x_{k-1}}^t \delta \right)^\frac{q}{r} \varphi(t)^{-\frac{q}{p(1-q)}} dt \right)^{\frac{1-q}{q}}.
	\end{align*}
	
	\rm{(vi)} If $p\leq q<1$, $p\leq r < 1$, then $C\approx C_{1,5} + C_{1,6} + C_{3,2} + C_{4,1}$, where
	\begin{align*}
		C_{1,6} &:=  \sup_{N+1\leq k \leq M} \left( \int_{x_{k-1}}^{x_k} \left( \int_t^{x_{k}} \Delta^{-\frac{q}{r}} w \right)^\frac{q}{1-q} \Delta(t)^{-\frac{q}{r}} w(t) \right.\\
		&\hspace{2cm} \times\left.  \left(\int_{x_{k-1}}^t \left( \int_{x_{k-1}}^s \delta \right)^\frac{r}{1-r} \delta(s)\varphi(s)^{-\frac{r}{p(1-r)}}ds \right)^{\frac{q(1-r)}{r(1-q)}} dt \right)^{\frac{1-q}{q}}.
	\end{align*}
	
	\rm{(vii)} If $r<p\leq q<1$, then $C\approx C_{1,5} + C_{1,6} + C_{3,4} + C_{4,1}$.
\end{theorem}

\begin{proof}
	Owing to Theorem~\ref{thm:discr:resonanceform}, we have
	\begin{equation}\label{thm:final_discr_form_C}
		C \approx C_1 + C_2 + C_3 + C_4,
	\end{equation}
	in which $C_1$, $C_2$, $C_3$ and $C_4$ are the optimal constants in \eqref{D1}, \eqref{D2}, \eqref{D3} and \eqref{D4}, respectively.
	
	First, we shall find equivalent expressions for $C_1$. In view of \eqref{eq:landau}, we have
	\begin{equation*}
		C_1 = \ksup{N+1}{M} A(x_{k-1}, x_k),
	\end{equation*}
	where the quantities $A(x_{k-1}, x_k)$ are defined by \eqref{thm:discr:Ak}.
	By \cite{GMPTU}*{Theorem~A} we have
	\begin{equation}\label{thm:final_discr_form_C1}
		C_1 = \ksup{N+1}{M} A(x_{k-1}, x_k) \approx \begin{cases}
			C_{1, 1} + C_{1, 2} \quad &\text{if $1 \leq \min\{q,r\}$;}\\
			C_{1, 2} + C_{1, 3} \quad &\text{if $r < 1 \leq q$;}\\
			C_{1, 4} + C_{1, 5} \quad &\text{if $q < 1 \leq r$;}\\
			C_{1, 5} + C_{1, 6} \quad &\text{if $\max\{q,r\} < 1$.}
		\end{cases}
	\end{equation}
	
	Second, we shall find equivalent expressions for $C_2$. Using \eqref{eq:Bennett_discr_Hardy} with
	$$
	b_k = B(x_{k-1}, x_k)^r, \quad  k\in\Z, \; N + 1 \leq k \leq M-1,
	$$
	where the quantities $B(x_{k-1}, x_k)$ are defined by \eqref{thm:discr:Bk}, and
	$$
	d_k = \int_{x_k}^{x_{k+1}}  \Delta^{-\frac{q}{r}} w, \quad k\in\Z, \; N + 1 \leq k \leq M-1,
	$$
	gives us
	\begin{equation*}
		C_2 \approx \begin{cases}
			\ksup{N+1}{M-1} \left(\int_{x_k}^L  \Delta^{-\frac{q}{r}} w\right)^\frac1{q} B(x_{k-1}, x_k) \, &\text{if $p \leq \min\{q,r\}$;}\\
			\ksup{N+1}{M-1} \left(\int_{x_k}^L  \Delta^{-\frac{q}{r}} w\right)^\frac1{q} \left( \ksum[i]{N+1}{k} B(x_{i-1}, x_i)^{\frac{p r}{p - r}} \right)^{\frac{p - r}{pr}} \, &\text{if $r < p \leq q$.}
		\end{cases}
	\end{equation*}
	Combining that with Remark~\ref{rem:Cts_Hardy_Char}, we obtain
	\begin{equation}\label{thm:final_discr_form_C2}
		C_2 \approx \begin{cases}
			C_{3,1} \quad &\text{if $p \leq \min\{q,r\}$, $1 \leq r$;}\\
			C_{3,3} \quad &\text{if $1\leq r < p \leq q$;}\\
			C_{3,2} \quad &\text{if $p \leq \min\{q,r\}$, $r < 1$;}\\
			C_{3,4} \quad &\text{if $r < p \leq q$, $r < 1$.}
		\end{cases}
	\end{equation}
	
	Next, we shall turn our attention to $C_3$. Using \eqref{eq:Bennett_discr_Hardy} with $b_k = \varphi(x_k)^{-\frac{r}{p}} \int_{x_{k-1}}^{x_k} \delta$ and $d_k = \int_{x_k}^{x_{k+1}}  \Delta^{-\frac{q}{r}} w$, $k\in\Z$, $N + 1 \leq k \leq M-1$, we infer that
	\begin{equation}\label{thm:final_discr_form_C3}
		C_3 \approx \begin{cases}
			C_{2,1} \quad &\text{if $p \leq \min\{q,r\}$;}\\
			C_{2,2} \quad &\text{if $r < p \leq q$},
		\end{cases}
	\end{equation}
	where
	\begin{align*}
		&C_{2,1} :=  \sup_{N+1\leq k \leq M-1} \bigg(\int_{x_k}^L \Delta^{-\frac{q}{r}} w \bigg)^{\frac{1}{q}} \bigg(\int_{x_{k-1}}^{x_k} \delta\bigg)^{\frac{1}{r}} \varphi(x_k)^{-\frac{1}{p}}, \\
		&C_{2,2} :=  \sup_{N+1\leq k \leq M-1} \bigg(\int_{x_k}^L \Delta^{-\frac{q}{r}} w \bigg)^{\frac{1}{q}} \bigg(\sum_{i=N+1}^k  \bigg(\int_{x_{i-1}}^{x_i} \delta\bigg)^{\frac{p}{p-r}} \varphi(x_i)^{-\frac{r}{p-r}}
		\bigg)^{\frac{p-r}{pr}}.
	\end{align*}
	
	Now, by the same argument as in the case $C_1$, we have
	\begin{equation}\label{thm:final_discr_form_C4}
		C_4 = \ksup{N + 1}{M - 1} \left( \int_{x_{k-1}}^{x_k} w \right)^\frac1{q} \varphi(x_k)^{-\frac1{p}} = C_{4,1}.
	\end{equation}
	Then, combining \eqref{thm:final_discr_form_C}  with \eqref{thm:final_discr_form_C1}, \eqref{thm:final_discr_form_C2}, \eqref{thm:final_discr_form_C3} and \eqref{thm:final_discr_form_C4}, we obtain
	\begin{equation}\label{C-new}
		C \approx \begin{cases}
			C_{1,1}+C_{1,2}+ C_{3,1} + C_{2,1} + C_{4,1} \quad &\text{if $p \leq q$, $p\leq r$, $1\leq q$, $1\leq r$;}\\
			C_{1,2} + C_{1,3} + C_{3,2} + C_{2,1} +C_{4,1} \quad &\text{if $p\leq r < 1 \leq q$;}\\
			C_{1,1} + C_{1,2} + C_{3,3} + C_{2,2} +C_{4,1}\quad &\text{if $1\leq r<p\leq q$;}\\
			C_{1,2}+C_{1,3} + C_{3,4}+C_{2,2} +C_{4,1} \quad &\text{if $r<p\leq q$, $r<1\leq q$;}\\
			C_{1,4}+C_{1,5}+ C_{3,1} +C_{2,1} + C_{4,1} \quad &\text{if $p\leq q<1\leq r$;}\\
			C_{1,5} + C_{1,6} + C_{3,2} +C_{2,1} + C_{4,1}\quad &\text{if $p\leq q<1$, $p\leq r < 1$;}\\
			C_{1,5} + C_{1,6} + C_{3,4} +C_{2,2} + C_{4,1} \quad &\text{if $r<p\leq q<1$.}
		\end{cases}
	\end{equation}
	On the other hand, it is clear that \begin{align}
		C_{2,1} \leq C_{3,1} \label{C21<C31}
		\intertext{and}
		C_{2,2} \leq C_{3,3}. \label{C22<C33}
	\end{align}
	Moreover, observe that for $N+1 \leq k \leq M$
	\begin{align}\label{sup-int}
		\esup_{s\in(x_{k-1}, t)} \left( \int_{x_{k-1}}^s \delta \right)^\frac1{r} \varphi(s)^{-\frac1{p}} &\approx \esup_{s\in(x_{k-1}, t)} \left(\int_{x_{k-1}}^s \left(\int_{x_{k-1}}^{\tau} \delta\right)^{\frac{r}{1-r}} \delta(\tau) \,d\tau \right)^\frac{1-r}{r} \varphi(s)^{-\frac1{p}} \nonumber\\
		&\leq \left(\int_{x_{k-1}}^t \left(\int_{x_{k-1}}^{\tau} \delta\right)^{\frac{r}{1-r}} \delta(\tau) \varphi(\tau)^{-\frac{r}{p(1-r)}} \,d\tau \right)^\frac{1-r}{r}.
	\end{align}
	Then it is clear that
	\begin{align}\label{EQ:C31lessC32}
		C_{3,1} & = \sup_{N+1\leq k \leq M-1} \bigg(\int_{x_k}^L \Delta^{-\frac{q}{r}} w \bigg)^{\frac{1}{q}} \esup_{t\in(x_{k-1}, x_k)} \bigg( \int_{x_{k-1}}^t \delta\bigg)^{\frac{1}{r}}
		\varphi(t)^{-\frac{1}{p}}\nonumber\\
		& \lesssim \sup_{N+1\leq k \leq M-1} \bigg(\int_{x_k}^L \Delta^{-\frac{q}{r}} w \bigg)^{\frac{1}{q}} \nonumber \\
		& \hspace{3cm} \times\bigg(\int_{x_{k-1}}^{x_k} \bigg(\int_{x_{k-1}}^{\tau} \delta\bigg)^{\frac{r}{1-r}} \delta(\tau) \varphi(\tau)^{-\frac{r}{p(1-r)}} \,d\tau \bigg)^\frac{1-r}{r} \nonumber \\
		& = C_{3,2},
	\end{align}
	and 
	\begin{align}\label{C33<C34}
		C_{3,3} &=  \sup_{N+1\leq k \leq M-1} \bigg(\int_{x_k}^L \Delta^{-\frac{q}{r}} w \bigg)^{\frac{1}{q}} \bigg(\sum_{i=N+1}^k  \esup_{t\in(x_{i-1}, x_i)} \bigg( \int_{x_{i-1}}^t \delta\bigg)^{\frac{p}{p-r}}
		\varphi(t)^{-\frac{r}{p-r}}  \bigg)^{\frac{p-r}{pr}}\nonumber\\
		& \lesssim \sup_{N+1\leq k \leq M-1} \bigg(\int_{x_k}^L \Delta^{-\frac{q}{r}} w \bigg)^{\frac{1}{q}} \nonumber\\
		&\hspace{2cm}\times\bigg(\sum_{i=N+1}^k \bigg(\int_{x_{i-1}}^{x_i} \bigg(\int_{x_{i-1}}^t \delta\bigg)^{\frac{r}{1-r}} \delta(t) \varphi(t)^{-\frac{r}{p(1-r)}}dt \bigg)^{\frac{p(1-r)}{p-r}}  \bigg)^{\frac{p-r}{pr}}\nonumber\\
		& = C_{3,4}.
	\end{align}
	Finally, the assertion follows from the combination of \eqref{C-new}, \eqref{C21<C31}, \eqref{C22<C33}, \eqref{EQ:C31lessC32} and \eqref{C33<C34}.

\section{Antidiscretization}\label{S:antidiscretization}

We start with a technical lemma, which will prove useful later.
\begin{lemma}\label{TH:antid_lemma1}
	Let $0< r < p <\infty$. Let $\varphi$ and $\sigma$ be functions defined by \eqref{varphi} and \eqref{E:sigma}, respectively. Assume that there is $t_0\in(0,L)$ such that $0 < \varphi(t_0) < \infty$. Suppose that $\{x_k\}_{k=N}^{M}\in CS(\varphi,U^p,a)$ and $\mathcal{Z}_1, \mathcal{Z}_2$ are the decomposition of the index set $\K^+=\{k\in\Z\colon N+1\leq k \leq M\}$ given in Lemma~\ref{Lem:Decomp}.
	\begin{enumerate}
		
		\item Let $i\in\Z$, $N+2\leq i \leq M$ and $y\in[x_{i-1}, x_i]$. Let $h\in Q_{U}(0,y)$. We have
		\begin{align}
			&\int_{x_{i-1}}^{y}\sigma(t) h(t)^{\frac{pr}{p-r}}\,dt \lesssim h(y)^{\frac{pr}{p-r}}\varphi(y)^{-\frac{r}{p-r}} \label{EQ:antid_lemma_1}\quad \text{if $i\in\mathcal{Z}_1$,} \\
			\intertext{and}
			&\int_{x_{i-1}}^{y} \sigma(t)h(t)^{\frac{pr}{p-r}}\,dt \lesssim h(x_{i-1})^{\frac{pr}{p-r}}\varphi(x_{i-1})^{-\frac{r}{p-r}} \label{EQ:antid_lemma_2} \quad \text{if $i\in\mathcal{Z}_2$}.
		\end{align}
		
		\item If $N>-\infty$, then $N+1 \in \mathcal{Z}_2$ and, for every $y\in(0, x_{N+1}]$ and $h\in Q_{U}(0,y)$,
		\begin{equation}\label{EQ:antid_lemma_3}
			\int_{0}^{y}\sigma(t)h(t)^{\frac{pr}{p-r}}\,dt \lesssim \sup_{t\in(0,y]}h(t)^{\frac{pr}{p-r}}\varphi(t)^{-\frac{r}{p-r}}.
		\end{equation}
		
		\item Let $k\in\Z$, $N+1\leq k \leq M$. If $h\in Q_{U}(0,x_k)$, then
		\begin{align}
			\sum_{i=N+1}^{k-1}h(x_i)^{\frac{pr}{p-r}}\varphi(x_i)^{-\frac{r}{p-r}}
			&\lesssim
			\int_0^{x_k}\sigma(t)h(t)^{\frac{pr}{p-r}}\,dt \nonumber\\
			&\hspace{-2cm}\lesssim
			\sum_{i=N+1}^{k}h(x_i)^{\frac{pr}{p-r}}\varphi(x_i)^{-\frac{r}{p-r}}
			+
			\sup_{t\in(0,x_{N+1})}h(t)^{\frac{pr}{p-r}}\varphi(t)^{-\frac{r}{p-r}}. \label{EQ:antid_lemma_4}
		\end{align}
	\end{enumerate}
	
	The multiplicative constants in this lemma depend only on $a$, $p$, and $r$.
\end{lemma}

\begin{proof}
	First, observe that
	\begin{equation}\label{TH:antid_lemma1_varphi}
		\varphi(t) = V(t) + U(t)^p \left( \int_t^L U^{-p} v \right) \quad \text{for every $t\in(0,L)$}.
	\end{equation}
	For future reference, note that
	\begin{align}
		V(t) &\leq \varphi(t) \quad \text{for every $t\in(0,L)$} \label{TH:antid_lemma1_V_dominated_by_varphi}\\
		\intertext{and}
		U(t)^p \left( \int_t^L U^{-p} v \right) &\leq \varphi(t) \quad \text{for every $t\in(0,L)$}. \label{TH:antid_lemma1_Up_times_integral_dominated_by_varphi}
	\end{align}
	Furthermore, we have
	\begin{equation}\label{TH:antid_lemma1_varphi_derivative}
		\varphi'(t) = p U(t)^{p-1} u(t) \left( \int_t^L U^{-p} v \right) \quad \text{for a.e.~$t\in(0,L)$}
	\end{equation}
	and
	\begin{equation}\label{TH:antid_lemma1_varphi_over_Up_derivative}
		\left( \frac{\varphi}{U^p} \right)'(t) = - p V(t) U(t)^{-p-1} u(t) \quad \text{for a.e.~$t\in(0,L)$}.
	\end{equation}
	Second, recall that, for $N+2 \leq i \leq M$,
	\begin{align}
		\varphi(y)\approx \varphi(x_{i-1})\quad&\text{for all $y\in [x_{i-1},x_i]$ and every $i\in\mathcal{Z}_1$}, \label{TH:antid_lemma1_varphi_Z1}\\
		\intertext{and}
		\frac{\varphi}{U^p}(y)\approx \frac{\varphi}{U^p}(x_{i-1}) \quad&\text{for all $y\in [x_{i-1},x_i]$ and every $i\in\mathcal{Z}_2$}, \label{TH:antid_lemma1_varphi_Z2}
	\end{align}
	since $\varphi \in Q_{U^p}(0,L)$ and $\{x_k\}_{k=N}^{M}\in CS(\varphi,U^p,a)$. The multiplicative constants depend only on $a$.
	
	Now, if $i\in\mathcal{Z}_1$, $N+2\leq i\leq M$, then, for each $y\in [x_{i-1},x_i]$, we have
	\begin{align*}
		&\int_{x_{i-1}}^{y}\varphi(t)^{-\frac{r}{p-r}-2}V(t)\left(\int_t^L U^{-p}v\right)U(t)^{p-1}u(t)h(t)^{\frac{pr}{p-r}}\,dt\\
		&\qquad\leq
		h(y)^{\frac{pr}{p-r}}\int_{x_{i-1}}^{y}\varphi(t)^{-\frac{r}{p-r}-2}V(t)\left(\int_t^L U^{-p}v\right)U(t)^{p-1}u(t)\,dt\\
		&\qquad\leq
		h(y)^{\frac{pr}{p-r}}\int_{x_{i-1}}^{y}\varphi(t)^{-\frac{r}{p-r}-1}\left(\int_t^L U^{-p}v\right)U(t)^{p-1}u(t)\,dt\\
		&\qquad\approx
		h(y)^{\frac{pr}{p-r}}\int_{x_{i-1}}^{y}d\left[-\varphi^{-\frac{r}{p-r}}\right]
		\leq
		h(y)^{\frac{pr}{p-r}}\varphi(x_{i-1})^{-\frac{r}{p-r}}\approx
		h(y)^{\frac{pr}{p-r}}\varphi(y)^{-\frac{r}{p-r}}.
	\end{align*}
	We used the fact that the function $h$ is nondecreasing and \eqref{TH:antid_lemma1_V_dominated_by_varphi} in the first and second inequalities, respectively. We used \eqref{TH:antid_lemma1_varphi_derivative} together with a change of variables in the first equivalence and \eqref{TH:antid_lemma1_varphi_Z1} in the second one. Thus, we have proved \eqref{EQ:antid_lemma_1}.
	On the other hand, if $i\in\mathcal{Z}_2$, $N+2\leq i\leq M$, then, for each $y\in [x_{i-1},x_i]$, we have
	\begin{align}
		&\int_{x_{i-1}}^{y}\varphi(t)^{-\frac{r}{p-r}-2}V(t)\left(\int_t^L U^{-p}v\right)U(t)^{p-1}u(t)h(t)^{\frac{pr}{p-r}}\,dt \notag\\
		&\qquad\leq\left(\frac{h(x_{i-1})}{U(x_{i-1})}\right)^{\frac{pr}{p-r}} \label{TH:antid_lemma1_eq1}
		\int_{x_{i-1}}^{y}\varphi^{-\frac{r}{p-r}-1}VU^{-1}uU^{\frac{pr}{p-r}}\\
		&\qquad\approx
		\left(\frac{h(x_{i-1})}{U(x_{i-1})}\right)^{\frac{pr}{p-r}}\int_{x_{i-1}}^{y}d\left[\left(\frac{\varphi}{U^p}\right)^{-\frac{r}{p-r}}\right] \label{TH:antid_lemma1_eq2}\\
		&\qquad \leq
		\left(\frac{h(x_{i-1})}{U(x_{i-1})}\right)^{\frac{pr}{p-r}}\frac{\varphi(y)^{-\frac{r}{p-r}}}{U(y)^{-\frac{pr}{p-r}}}
		\approx h(x_{i-1})^{\frac{pr}{p-r}}\varphi(x_{i-1})^{-\frac{r}{p-r}}. \notag
	\end{align}
	We used the fact the function  $h/U$ is nonincreasing combined with \eqref{TH:antid_lemma1_Up_times_integral_dominated_by_varphi} in the first inequality. The first equivalence follows from \eqref{TH:antid_lemma1_varphi_over_Up_derivative} together with a change of variables, and the second one follows from \eqref{TH:antid_lemma1_varphi_Z2}. Thus, we have proved \eqref{EQ:antid_lemma_2}.
	
	Next, we shall prove \eqref{EQ:antid_lemma_3}. To this end, note that $\lim_{t\rightarrow0+}\varphi(t)=0$ thanks to the dominated convergence theorem, inasmuch as there is $t_0\in(0,L)$ such that $\varphi(t_0) < \infty$. It follows that $N+1\in\mathcal{Z}_2$ if $N>-\infty$. Indeed, suppose that $N>-\infty$ and that $N+1\in\mathcal{Z}_1$. If that were the case, we would have
	\begin{equation*}
		\varphi(t) \approx \varphi(x_{N+1}) > 0 \quad \text{for every $t\in(0, x_{N+1}]$},
	\end{equation*}
	which would contradict the fact that $\lim_{t\rightarrow0+}\varphi(t)=0$. Hence
	\begin{equation*}
		\lim_{t\rightarrow0+}\left(\frac{U(t)^{p}}{\varphi(t)}\right)^{\frac{r}{p-r}} \in (0, \infty) \qquad \text{if $N>-\infty$}.
	\end{equation*}
	Therefore, if $N>-\infty$, then, for every $y\in(0, x_{N+1}]$,
	\begin{align*}
		&\int_{0}^{y}\varphi(t)^{-\frac{r}{p-r}-2}V(t)\left(\int_t^L U^{-p}v\right)U(t)^{p-1}u(t)h(t)^{\frac{pr}{p-r}}\,dt\\
		&\qquad\lesssim
		\lim_{t\rightarrow0+}\left(\frac{h(t)}{U(t)}\right)^{\frac{pr}{p-r}}\int_{0}^{y}d\left[\left(\frac{\varphi}{U^p}\right)^{-\frac{r}{p-r}}\right]
		\leq
		\lim_{t\rightarrow0+}\left(\frac{h(t)}{U(t)}\right)^{\frac{pr}{p-r}}\lim_{t\rightarrow0+}\left(\frac{U(t)^{p}}{\varphi(t)}\right)^{\frac{r}{p-r}} \\
		&\qquad=  \lim_{t\rightarrow0+}\left(\frac{h(t)}{U(t)}\right)^{\frac{pr}{p-r}}\left(\frac{\varphi(y)}{U(y)^p}\right)^{-\frac{r}{p-r}} \lesssim
		\sup_{t\in(0,y]}h(t)^{\frac{pr}{p-r}}\varphi(t)^{-\frac{r}{p-r}}.
	\end{align*}
	We obtain the first inequality by combining the same arguments as in \eqref{TH:antid_lemma1_eq1} and \eqref{TH:antid_lemma1_eq2}. Thus, we have proved \eqref{EQ:antid_lemma_3}.
	
	Finally, it remains to prove \eqref{EQ:antid_lemma_4}. Note that it clearly holds if $k=N+1 > -\infty$ thanks to \eqref{EQ:antid_lemma_3}. Let $k\in\Z$, $N+2\leq k\leq M$. We have, whether $N > -\infty$ or $N = -\infty$,
	\begin{align*}
		\int_{0}^{x_{k}}\sigma h^{\frac{pr}{p-r}}
		&=
		\int_{0}^{x_{N+1}}\sigma h^{\frac{pr}{p-r}}
		+\sum_{i\in\mathcal{Z}_1, N+2\leq i\leq k}\int_{x_{i-1}}^{x_{i}}\sigma h^{\frac{pr}{p-r}}\\
		&\quad+\sum_{i\in\mathcal{Z}_2, N+2\leq i\leq k}\int_{x_{i-1}}^{x_{i}}\sigma h^{\frac{pr}{p-r}}\\
		&\lesssim
		\sup_{t\in(0,x_{N+1})}h(t)^{\frac{pr}{p-r}}\varphi(t)^{-\frac{r}{p-r}}
		+\sum_{i\in\mathcal{Z}_1, N+2\leq i\leq k}h(x_i)^{\frac{pr}{p-r}}\varphi(x_i)^{-\frac{r}{p-r}}\\
		&\quad+\sum_{i\in\mathcal{Z}_2, N+2\leq i\leq k}h(x_{i-1})^{\frac{pr}{p-r}}\varphi(x_{i-1})^{-\frac{r}{p-r}}\\
		&\lesssim
		\sup_{t\in(0,x_{N+1})}h(t)^{\frac{pr}{p-r}}\varphi(t)^{-\frac{r}{p-r}}+\sum_{i=N+1}^k h(x_i)^{\frac{pr}{p-r}}\varphi(x_i)^{-\frac{r}{p-r}},
	\end{align*}
	where we combined \eqref{EQ:antid_lemma_1}, \eqref{EQ:antid_lemma_2} and \eqref{EQ:antid_lemma_3} in the first inequality. Thus, we have proved the second inequality in \eqref{EQ:antid_lemma_4}. As for the other, note that
	\begin{equation}\label{TH:antid_lemma1_eq3}
		\sum_{i=N+1}^{k-1} \int_{x_{i-1}}^{x_i}\sigma h^{\frac{pr}{p-r}} + \sum_{i=N+1}^{k-1} \int_{x_{i}}^{x_{i+1}}\sigma h^{\frac{pr}{p-r}} \lesssim \int_0^{x_k}\sigma h^{\frac{pr}{p-r}}.
	\end{equation}
	
	Recalling \eqref{E:sigma} and using the fact that the function $h/U$ is nonincreasing, we have
	\begin{equation*}
		\sigma(t) h(t)^{\frac{pr}{p - r}} \geq \left( \frac{h(x_i)}{U(x_i)} \right)^{\frac{pr}{p - r}} \left(\int_{x_i}^L U^{-p}v\right) U(t)^{\frac{pr}{p - r}}  \varphi(t)^{-\frac{r}{p-r}-2}V(t)U(t)^{p-1}u(t)
	\end{equation*}
	for every $t\in[x_{i-1}, x_i]$, $i\in\{N+1, \dots, k - 1\} \cap \Z$. Furthermore, using \eqref{TH:antid_lemma1_varphi_over_Up_derivative}, we observe that
	\begin{equation*}
		\left( \left(U^{-p}\varphi\right)^{-\frac{r}{p-r}-1} \right)'(t) \approx U(t)^{\frac{pr}{p - r}}  \varphi(t)^{-\frac{r}{p-r}-2}V(t)U(t)^{p-1}u(t) \quad \text{for a.e.~$t\in(0, L)$}.
	\end{equation*}
	Now, combining these two observations, we arrive at
	\begin{align*}
		&\sum_{i=N+1}^{k-1} \int_{x_{i-1}}^{x_i}\sigma h^{\frac{pr}{p-r}} \gtrsim \sum_{i=N+1}^{k-1} \left(\frac{h(x_i)}{U(x_i)}\right)^{\frac{pr}{p-r}}\left(\int_{x_{i}}^LU^{-p}
		v\right)\int_{x_{i-1}}^{x_i}d\left[\left(U^{-p}\varphi\right)^{-\frac{r}{p-r}-1}\right]\\
		&\quad=
		\sum_{i=N+1}^{k-1}\left(\frac{h(x_i)}{U(x_i)}\right)^{\frac{pr}{p-r}}\left(\int_{x_{i}}^LU^{-p}
		v\right)\left(\left(U^{-p}\varphi\right)(x_i)^{-\frac{r}{p-r}-1}-\left(U^{-p}\varphi\right)(x_{i-1})^{-\frac{r}{p-r}-1}\right).
	\end{align*}
	Moreover, as $\{x_k\}_{k=N}^{M}\in CS(\varphi,U^p,a)$, the sequence $\{\left(U^{-p}\varphi\right)^{-\frac{r}{p-r}}(x_i)\}_{i=N+1}^{M-1}$ is strongly increasing, and we have
	\begin{equation*}
		\left(U^{-p}\varphi\right)(x_i)^{-\frac{r}{p-r}-1}-\left(U^{-p}\varphi\right)(x_{i-1})^{-\frac{r}{p-r}-1} \geq \big(1-a^{-\frac{r}{p-r} - 1}\big) \left(U^{-p}\varphi\right)(x_i)^{-\frac{r}{p-r}-1}.
	\end{equation*}
	Note that $a^{-r/(p-r) - 1}\in(0,1)$. Therefore,
	\begin{equation}\label{TH:antid_lemma1_eq4}
		\sum_{i=N+1}^{k-1} \int_{x_{i-1}}^{x_i}\sigma h^{\frac{pr}{p-r}} \gtrsim \sum_{i=N+1}^{k-1}\left(\frac{h(x_i)}{U(x_i)}\right)^{\frac{pr}{p-r}}\left(\int_{x_{i}}^LU^{-p}
		v\right)\left(U^{-p}\varphi\right)(x_i)^{-\frac{r}{p-r}-1}.
	\end{equation}
	
	As for the second sum in \eqref{TH:antid_lemma1_eq3}, we have
	\begin{equation*}
		\sigma(t) h(t)^{\frac{pr}{p - r}} \geq h(x_i)^{\frac{pr}{p - r}} V(x_i) \varphi(t)^{-\frac{r}{p-r}-2} \left(\int_t^L U^{-p}v\right) U(t)^{p-1} u(t)
	\end{equation*}
	for every $t\in[x_i, x_{i + 1}]$, $i\in\{N+1, \dots, k - 1\} \cap \Z$, thanks to \eqref{E:sigma} and the fact that the functions $h$ and $V$ are nondecreasing. Combining this with \eqref{TH:antid_lemma1_varphi_derivative}, we observe that
	\begin{align*}
		&\sum_{i=N+1}^{k-1}\int_{x_{i}}^{x_{i+1}}\sigma h^{\frac{pr}{p-r}} \gtrsim h(x_i)^{\frac{pr}{p-r}}V(x_i)\int_{x_{i}}^{x_{i+1}}d\left[-\varphi^{-\frac{r}{p-r}-1}\right]\\
		&\quad= h(x_i)^{\frac{pr}{p-r}}V(x_i)\left(\varphi(x_i)^{-\frac{r}{p-r}-1}-\varphi(x_{i+1})^{-\frac{r}{p-r}-1}\right).
	\end{align*}
	Since $\{x_k\}_{k=N}^{M}\in CS(\varphi,U^p,a)$, the sequence $\{\varphi(x_i)^{-\frac{r}{p-r}-1}\}_{i=N+1}^{M-1}$ is strongly  decreasing, and we have
	\begin{equation*}
		\varphi(x_i)^{-\frac{r}{p-r}-1}-\varphi(x_{i+1})^{-\frac{r}{p-r}-1} \geq \big(1-a^{-\frac{r}{p-r} - 1}\big) \varphi(x_i)^{-\frac{r}{p-r}-1}.
	\end{equation*}
	By combining the last two inequalities, we obtain
	\begin{equation}\label{TH:antid_lemma1_eq5}
		\sum_{i=N+1}^{k-1}\int_{x_{i}}^{x_{i+1}}\sigma h^{\frac{pr}{p-r}} \gtrsim h(x_i)^{\frac{pr}{p-r}}V(x_i)\varphi(x_i)^{-\frac{r}{p-r}-1}.
	\end{equation}
	Hence, thanks to \eqref{TH:antid_lemma1_eq4} and \eqref{TH:antid_lemma1_eq5}, we have
	\begin{align*}
		&\sum_{i=N+1}^{k-1} \int_{x_{i-1}}^{x_i}\sigma h^{\frac{pr}{p-r}} + \sum_{i=N+1}^{k-1} \int_{x_{i}}^{x_{i+1}}\sigma h^{\frac{pr}{p-r}} \\
		&\qquad\gtrsim \sum_{i=N+1}^{k-1}\Bigg(\bigg(\frac{h(x_i)}{U(x_i)}\bigg)^{\frac{pr}{p-r}}\bigg(\int_{x_{i}}^LU^{-p} v \bigg)\varphi(x_i)^{-\frac{r}{p-r}-1}U(x_{i})^{\frac{pr}{p-r}+p}\\
		&\qquad\quad+
		h(x_i)^{\frac{pr}{p-r}}V(x_i)\varphi(x_i)^{-\frac{r}{p-r}-1}\Bigg)\\
		&\qquad=
		\sum_{i=N+1}^{k-1}h(x_i)^{\frac{pr}{p-r}}\varphi(x_i)^{-\frac{r}{p-r}-1}\bigg(U(x_i)^p\bigg(\int_{x_{i}}^LU^{-p}
		v\bigg)+V(x_i)\bigg)\\
		&\qquad=\sum_{i=N+1}^{k-1}h(x_i)^{\frac{pr}{p-r}}\varphi(x_i)^{-\frac{r}{p-r}},
	\end{align*}
	where we used \eqref{TH:antid_lemma1_varphi} in the last equality. Finally, the first inequality in \eqref{EQ:antid_lemma_4} follows from this combined with \eqref{TH:antid_lemma1_eq3}.
\end{proof}

We now apply the preceding lemma to two particular choices of $h$, which will be useful for future reference.

\begin{remark}
	Let $r, p$, $\varphi$ and $\{x_k\}_{k=N}^{M}$ be as in Lemma~\ref{TH:antid_lemma1}. Let $u,\delta$ be weights on $(0,L)$ and $\Delta, \, U$ are primitive functions of $u,\, \delta$, respectively. Let $k\in\Z$, $N+1\leq k \leq M-1$. In what follows, the multiplicative constants depend only on the discretization parameter $a$ and the exponents $p$ and $r$.
	
	{\rm(i)} Consider the function
	\begin{equation*}
		h(t)=U(t)\sup_{\tau\in(t,x_k)}\Delta(\tau)^{\frac{1}{r}}U(\tau)^{-1},\ t\in(0,x_k).
	\end{equation*}
	Clearly, the function $h/U$ is nonincreasing on $(0,x_k)$. Furthermore, we have, for every $0 < s < t < x_k$,
	\begin{align*}
		&h(s) = U(s) \max\left\{ \sup_{\tau\in(s, t]}\Delta(\tau)^{\frac{1}{r}}U(\tau)^{-1}, \sup_{\tau\in(t, x_k)}\Delta(\tau)^{\frac{1}{r}}U(\tau)^{-1} \right\} \\
		&\qquad \leq \max\left\{ \Delta(t)^{\frac{1}{r}}, U(t) \sup_{\tau\in(t, x_k)}\Delta(\tau)^{\frac{1}{r}}U(\tau)^{-1} \right\} \leq U(t)\sup_{\tau\in(t,x_k)}\Delta(\tau)^{\frac{1}{r}}U(\tau)^{-1}\\
		&\qquad = h(t).
	\end{align*}
	Hence $h\in Q_U(0,x_k)$. Now, plugging this function $h$ into \eqref{EQ:antid_lemma_4}, we obtain (recall \eqref{E:sigma})
	\begin{align}
		&\hspace{-2cm}\sum_{i=N+1}^{k-1} \varphi(x_i)^{-\frac{r}{p-r}}U(x_i)^{\frac{pr}{p-r}} \sup_{\tau\in(x_i,x_k)}\Delta(\tau)^{\frac{p}{p-r}}U(\tau)^{-\frac{pr}{p-r}}\nonumber\\
		&\hspace{-1cm}\lesssim  \int_0^{x_k}\varphi(t)^{-\frac{r}{p-r}-2}V(t)\left(\int_t^L U^{-p}v\right)U(t)^{p-1}u(t)U(t)^{\frac{pr}{p-r}} \nonumber\\
		& \hspace{2cm} \times
		\sup_{\tau\in(t,x_k)}\Delta(\tau)^{\frac{p}{p-r}}U(\tau)^{-\frac{pr}{p-r}}\,dt\label{R1}\\
		&\hspace{-1cm}\lesssim
		\sum_{i=N+1}^{k}\varphi(x_i)^{-\frac{r}{p-r}}U(x_i)^{\frac{pr}{p-r}}\sup_{\tau\in(x_i,x_k)}\Delta(\tau)^{\frac{p}{p-r}}U(\tau)^{-\frac{pr}{p-r}}\nonumber\\
		&\quad+
		\sup_{t\in(0,x_{N+1})}\varphi(t)^{-\frac{r}{p-r}}U(t)^{\frac{pr}{p-r}}\sup_{\tau\in(t,x_k)}\Delta(\tau)^{\frac{p}{p-r}}U(\tau)^{-\frac{pr}{p-r}}.\nonumber
	\end{align}
	Furthermore, using the fact that the function $U^p/\varphi$ is nondecreasing, we observe that
	\begin{align*}
		&\sup_{t\in(0,x_{N+1})}\varphi(t)^{-\frac{r}{p-r}}U(t)^{\frac{pr}{p-r}}\sup_{\tau\in(t,x_k)}\Delta(\tau)^{\frac{p}{p-r}}U(\tau)^{-\frac{pr}{p-r}} \\
		&\qquad\leq  \sup_{t\in(0,x_{N+1})}\varphi(t)^{-\frac{r}{p-r}}U(t)^{\frac{pr}{p-r}}\sup_{\tau\in(t,x_{N+1})}\Delta(\tau)^{\frac{p}{p-r}}U(\tau)^{-\frac{pr}{p-r}} \\
		&\qquad\quad+ \sum_{i=N+1}^{k}\varphi(x_i)^{-\frac{r}{p-r}}U(x_i)^{\frac{pr}{p-r}}\sup_{\tau\in(x_i,x_k)}\Delta(\tau)^{\frac{p}{p-r}}U(\tau)^{-\frac{pr}{p-r}}.
	\end{align*}
	Hence, combining this with the preceding inequality, we obtain
	\begin{align}
		&\hspace{-1.5cm} \int_0^{x_k}\varphi(t)^{-\frac{r}{p-r}-2}V(t)\left(\int_t^L U^{-p}v\right)U(t)^{p-1}u(t)U(t)^{\frac{pr}{p-r}} \nonumber\\
		& \hspace{2cm} \times
		\sup_{\tau\in(t,x_k)}\Delta(\tau)^{\frac{p}{p-r}}U(\tau)^{-\frac{pr}{p-r}}\,dt\nonumber\\
		&\hspace{-1cm}\lesssim
		\sum_{i=N+1}^{k}\varphi(x_i)^{-\frac{r}{p-r}}U(x_i)^{\frac{pr}{p-r}}\sup_{\tau\in(x_i,x_k)}\Delta(\tau)^{\frac{p}{p-r}}U(\tau)^{-\frac{pr}{p-r}}\nonumber\\
		&\quad+
		\sup_{t\in(0,x_{N+1})}\varphi(t)^{-\frac{r}{p-r}}U(t)^{\frac{pr}{p-r}}\sup_{\tau\in(t,x_{N+1})}\Delta(\tau)^{\frac{p}{p-r}}U(\tau)^{-\frac{pr}{p-r}}\nonumber\\
		& \hspace{-1cm}=
		\sum_{i=N+1}^{k}\varphi(x_i)^{-\frac{r}{p-r}}U(x_i)^{\frac{pr}{p-r}}\sup_{\tau\in(x_i,x_k)}\Delta(\tau)^{\frac{p}{p-r}}U(\tau)^{-\frac{pr}{p-r}} \nonumber\\
		& \quad
		+\sup_{t\in(0,x_{N+1})}\varphi(t)^{-\frac{r}{p-r}}\Delta(t)^{\frac{p}{p-r}}\label{R2}.
	\end{align}
	The equality is obtained by interchanging the order of the suprema in the second term and exploiting the fact that the function $U^p/\varphi$ is nondecreasing.
	
	{\rm(ii)} Assume, in addition, that $r\in(0, 1)$. This time we consider the function $h$ defined as
	\begin{equation*}
		h(t)=\left(\int_0^{x_k}\Delta(s)^{\frac{r}{1-r}}\delta(s)U(s)^{-\frac{r}{1-r}}\min\{U(s)^{\frac{r}{1-r}},U(t)^{\frac{r}{1-r}}\}\,ds\right)^{\frac{1-r}{r}},\ t\in(0,x_k).
	\end{equation*}
	Note that $h\in Q_U(0,x_k)$. Plugging this into \eqref{EQ:antid_lemma_4}, we obtain (recall \eqref{E:sigma})
	\begin{align}
		&\sum_{i=N+1}^{k-1}\varphi(x_i)^{-\frac{r}{p-r}}\left(\int_0^{x_k}\Delta(s)^{\frac{r}{1-r}}\delta(s)U(s)^{-\frac{r}{1-r}}\min\{U(s)^{\frac{r}{1-r}},U(x_i)^{\frac{r}{1-r}}\}\,ds\right)^{\frac{p(1-r)}{p-r}}\nonumber\\
		&\lesssim
		\int_0^{x_k}\sigma(t)\left(\int_0^{x_k}\Delta(s)^{\frac{r}{1-r}}\delta(s)U(s)^{-\frac{r}{1-r}}\min\{U(s)^{\frac{r}{1-r}},U(t)^{\frac{r}{1-r}}\}\,ds\right)^{\frac{p(1-r)}{p-r}}\,dt \label{R3}\\
		&\lesssim \sum_{i=N+1}^{k}\varphi(x_i)^{-\frac{r}{p-r}} \nonumber\\
		& \hspace{1.5cm} \times \left(\int_0^{x_k}\Delta(s)^{\frac{r}{1-r}}\delta(s)U(s)^{-\frac{r}{1-r}}\min\{U(s)^{\frac{r}{1-r}},U(x_i)^{\frac{r}{1-r}}\}\,ds\right)^{\frac{p(1-r)}{p-r}}\nonumber\\
		&\quad+  \sup_{t\in(0,x_{N+1})}\varphi(t)^{-\frac{r}{p-r}} \nonumber\\
		& \hspace{1.5cm} \times  \left(\int_0^{x_k}\Delta(s)^{\frac{r}{1-r}}\delta(s)U(s)^{-\frac{r}{1-r}}\min\{U(s)^{\frac{r}{1-r}},U(t)^{\frac{r}{1-r}}\}\,ds\right)^{\frac{p(1-r)}{p-r}}.\nonumber
	\end{align}
	Furthermore, since the function $U^{r/(1-r)}$ is nondecreasing as $r\in(0,1)$, we have
	\begin{align}
		&\int_0^{x_k}\Delta(s)^{\frac{r}{1-r}}\delta(s)U(s)^{-\frac{r}{1-r}}\min\{U(s)^{\frac{r}{1-r}},U(t)^{\frac{r}{1-r}}\}\,ds \nonumber\\
		&\qquad = \int_0^t\Delta^{\frac{r}{1-r}}\delta + U(t)^{\frac{r}{1-r}} \int_t^{x_k}\Delta^{\frac{r}{1-r}}\delta U^{-\frac{r}{1-r}} \nonumber\\
		&\qquad = (1-r)\Delta(t)^{\frac{1}{1-r}} + U(t)^{\frac{r}{1-r}} \int_t^{x_k}\Delta^{\frac{r}{1-r}}\delta U^{-\frac{r}{1-r}} \label{TH:antid_lemma1_rem_aux_eq1}
	\end{align}
	for every $t\in(0, x_k)$, where we used a change of variables in the second equality. Now, using the fact that the function $U^p/\varphi$ is nondecreasing, we obtain
	\begin{align*}
		&\sup_{t\in(0,x_{N+1})}\varphi(t)^{-\frac{r}{p-r}}U(t)^\frac{pr}{p-r}\left(\int_t^{x_k}\Delta^{\frac{r}{1-r}}\delta U^{-\frac{r}{1-r}}\right)^{\frac{p(1-r)}{p-r}} \\
		&\qquad \leq \sup_{t\in(0,x_{N+1})}\varphi(t)^{-\frac{r}{p-r}}U(t)^\frac{pr}{p-r}\left(\int_t^{x_{N+1}}\Delta^{\frac{r}{1-r}}\delta U^{-\frac{r}{1-r}}\right)^{\frac{p(1-r)}{p-r}} \\
		&\qquad\quad + \varphi(x_{N+1})^{-\frac{r}{p-r}}U(x_{N+1})^\frac{pr}{p-r}\left(\int_{x_{N+1}}^{x_k}\Delta^{\frac{r}{1-r}}\delta U^{-\frac{r}{1-r}}\right)^{\frac{p(1-r)}{p-r}}.
	\end{align*}
	Hence, combining the last two observations together, we have
	\begin{align}
		&\sup_{t\in(0,x_{N+1})}\varphi(t)^{-\frac{r}{p-r}}\left(\int_0^{x_k}\Delta(s)^{\frac{r}{1-r}}\delta(s)U(s)^{-\frac{r}{1-r}}\min\{U(s)^{\frac{r}{1-r}},U(t)^{\frac{r}{1-r}}\}\,ds\right)^{\frac{p(1-r)}{p-r}} \nonumber\\
		&\qquad\approx
		\sup_{t\in(0,x_{N+1})}\varphi(t)^{-\frac{r}{p-r}}\left(\Delta(t)^\frac{1}{1-r}+U(t)^\frac{r}{1-r}\int_t^{x_k}\Delta^{\frac{r}{1-r}}\delta U^{-\frac{r}{1-r}}\right)^{\frac{p(1-r)}{p-r}} \nonumber\\
		&\qquad\approx
		\sup_{t\in(0,x_{N+1})}\varphi(t)^{-\frac{r}{p-r}}\Delta(t)^\frac{p}{p-r} \nonumber\\
		& \qquad\quad +
		\sup_{t\in(0,x_{N+1})}\varphi(t)^{-\frac{r}{p-r}}U(t)^\frac{pr}{p-r}\left(\int_t^{x_k}\Delta^{\frac{r}{1-r}}\delta U^{-\frac{r}{1-r}}\right)^{\frac{p(1-r)}{p-r}} \nonumber\\
		&\qquad\leq
		\sup_{t\in(0,x_{N+1})}\varphi(t)^{-\frac{r}{p-r}}\Delta(t)^\frac{p}{p-r} \nonumber\\
		& \qquad\quad +
		\sup_{t\in(0,x_{N+1})}\varphi(t)^{-\frac{r}{p-r}}U(t)^\frac{pr}{p-r}\left(\int_t^{x_{N+1}}\Delta^{\frac{r}{1-r}}\delta U^{-\frac{r}{1-r}}\right)^{\frac{p(1-r)}{p-r}} \nonumber\\
		&\qquad\quad +
		\varphi(x_{N+1})^{-\frac{r}{p-r}}U(x_{N+1})^\frac{pr}{p-r}\left(\int_{x_{N+1}}^{x_k}\Delta^{\frac{r}{1-r}}\delta U^{-\frac{r}{1-r}}\right)^{\frac{p(1-r)}{p-r}} \nonumber\\
		&\qquad\lesssim \sup_{t\in(0,x_{N+1})}\varphi(t)^{-\frac{r}{p-r}}\left(\Delta(t)^{\frac{1}{1-r}}+U(t)^{\frac{r}{1-r}}\int_{t}^{x_{N+1}}\Delta^{\frac{r}{1-r}}\delta U^{-\frac{r}{1-r}}\right)^{\frac{p(1-r)}{p-r}} \nonumber\\
		&\qquad\quad+ \sum_{i=N+1}^{k}\varphi(x_i)^{-\frac{r}{p-r}}\left(U(x_i)^{\frac{r}{1-r}}\int_{x_i}^{x_k}\Delta^{\frac{r}{1-r}}\delta U^{-\frac{r}{1-r}}\right)^{\frac{p(1-r)}{p-r}}. \label{TH:antid_lemma1_rem_aux_eq2}
	\end{align}
	
	Finally, thanks to \eqref{R3} combined with \eqref{TH:antid_lemma1_rem_aux_eq1} and \eqref{TH:antid_lemma1_rem_aux_eq2}, we arrive at
	\begin{align}
		&\hspace{-0.5cm}\int_0^{x_k}\varphi(t)^{-\frac{r}{p-r}-2}V(t)\left(\int_t^L U^{-p}v\right)U(t)^{p-1}u(t)
		\nonumber\\
		&\hspace{1cm}\times\left(\int_0^{x_k}\Delta(s)^{\frac{r}{1-r}}\delta(s)U(s)^{-\frac{r}{1-r}}\min\{U(s)^{\frac{r}{1-r}},U(t)^{\frac{r}{1-r}}\}\,ds\right)^{\frac{p(1-r)}{p-r}}\,dt\nonumber\\
		&\lesssim
		\sum_{i=N+1}^{k}\varphi(x_i)^{-\frac{r}{p-r}}\left(\Delta(x_i)^{\frac{1}{1-r}}+U(x_i)^{\frac{r}{1-r}}\int_{x_i}^{x_k}\Delta^{\frac{r}{1-r}}\delta U^{-\frac{r}{1-r}}\right)^{\frac{p(1-r)}{p-r}}\nonumber\\
		&\quad+
		\sup_{t\in(0,x_{N+1})}\varphi(t)^{-\frac{r}{p-r}}\left(\Delta(t)^{\frac{1}{1-r}}+U(t)^{\frac{r}{1-r}}\int_{t}^{x_{N+1}}\Delta^{\frac{r}{1-r}}\delta U^{-\frac{r}{1-r}}\right)^{\frac{p(1-r)}{p-r}}.\label{R4}
	\end{align}

\section{Proof of the main result}\label{S:proof}

\begin{proof}[Proof of Theorem~\ref{T:main}]
	We start by fixing a covering sequence $\{x_k\}_{k=N}^M \in CS(\varphi, U^p, a)$ with any $a>108$ (for example, $a=109$). In the entire proof, equivalence constants depend only on $p,q,r$ (and on the completely immaterial choice of $a>108$, see Remark~\ref{rem:value_of_a_is_immaterial}). When proving a desired upper/lower bound on $C$, we always implicitly assume that the quantity on the right/left-hand side is finite.
	
	\rm{(i)} By Theorem~\ref{thm:main_discretization}, we have
	\begin{equation*}
		C\approx C_{1,1}+C_{1,2}+C_{3,1}+C_{4,1}.
	\end{equation*}
	
	First, interchanging the order of suprema gives,
	\begin{align}
		C_{1,1} & = \sup_{N+1\leq k \leq M} \esup_{s\in(x_{k-1}, x_k)}   \left( \int_{x_{k-1}}^s \delta \right)^\frac1{r} \varphi(s)^{-\frac1{p}}  \left( \int_s^{x_k} \Delta^{-\frac{q}{r}} w \right)^\frac1{q} \notag \\
		&\leq
		\sup_{N+1\leq k\leq M}\esup_{s\in(x_{k-1},x_k)} \Delta(s)^{\frac{1}{r}}\varphi(s)^{-\frac{1}{p}} \left(\int_s^L \Delta^{-\frac{q}{r}}w\right)^\frac{1}{q}
		=
		B_2. \label{EQ:C11lesB2}
	\end{align}
	
	Also, it is easy to see that
	\begin{align}
		C_{1,2}
		&\leq
		\sup_{N+1\leq k\leq M}\esup_{t\in(x_{k-1},x_k)}\left(\int_{x_{k-1}}^t w\right)^\frac{1}{q}\varphi(t)^{-\frac{1}{p}}
		\leq
		B_1, \label{EQ:C12lesB1}\\
		C_{3,1}&\leq  B_2, \label{EQ:C31lesB2}\\
		C_{4,1}&\leq B_1. \label{EQ:C41lesB1}
	\end{align}
	Hence
	\begin{align}\label{EQ:anti1up}
		C_{1,1}+C_{1,2}+C_{3,1}+C_{4,1}\lesssim B_1+B_2.
	\end{align}
	As for the opposite inequality,
	\begin{align*}
		B_1 &=
		\sup_{N+1\leq k\leq M}\esup_{t\in(x_{k-1},x_k)}W(t)^{\frac1q}\varphi(t)^{-\frac1p} \\
		&\approx
		\esup_{t\in(0,x_{N+1})} W(t)^{\frac1q}\varphi(t)^{-\frac1p} +
		\sup_{N+2\leq k\leq M} \esup_{t\in(x_{k-1},x_k)} W(t)^{\frac1q}\varphi(t)^{-\frac1p}.
	\end{align*}
	Note that $W(t)=\int_0^{x_{k-1}}w+\int_{x_{k-1}}^t w $, for $t\in (x_{k-1},x_k)$, $N+2\leq k\leq M$. Then  
	\begin{align*}
		B_1 &\approx
		\sup_{N+2\leq k\leq M} W(x_{k-1})^{\frac1q} \varphi(x_{k-1})^{-\frac1p} +\sup_{N+1\leq k\leq M} \esup_{t\in(x_{k-1},x_k)} \bigg(\int_{x_{k-1}}^t w\bigg)^{\frac1q}  \varphi(t)^{-\frac1p}.
	\end{align*}
	Thus first reindexing $(k-1)\mapsto k$ in the first term, next using \eqref{EQ:strongly_decreasing_sup_sum},  we obtain
	\begin{align}
		B_1 &  = \sup_{N+1\leq k\leq M-1} \left(\sum_{i=N+1}^k \int_{x_{i-1}}^{x_i} w\right)^{\frac1q} \varphi(x_{k-1})^{-\frac1p}  \nonumber\\
		& \hspace{2cm} +\sup_{N+1\leq k\leq M} \esup_{t\in(x_{k-1},x_k)}  \bigg(\int_{x_{k-1}}^t w\bigg)^{\frac1q}  \varphi(t)^{-\frac1p} \nonumber\\
		&\approx
		C_{4,1}+\sup_{N+1\leq k\leq M}\esup_{t\in(x_{k-1},x_k)}\left(\int_{x_{k-1}}^t w\right)^{\frac1q}\varphi(t)^{-\frac1p}. \label{EQ:C41+sthlesB1}
	\end{align}

	Next,
	\begin{align*}
		B_1
		&\approx
		C_{4,1}+\sup_{N+1\leq k\leq M}\esup_{t\in(x_{k-1},x_k)}\left(\int_{x_{k-1}}^t \Delta^{-\frac{q}{r}}w\Delta^{\frac{q}{r}}\right)^{\frac1q}\varphi(t)^{-\frac1p}\\
		&  \approx C_{4,1}+  \esup_{t\in(x_N,x_{N+1})}\left(\int_{x_N}^t \Delta^{-\frac{q}{r}}w\Delta^{\frac{q}{r}}\right)^{\frac1q}\varphi(t)^{-\frac1p}\\
		&\hspace{2cm} +\sup_{N+2\leq k\leq M}\esup_{t\in(x_{k-1},x_k)}\left(\int_{x_{k-1}}^t \Delta^{-\frac{q}{r}}w\Delta^{\frac{q}{r}}\right)^{\frac1q}\varphi(t)^{-\frac1p}.
	\end{align*}
	
	Moreover, since  \begin{equation}\label{int-cut-D}
		\Delta(t)=\Delta(x_{k-1})+\int_{x_{k-1}}^t \delta, \,\, t\in (x_{k-1},x_k), \, N+2\leq k\leq M,
	\end{equation}
	we have 
	\begin{align*}
		B_1 
		&\approx
		C_{4,1}
		+
		\sup_{N+1\leq k\leq M}\esup_{t\in(x_{k-1},x_k)}\left(\int_{x_{k-1}}^t \Delta (s)^{-\frac{q}{r}}w(s)\left(\int_{x_{k-1}}^s\delta\right)^{\frac{q}{r}}\,ds\right)^{\frac1q}\varphi(t)^{-\frac1p}
		\nonumber\\
		&\qquad +
		\sup_{N+2\leq k\leq M}\Delta(x_{k-1})^{\frac{1}{r}}\esup_{t\in(x_{k-1},x_k)}\left(\int_{x_{k-1}}^t \Delta^{-\frac{q}{r}}w\right)^{\frac1q}\varphi(t)^{-\frac1p}\\
		&\leq
		C_{4,1}
		+
		C_{1,2}
		+
		\sup_{N+2\leq k\leq M}\Delta(x_{k-1})^{\frac{1}{r}}\left(\int_{x_{k-1}}^L \Delta^{-\frac{q}{r}}w\right)^{\frac1q}\varphi(x_{k-1})^{-\frac1p}.
		\nonumber
	\end{align*}
	On the other hand, reindexing $k \to k+1$ and using 
	\eqref{EQ:strongly_decreasing_sup_sum} again (note that the sequence $\{\left(\int_{x_{k}}^L \Delta^{-\frac{q}{r}}w\right)^{\frac1q}\varphi(x_{k})^{-\frac1p}\}_{k=N+1}^{M-1}$ is strongly decreasing)
	gives,
	\begin{align}\label{C-31-estimate-new}
		& \sup_{N+2\leq k\leq M} \Delta(x_{k-1})^{\frac1r}\varphi(x_{k-1})^{-\frac1p}\left(\int_{x_{k-1}}^L\Delta^{-\frac{q}{r}}w\right)^{\frac{1}{q}} \nonumber\\
		& \qquad \leq 
		\sup_{N+1\leq k\leq M-1} \left(\int_{x_{k}}^L\Delta^{-\frac{q}{r}}w\right)^{\frac{1}{q}}\varphi(x_{k})^{-\frac1p} \left(\sum_{i=N+1}^k\int_{x_{i-1}}^{x_i}\delta\right)^{\frac1r}\nonumber
		\\
		&\qquad  \approx \sup_{N+1\leq k\leq M-1} \left(\int_{x_{k}}^L\Delta^{-\frac{q}{r}}w\right)^{\frac{1}{q}}\varphi(x_{k})^{-\frac1p} \left(\int_{x_{k-1}}^{x_k}\delta\right)^{\frac1r}\nonumber \\
		&\qquad \leq C_{3,1}.
	\end{align}
	Then plugging this in $B_1$, we get
	\begin{align}\label{EQ:antid_i_B1}
		B_1  &\lesssim
		C_{4,1}
		+
		C_{1,2}
		+C_{3,1}.
	\end{align}

	On the other hand, in view of \eqref{int-cut-D}, we have
	\begin{align*}
		B_2 &= \sup_{N+1\leq k\leq M} \esup_{t\in(x_{k-1},x_k)} \Delta(t)^{\frac1r} \varphi(t)^{-\frac1p} \left(\int_t^L\Delta^{-\frac{q}{r}} w \right)^{\frac{1}{q}}
		\\
		&\approx
		\esup_{t\in(x_N,x_{N+1})}  \bigg(\int_{x_N}^t \delta\bigg)^{\frac1r} \varphi(t)^{-\frac1p}\left(\int_t^L\Delta^{-\frac{q}{r}}w\right)^{\frac{1}{q}}  \\
		& \qquad +
		\sup_{N+2\leq k\leq M}\esup_{t\in(x_{k-1},x_k)} \Delta(t)^{\frac1r}\varphi(t)^{-\frac1p}\left(\int_t^L\Delta^{-\frac{q}{r}}w\right)^{\frac{1}{q}}
		\\
		&\approx
		\sup_{N+1\leq k\leq M} \esup_{t\in(x_{k-1},x_k)} \left(\int_{x_{k-1}}^t\delta\right)^{\frac1r}\varphi(t)^{-\frac1p}\left(\int_t^{L}\Delta^{-\frac{q}{r}}w\right)^{\frac{1}{q}}
		\\
		&\qquad+
		\sup_{N+2\leq k\leq M} \Delta(x_{k-1})^{\frac1r}\varphi(x_{k-1})^{-\frac1p}\left(\int_{x_{k-1}}^L\Delta^{-\frac{q}{r}}w\right)^{\frac{1}{q}}
		\\
		&\approx
		\esup_{t\in(x_{M-1}, x_M)} \left(\int_{x_{M-1}}^t\delta\right)^{\frac1r}\varphi(t)^{-\frac1p}\left(\int_t^{x_M}\Delta^{-\frac{q}{r}}w\right)^{\frac{1}{q}}
		\\
		&\qquad +\sup_{N+1\leq k\leq M-1} \esup_{t\in(x_{k-1},x_k)} \left(\int_{x_{k-1}}^t\delta\right)^{\frac1r}\varphi(t)^{-\frac1p}\left(\int_t^{L}\Delta^{-\frac{q}{r}}w\right)^{\frac{1}{q}}
		\\
		&\qquad+
		\sup_{N+2\leq k\leq M} \Delta(x_{k-1})^{\frac1r}\varphi(x_{k-1})^{-\frac1p}\left(\int_{x_{k-1}}^L\Delta^{-\frac{q}{r}}w\right)^{\frac{1}{q}}.
	\end{align*}
	Furthermore, as
	\begin{equation}\label{int-cut-Dw}
		\int_t^L \Delta^{-\frac{q}{r}}w =\int_t^{x_{k}} \Delta^{-\frac{q}{r}}w + \int_{x_{k}}^L \Delta^{-\frac{q}{r}}w, \, \, t\in (x_{k-1},x_k),\, N+1\leq k\leq M-1,
	\end{equation}
	we have
	\begin{align}\label{EQ:antid_i_B2-1}
		B_2
		&\approx
		\sup_{N+1\leq k\leq M}\esup_{t\in(x_{k-1},x_k)} \left(\int_{x_{ k-1}}^t\delta\right)^{\frac1r}\varphi(t)^{-\frac1p}\left(\int_t^{x_k}\Delta^{-\frac{q}{r}}w\right)^{\frac{1}{q}}
		\nonumber\\
		&\qquad+
		\sup_{N+1\leq k\leq M-1}\left(\int_{x_k}^L\Delta^{-\frac{q}{r}}w\right)^{\frac{1}{q}}\esup_{t\in(x_{k-1},x_k)} \left(\int_{x_{k-1}}^t\delta\right)^{\frac1r}\varphi(t)^{-\frac1p}
		\nonumber\\
		&\qquad+
		\sup_{N+2\leq k\leq M} \Delta(x_{k-1})^{\frac1r}\varphi(x_{k-1})^{-\frac1p}\left(\int_{x_{k-1}}^L\Delta^{-\frac{q}{r}}w\right)^{\frac{1}{q}}
		\nonumber\\
		&\leq
		C_{1,1}+C_{3,1}+
		\sup_{N+2\leq k\leq M} \Delta(x_{k-1})^{\frac1r}\varphi(x_{k-1})^{-\frac1p}\left(\int_{x_{k-1}}^L\Delta^{-\frac{q}{r}}w\right)^{\frac{1}{q}}.
	\end{align}
	Therefore, plugging \eqref{C-31-estimate-new} into  \eqref{EQ:antid_i_B2-1}, we have
	\begin{equation}\label{EQ:antid_i_B2}
		B_2\lesssim C_{1,1} + C_{3,1}.
	\end{equation}
	
	Thus arrive at
	\begin{equation}\label{EQ:B1+B2UpperBoundCase1}
		B_1+B_2\lesssim C_{1,1} + C_{1,2} + C_{3,1} + C_{4,1},
	\end{equation}
	which together with \eqref{EQ:anti1up} gives
	\begin{align*}
		C\approx B_1+B_2.
	\end{align*}
	
	\rm{(ii)} By Theorem~\ref{thm:main_discretization}, we have
	\begin{equation*}
		C\approx C_{1,2}+C_{1,3}+C_{3,2}+C_{4,1}.
	\end{equation*}
	We start by establishing the desired upper estimate on $C$. In view of \eqref{EQ:C12lesB1} and \eqref{EQ:C41lesB1}, we only need to prove suitable upper estimates on $C_{1,3}$ and $C_{3,2}$.
	
	Observe that, since $\varphi\in Q_{U^p}$, $\{x_k\}_{k=N}^{M}\in CS(\varphi,U^p,a)$, by Lemma~\ref{Lem:Decomp}, the index set $\K^+=\{k\in\Z\colon N+1\leq k \leq M\}$ can be decomposed into $\K^+ = \mathcal{Z}_1\cup\mathcal{Z}_2$, where $\mathcal{Z}_1\cap\mathcal{Z}_1=\emptyset$, in such a way that 
	\begin{align}\label{phi-Z1}
		\varphi(t) \approx \varphi(x_k) \quad&\text{for all $t\in[x_{k-1},x_k]$ and every $k\in\mathcal{Z}_1$},
	\end{align}
	and
	\begin{align}\label{Uphi-Z2}
		\frac{U(t)}{\varphi(t)^{\frac{1}{p}}}\approx \frac{U(x_k)}{\varphi(x_k)^{\frac{1}{p}}}\quad&\text{for all $t\in[x_{k-1},x_k]$ and every $k\in\mathcal{Z}_2$}.
	\end{align}
	Then, for $k\in\mathcal{Z}_1$ and $t\in(x_{k-1},x_{k}]$
	\begin{align}\label{EQ:antid_ii_1}
		\left( \int_{x_{k-1}}^t \left( \int_{x_{k-1}}^{s} \delta \right)^\frac{r}{1-r} \delta(s)\varphi(s)^{-\frac{r}{p(1-r)}}\, ds \right)^{\frac{1-r}{r}}& \nonumber\\
		&\hspace{-2.5cm}\approx
		\varphi(x_k)^{-\frac{1}{p}}\left( \int_{x_{k-1}}^t \left( \int_{x_{k-1}}^{s} \delta \right)^\frac{r}{1-r} \delta(s)\, ds \right)^{\frac{1-r}{r}}\nonumber\\
		&\hspace{-2.5cm}\leq
		\varphi(x_k)^{-\frac{1}{p}}\Delta(t)^{\frac1r},
	\end{align}
	and  for $k\in\mathcal{Z}_2$ and $t\in(x_{k-1},x_{k}]$
	\begin{align}
		&\left( \int_{x_{k-1}}^t \left( \int_{x_{k-1}}^{s} \delta \right)^\frac{r}{1-r} \delta(s)\varphi(s)^{-\frac{r}{p(1-r)}}\, ds \right)^{\frac{1-r}{r}}\nonumber\\
		&\qquad\qquad\approx
		U(x_{k-1})\varphi(x_{k-1})^{-\frac{1}{p}}\left( \int_{x_{k-1}}^t \left( \int_{x_{k-1}}^{s} \delta \right)^\frac{r}{1-r} \delta(s)U(s)^{-\frac{r}{1-r}}\, ds \right)^{\frac{1-r}{r}}\label{nofor6}\\
		&\qquad\qquad\leq
		\esup_{\tau\in(x_{k-1},t)}U(\tau)\varphi(\tau)^{-\frac{1}{p}}\left( \int_{\tau}^t \Delta^\frac{r}{1-r} \delta U^{-\frac{r}{1-r}} \right)^{\frac{1-r} {r}}.\label{EQ:antid_ii_2}
	\end{align}
	Consequently, using \eqref{EQ:antid_ii_1} and \eqref{EQ:antid_ii_2}, we have
	\begin{align*}
		C_{1,3}
		&\lesssim
		\sup_{\substack{ k\in\mathcal{Z}_1 \\ N+1\le k\le M}} \esup_{t\in(x_{k-1}, x_k)} \left( \int_t^{x_k} \Delta^{-\frac{q}{r}} w \right)^\frac1{q}\varphi(x_k)^{-\frac1p}\Delta(t)^{\frac1r}\\
		&\quad+
		\sup_{\substack{ k\in\mathcal{Z}_2 \\ N+1\le k\le M}}\esup_{t\in(x_{k-1}, x_k)} \left( \int_t^{x_k} \Delta^{-\frac{q}{r}} w \right)^\frac1{q} \\
		& \hspace{3cm} \times \esup_{\tau\in(x_{k-},t)}U(\tau)\varphi(\tau)^{-\frac{1}{p}}\left( \int_{\tau}^t \Delta^\frac{r}{1-r} \delta U^{-\frac{r}{1-r}} \right)^{\frac{1-r}{r}} \\
		&\leq \sup_{\substack{ k\in\mathcal{Z}_1 \\ N+1\le k\le M}}\left( \int_{x_{k-1}}^{x_k}  w \right)^\frac1{q}\varphi(x_k)^{-\frac1p}\\
		&\quad+
		\sup_{\substack{ k\in\mathcal{Z}_2 \\ N+1\le k\le M}}\esup_{t\in(x_{k-1}, x_k)} \left( \int_t^{x_k} \Delta^{-\frac{q}{r}} w \right)^\frac1{q} \\
		& \hspace{3cm} \times \esup_{\tau\in(x_{k-1},t)}U(\tau)\varphi(\tau)^{-\frac{1}{p}}\left( \int_{\tau}^t \Delta^\frac{r}{1-r} \delta U^{-\frac{r}{1-r}} \right)^{\frac{1-r}{r}},
	\end{align*}
	where we used the monotonicity of $\Delta$ for the last inequality. Thus
	\begin{align}
		C_{1,3}
		&\lesssim
		\sup_{\substack{ k\in\mathcal{Z}_1 \\ N+1\le k\le M}} W(x_k)^{\frac1q} \varphi(x_k)^{-\frac1p} \nonumber\\
		&\quad+
		\sup_{\substack{ k\in\mathcal{Z}_2 \\ N+1\le k\le M}}\esup_{t\in(x_{k-1}, x_k)} \left( \int_t^{L} \Delta^{-\frac{q}{r}} w \right)^\frac1{q} \nonumber\\
		& \hspace{3cm}\esup_{\tau\in(0,t)}U(\tau)\varphi(\tau)^{-\frac{1}{p}}\left( \int_{\tau}^t \Delta^\frac{r}{1-r} \delta U^{-\frac{r}{1-r}}\right)^{\frac{1-r}{r}} \nonumber\\
		&\leq
		B_1+B_3. \label{EQ:C13lessB1+B3}
	\end{align}
	Next, since $U\varphi^{-\frac1p}\in Q_{U}(0,L)$, \eqref{KMT-Lemma3.5} applied to $p\mapsto \frac{1-r}{r}$, $\varrho\mapsto U$, ${ \widetilde{\varphi}}\mapsto U\varphi^{-\frac1p}$ and $g\mapsto \Delta^{\frac{r}{1-r}}\delta$, in which the symbols on the left-hand sides refer to those in \eqref{KMT-Lemma3.5}, gives us
	\begin{align}\label{EQ:antid_ii_4}
		&\esup_{s\in(0,x_k)}U(s)\varphi(s)^{-\frac1p}\left(\int_0^{x_k}\frac{\Delta(\tau)^{\frac{r}{1-r}}\delta(\tau)}{U(s)^\frac{r}{1-r}+U(\tau)^\frac{r}{1-r}}\,d\tau \right)^{\frac{1-r}{r}}\nonumber\\
		&\qquad\approx
		\sup_{N\leq i\leq k}U(x_i)\varphi(x_i)^{-\frac1p}\left(\int_0^{x_k}\frac{\Delta(\tau)^{\frac{r}{1-r}}\delta(\tau)}{U(x_i)^\frac{r}{1-r}+U(\tau)^\frac{r}{1-r}}\, d\tau \right)^{\frac{1-r}{r}}\nonumber\\
		&\qquad\approx
		\sup_{N+1\leq i\leq k}\left(\int_{x_{i-1}}^{x_i}\Delta^{\frac{r}{1-r}}\delta \varphi^{-\frac{r}{p(1-r)}} \right)^{\frac{1-r}{r}}.
	\end{align}
	Consequently
	\begin{align}
		C_{3,2}
		&\leq
		\sup_{N+1\leq k \leq M-1} \bigg(\int_{x_k}^L \Delta^{-\frac{q}{r}} w \bigg)^{\frac{1}{q}}  \sup_{N+1\leq i\leq k}\bigg(\int_{x_{i-1}}^{x_i} \Delta^{\frac{r}{1-r}} \delta \varphi^{-\frac{r}{p(1-r)}}\bigg)^{\frac{1-r}{r}} \nonumber\\
		&\approx
		\sup_{N+1\leq k \leq M-1} \bigg(\int_{x_k}^L \Delta^{-\frac{q}{r}} w \bigg)^{\frac{1}{q}}\nonumber\\
		& \hspace{2cm} \times \esup_{s\in(0,x_k)}U(s)\varphi(s)^{-\frac1p}\left(\int_0^{x_k}\frac{\Delta(\tau)^{\frac{r}{1-r}}\delta(\tau)}{U(s)^\frac{r}{1-r}+U(\tau)^\frac{r}{1-r}}\,d\tau \right)^{\frac{1-r}{r}} \nonumber\\
		&\leq
		\sup_{N+1\leq k \leq M-1} \esup_{t\in(x_{k-1},x_k)} \bigg(\int_{t}^L \Delta^{-\frac{q}{r}} w \bigg)^{\frac{1}{q}}\nonumber\\
		& \hspace{2cm} \times \esup_{s\in(0,t)}U(s)\varphi(s)^{-\frac1p}\left(\int_0^{t}\frac{\Delta(\tau)^{\frac{r}{1-r}}\delta(\tau)}{U(s)^\frac{r}{1-r}+U(\tau)^\frac{r}{1-r}}\,d\tau \right)^{\frac{1-r}{r}} \nonumber\\
		&\approx \sup_{N+1\leq k \leq M-1} \esup_{t\in(x_{k-1},x_k)} \bigg(\int_{t}^L \Delta^{-\frac{q}{r}} w \bigg)^{\frac{1}{q}}\esup_{s\in(0,t)}\varphi(s)^{-\frac1p}\left(\int_0^{s}\Delta^{\frac{r}{1-r}}\delta \right)^{\frac{1-r}{r}}
		\nonumber\\
		&\qquad \quad + \sup_{N+1\leq k \leq M-1} \esup_{t\in(x_{k-1},x_k)} \bigg(\int_{t}^L \Delta^{-\frac{q}{r}} w \bigg)^{\frac{1}{q}} \nonumber\\ & \hspace{3cm}\times\esup_{s\in(0,t)}U(s)\varphi(s)^{-\frac1p}\left(\int_s^t \Delta^{\frac{r}{1-r}}\delta U^{-\frac{r}{1-r}} \right)^{\frac{1-r}{r}} \nonumber
	\end{align}
	where, we applied \eqref{EQ:antid_ii_4} for the first equivalence and \eqref{min-equiv} for the second one. Thus, monotonicity of $\Delta$ yields,
	\begin{align}
		C_{3,2}
		&\lesssim \sup_{N+1\leq k \leq M-1} \esup_{t\in(x_{k-1},x_k)} \bigg(\int_{t}^L \Delta^{-\frac{q}{r}} w \bigg)^{\frac{1}{q}}\esup_{s\in(0,t)}\varphi(s)^{-\frac1p} \Delta(s)^{\frac{1}{r}}
		\nonumber\\
		&\qquad \quad + \sup_{N+1\leq k \leq M-1} \esup_{t\in(x_{k-1},x_k)} \bigg(\int_{t}^L \Delta^{-\frac{q}{r}} w \bigg)^{\frac{1}{q}} \nonumber\\ & \hspace{3cm}\times\esup_{s\in(0,t)}U(s)\varphi(s)^{-\frac1p}\left(\int_s^t \Delta^{\frac{r}{1-r}}\delta U^{-\frac{r}{1-r}} \right)^{\frac{1-r}{r}}    \nonumber\\
		&\approx B_2+B_3, \label{EQ:C32lessB2+B3}
	\end{align}
	where, we interchanged the order of the suprema in the first term to obtain the equivalence.
	
	Altogether, we have
	\begin{align}\label{EQ:antid_ii_B123}
		C_{1,2} + C_{1,3} + C_{3,2} + C_{4,1}\lesssim B_1+B_2+B_3.
	\end{align}
	
	On the other hand, we shall prove that $B_1+B_2+B_3 \lesssim C_{1,2}+C_{1,3}+C_{3,2}+C_{4,1}$. First, applying \eqref{sup-int}, it is clear that
	\begin{align}\label{EQ:C11lessC13}
		C_{1,1}  &= \sup_{N+1\leq k \leq M} \esup_{t\in(x_{k-1}, x_k)} \left( \int_t^{x_k} \Delta^{-\frac{q}{r}} w \right)^\frac1{q} \esup_{s\in(x_{k-1}, t)} \left( \int_{x_{k-1}}^s \delta \right)^\frac1{r} \varphi(s)^{-\frac1{p}} \nonumber\\
		& \lesssim  \sup_{N+1\leq k \leq M} \esup_{t\in(x_{k-1}, x_k)} \left( \int_t^{x_k} \Delta^{-\frac{q}{r}} w \right)^\frac1{q}\nonumber \\
		& \hspace{3cm} \times\bigg(\int_{x_{k-1}}^t \bigg(\int_{x_{k-1}}^{\tau} \delta\bigg)^{\frac{r}{1-r}} \delta(\tau) \varphi(\tau)^{-\frac{r}{p(1-r)}} \,d\tau \bigg)^\frac{1-r}{r} \nonumber \\
		& = C_{1,3}.
	\end{align}
	Therefore, using \eqref{EQ:B1+B2UpperBoundCase1} together with \eqref{EQ:C11lessC13} and \eqref{EQ:C31lessC32}, we arrive at
	\begin{align}\label{EQ:antid_ii_B12}
		B_1+B_2\lesssim C_{1,3}+C_{1,2}+C_{3,2}+C_{4,1}.
	\end{align}
	
	It remains to find a suitable upper estimate for $B_3$.  Observe first that
	\begin{align}\label{sup_cut}
		\esup_{s\in(0,t)} U(s)\varphi(s)^{-\frac1p} &\left(\int_s^t\Delta^{\frac{r}{1-r}}\delta U^{-\frac{r}{1-r}}\right)^{\frac{1-r}{r}} \nonumber\\
		&\quad \approx \esup_{s\in(0,x_{k-1})}U(s)\varphi(s)^{-\frac1p}\left(\int_s^t\Delta^{\frac{r}{1-r}}\delta U^{-\frac{r}{1-r}}\right)^{\frac{1-r}{r}}\nonumber\\
		&\qquad+\esup_{s\in(x_{k-1},t)}U(s)\varphi(s)^{-\frac1p}\left(\int_s^t\Delta^{\frac{r}{1-r}}\delta U^{-\frac{r}{1-r}}\right)^{\frac{1-r}{r}}
	\end{align}
	for $t\in (x_{k-1}, x_k)$, $N+2 \leq k \leq M$.
	
	Applying \eqref{int-cut-Dw}, we have
	\begin{align*}
		B_3 &=
		\sup_{N+1\leq k\leq M}\esup_{t\in(x_{k-1},x_k)} \left(\int_t^L\Delta^{-\frac{q}{r}}w\right)^{\frac{1}{q}} \esup_{s\in(0,t)}U(s)\varphi(s)^{-\frac1p}\left(\int_s^t\Delta^{\frac{r}{1-r}}\delta U^{-\frac{r}{1-r}}\right)^{\frac{1-r}{r}}\\
		&  \approx  \sup_{N+1\leq k\leq M-1}\esup_{t\in(x_{k-1},x_k)} \left(\int_t^L\Delta^{-\frac{q}{r}}w\right)^{\frac{1}{q}} \esup_{s\in(0,t)}U(s)\varphi(s)^{-\frac1p}\left(\int_s^t\Delta^{\frac{r}{1-r}}\delta U^{-\frac{r}{1-r}}\right)^{\frac{1-r}{r}}\\
		&  + \esup_{t\in(x_{M-1},x_M)} \left(\int_t^L\Delta^{-\frac{q}{r}}w\right)^{\frac{1}{q}} \esup_{s\in(0,t)}U(s)\varphi(s)^{-\frac1p}\left(\int_s^t\Delta^{\frac{r}{1-r}}\delta U^{-\frac{r}{1-r}}\right)^{\frac{1-r}{r}}\\
		&  \approx  \sup_{N+1\leq k\leq M-1} \esup_{t\in(x_{k-1},x_k)} \left(\int_t^{x_k}\Delta^{-\frac{q}{r}}w\right)^{\frac{1}{q}} \esup_{s \in (0,t)} U(s) \varphi(s)^{-\frac1p} \left(\int_s^t\Delta^{\frac{r}{1-r}}\delta U^{-\frac{r}{1-r}}\right)^{\frac{1-r}{r}}\\
		&  +\sup_{N+1\leq k\leq M-1} \left(\int_{x_k}^L \Delta^{-\frac{q}{r}}w\right)^{\frac{1}{q}} \esup_{s\in(0,x_k)}U(s)\varphi(s)^{-\frac1p}\left(\int_s^{x_k}\Delta^{\frac{r}{1-r}}\delta U^{-\frac{r}{1-r}}\right)^{\frac{1-r}{r}}\\
		&  + \esup_{t\in(x_{M-1},x_M)} \left(\int_t^{x_M}\Delta^{-\frac{q}{r}}w\right)^{\frac{1}{q}} \esup_{s \in (0,t)} U(s)\varphi(s)^{-\frac1p}\left(\int_s^t\Delta^{\frac{r}{1-r}}\delta U^{-\frac{r}{1-r}}\right)^{\frac{1-r}{r}}\\
		& \approx  \sup_{N+1\leq k\leq M} \esup_{t\in(x_{k-1},x_k)} \left(\int_t^{x_k}\Delta^{-\frac{q}{r}}w\right)^{\frac{1}{q}} \esup_{s\in(0,t)} U(s)\varphi(s)^{-\frac1p}\left(\int_s^t\Delta^{\frac{r}{1-r}}\delta U^{-\frac{r}{1-r}}\right)^{\frac{1-r}{r}}\\
		&  +\sup_{N+1\leq k\leq M-1} \left(\int_{x_k}^L \Delta^{-\frac{q}{r}}w\right)^{\frac{1}{q}} \esup_{s\in(0,x_k)}U(s)\varphi(s)^{-\frac1p}\left(\int_s^t\Delta^{\frac{r}{1-r}}\delta U^{-\frac{r}{1-r}}\right)^{\frac{1-r}{r}}.
	\end{align*}
	Now, using \eqref{sup_cut}, we obtain
	\begin{align*}
		B_3 &  \approx \esup_{t\in(x_{N},x_{N+1})} \left(\int_t^{x_{N+1}}\Delta^{-\frac{q}{r}}w\right)^{\frac{1}{q}} \esup_{s\in(x_N,t)}U(s)\varphi(s)^{-\frac1p}\left(\int_s^t\Delta^{\frac{r}{1-r}}\delta U^{-\frac{r}{1-r}}\right)^{\frac{1-r}{r}}\\
		&  + \sup_{N+2\leq k\leq M} \esup_{t\in(x_{k-1},x_k)} \left(\int_t^{x_k}\Delta^{-\frac{q}{r}}w\right)^{\frac{1}{q}} \esup_{s\in(0,t)} U(s)\varphi(s)^{-\frac1p}\left(\int_s^t\Delta^{\frac{r}{1-r}}\delta U^{-\frac{r}{1-r}}\right)^{\frac{1-r}{r}}\\
		&  +\sup_{N+1\leq k\leq M-1} \left(\int_{x_k}^L \Delta^{-\frac{q}{r}}w\right)^{\frac{1}{q}} \esup_{s\in(0,x_k)}U(s)\varphi(s)^{-\frac1p}\left(\int_s^t\Delta^{\frac{r}{1-r}}\delta U^{-\frac{r}{1-r}}\right)^{\frac{1-r}{r}}\\
		&  \approx \esup_{t\in(x_{N},x_{N+1})} \left(\int_t^{x_{N+1}}\Delta^{-\frac{q}{r}}w\right)^{\frac{1}{q}} \esup_{s\in(x_N,t)}U(s)\varphi(s)^{-\frac1p}\left(\int_s^t\Delta^{\frac{r}{1-r}}\delta U^{-\frac{r}{1-r}}\right)^{\frac{1-r}{r}}\\
		&  + \sup_{N+2\leq k\leq M} \esup_{t\in(x_{k-1},x_k)} \left(\int_t^{x_k}\Delta^{-\frac{q}{r}}w\right)^{\frac{1}{q}} \esup_{s\in(0,x_{k-1})} U(s)\varphi(s)^{-\frac1p}\left(\int_s^t\Delta^{\frac{r}{1-r}}\delta U^{-\frac{r}{1-r}}\right)^{\frac{1-r}{r}}\\
		&  + \sup_{N+2\leq k\leq M} \esup_{t\in(x_{k-1},x_k)} \left(\int_t^{x_k}\Delta^{-\frac{q}{r}}w\right)^{\frac{1}{q}} \esup_{s\in(x_{k-1},t)} U(s) \varphi(s)^{-\frac1p}\left(\int_s^t\Delta^{\frac{r}{1-r}}\delta U^{-\frac{r}{1-r}}\right)^{\frac{1-r}{r}}\\
		& +\sup_{N+1\leq k\leq M-1} \left(\int_{x_k}^L \Delta^{-\frac{q}{r}}w\right)^{\frac{1}{q}} \esup_{s\in(0,x_k)}U(s)\varphi(s)^{-\frac1p}\left(\int_s^t\Delta^{\frac{r}{1-r}}\delta U^{-\frac{r}{1-r}}\right)^{\frac{1-r}{r}}\\
		&\approx
		\sup_{N+2\leq k\leq M}\esup_{t\in(x_{k-1},x_k)} \left(\int_t^{x_k}\Delta^{-\frac{q}{r}}w\right)^{\frac{1}{q}} \esup_{s\in(0,x_{k-1})}U(s)\varphi(s)^{-\frac1p}\left(\int_s^t\Delta^{\frac{r}{1-r}}\delta U^{-\frac{r}{1-r}}\right)^{\frac{1-r}{r}}\\
		& +
		\sup_{N+1\leq k\leq M}\esup_{t\in(x_{k-1},x_k)} \left(\int_t^{x_{k}}\Delta^{-\frac{q}{r}}w\right)^{\frac{1}{q}} \esup_{s\in(x_{k-1},t)}U(s)\varphi(s)^{-\frac1p}\left(\int_s^t\Delta^{\frac{r}{1-r}}\delta U^{-\frac{r}{1-r}}\right)^{\frac{1-r}{r}}\\
		& +
		\sup_{N+1\leq k\leq M-1}\left(\int_{x_k}^L\Delta^{-\frac{q}{r}}w\right)^{\frac{1}{q}} \esup_{s\in(0,x_k)}U(s)\varphi(s)^{-\frac1p}\left(\int_s^{x_k}\Delta^{\frac{r}{1-r}}\delta U^{-\frac{r}{1-r}}\right)^{\frac{1-r}{r}}.
	\end{align*}
	Next, decomposing the integral $\int_s^t$ into the sum $\int_s^{x_{k-1}}+ \int_{x_{k-1}}^t$ and using the monotonicities of $\int_t^{x_k} \Delta^{-\frac{q}{r}}w$ and $U(t)\varphi(t)^{-\frac1p}$, we have
	\begin{align*}
		B_3 
		&\approx
		\sup_{N+2\leq k\leq M}  \bigg(\int_{x_{k-1}}^{x_k}\Delta^{-\frac{q}{r}}w\bigg)^{\frac{1}{q}} \esup_{s\in(0,x_{k-1})}U(s)\varphi(s)^{-\frac1p}\bigg(\int_s^{x_{k-1}}\Delta^{\frac{r}{1-r}}\delta U^{-\frac{r}{1-r}}\bigg)^{\frac{1-r}{r}}\\
		&\hspace{-0.1cm}+ \sup_{N+2\leq k\leq M}\esup_{t\in(x_{k-1},x_k)} \bigg(\int_t^{x_k}\Delta^{-\frac{q}{r}}w\bigg)^{\frac{1}{q}} U(x_{k-1})\varphi(x_{k-1})^{-\frac1p}\bigg(\int_{x_{k-1}}^t\Delta^{\frac{r}{1-r}}\delta U^{-\frac{r}{1-r}}\bigg)^{\frac{1-r}{r}}\\
		&\hspace{-0.1cm} +
		\sup_{N+1\leq k\leq M}\esup_{t\in(x_{k-1},x_k)} \bigg(\int_t^{x_{k}}\Delta^{-\frac{q}{r}}w\bigg)^{\frac{1}{q}} \esup_{s\in(x_{k-1},t)}U(s)\varphi(s)^{-\frac1p}\bigg(\int_s^t\Delta^{\frac{r}{1-r}}\delta U^{-\frac{r}{1-r}}\bigg)^{\frac{1-r}{r}}\\
		&\hspace{-0.1cm} +
		\sup_{N+1\leq k\leq M-1}\bigg(\int_{x_k}^L\Delta^{-\frac{q}{r}}w\bigg)^{\frac{1}{q}} \esup_{s\in(0,x_k)}U(s)\varphi(s)^{-\frac1p}\bigg(\int_s^{x_k}\Delta^{\frac{r}{1-r}}\delta U^{-\frac{r}{1-r}}\bigg)^{\frac{1-r}{r}}.
	\end{align*}
	Since, when $t\in (x_{k-1}, x_k)$ 
	\begin{align*}
		U(x_{k-1})\varphi(x_{k-1})^{-\frac1p}\left(\int_{x_{k-1}}^t\Delta^{\frac{r}{1-r}}\delta U^{-\frac{r}{1-r}}\right)^{\frac{1-r}{r}}& \\
		& \hspace{-4cm}\leq
		\esup_{s\in(x_{k-1},t)}U(s)\varphi(s)^{-\frac1p}\left(\int_s^t\Delta^{\frac{r}{1-r}}\delta U^{-\frac{r}{1-r}}\right)^{\frac{1-r}{r}}
	\end{align*}
	holds,   
	\begin{align*}
		B_3 &  \lesssim
		\sup_{N+2\leq k\leq M}  \left(\int_{x_{k-1}}^{L}\Delta^{-\frac{q}{r}}w\right)^{\frac{1}{q}} \esup_{s\in(0,x_{k-1})}U(s)\varphi(s)^{-\frac1p}\left(\int_s^{x_{k-1}}\Delta^{\frac{r}{1-r}}\delta U^{-\frac{r}{1-r}}\right)^{\frac{1-r}{r}}\\
		&  +
		\sup_{N+1\leq k\leq M}\esup_{t\in(x_{k-1},x_k)} \left(\int_t^{x_{k}}\Delta^{-\frac{q}{r}}w\right)^{\frac{1}{q}} \esup_{s\in(x_{k-1},t)}U(s)\varphi(s)^{-\frac1p}\left(\int_s^t\Delta^{\frac{r}{1-r}}\delta U^{-\frac{r}{1-r}}\right)^{\frac{1-r}{r}}\\
		& + 
		\sup_{N+1\leq k\leq M-1}\left(\int_{x_k}^L\Delta^{-\frac{q}{r}}w\right)^{\frac{1}{q}} \esup_{s\in(0,x_k)}U(s)\varphi(s)^{-\frac1p}\left(\int_s^{x_k}\Delta^{\frac{r}{1-r}}\delta U^{-\frac{r}{1-r}}\right)^{\frac{1-r}{r}}.
	\end{align*}
	Reindexing $k\mapsto k+1$ in the first term, we get
	\begin{align*}
		B_3&\lesssim
		\sup_{N+1\leq k\leq M-1}\left(\int_{x_k}^L\Delta^{-\frac{q}{r}}w\right)^{\frac{1}{q}} \\
		& \hspace{2cm} \times \esup_{s\in(0,x_k)}U(s)\varphi(s)^{-\frac1p}\left(\int_s^{x_k}\Delta^{\frac{r}{1-r}}\delta U^{-\frac{r}{1-r}}\, d\tau\right)^{\frac{1-r}{r}}\\
		&  +
		\sup_{N+1\leq k\leq M}\esup_{t\in(x_{k-1},x_k)} \left(\int_t^{x_{k}}\Delta^{-\frac{q}{r}}w\right)^{\frac{1}{q}} \\
		& \hspace{2cm} \times \esup_{s\in(x_{k-1},t)}U(s)\varphi(s)^{-\frac1p}\left(\int_s^t\Delta^{\frac{r}{1-r}}\delta U^{-\frac{r}{1-r}}\right)^{\frac{1-r}{r}}.
	\end{align*}
	
	In view of \eqref{min-equiv}, we have
	\begin{align}\label{1}
		\esup_{s\in(0,x_k)}U(s)\varphi(s)^{-\frac1p} &\left(\int_s^{x_k}\Delta^{\frac{r}{1-r}}\delta U^{-\frac{r}{1-r}}\right)^{\frac{1-r}{r}}\nonumber\\
		&  \lesssim \esup_{s\in(0,x_k)}U(s)\varphi(s)^{-\frac1p}\left(\int_0^{x_k} \frac{\Delta(\tau)^{\frac{r}{1-r}}\delta(\tau)}{U(s)^{\frac{r}{1-r}}+U(\tau)^{\frac{r}{1-r}}}\,d\tau \right)^{\frac{1-r}{r}}.
	\end{align}
	Then,
	\begin{align*}
		B_3
		&\lesssim
		\sup_{N+1\leq k\leq M-1}\left(\int_{x_k}^L\Delta^{-\frac{q}{r}}w\right)^{\frac{1}{q}} \\
		& \hspace{2cm} \times \esup_{s\in(0,x_k)}U(s)\varphi(s)^{-\frac1p}\left(\int_0^{x_k} \frac{\Delta(\tau)^{\frac{r}{1-r}}\delta(\tau)}{U(s)^{\frac{r}{1-r}}+U(\tau)^{\frac{r}{1-r}}} \,d\tau \right)^{\frac{1-r}{r}}\\
		&\quad+
		\sup_{N+1\leq k\leq M}\esup_{t\in(x_{k-1},x_k)} \left(\int_t^{x_{k}}\Delta^{-\frac{q}{r}}w\right)^{\frac{1}{q}} \\
		& \hspace{2cm} \times \esup_{s\in(x_{k-1},t)}U(s)\varphi(s)^{-\frac1p}\left(\int_s^t\Delta^{\frac{r}{1-r}}\delta U^{-\frac{r}{1-r}}\right)^{\frac{1-r}{r}}\\
		& =: \I + \II.
	\end{align*}
	Now, for each $i\in \Z,$ $N+2\leq i \leq k,$ integration by parts and \eqref{int-cut-D} yields
	\begin{align}\label{EQ:antid_ii_5}
		&\left(\int_{x_{i-1}}^t\Delta^{\frac{r}{1-r}}\delta\varphi^{-\frac{r}{p(1-r)}}\right)^{\frac{1-r}{r}}
		\lesssim
		\Delta(t)^{\frac{1}{r}}\varphi(t)^{-\frac{1}{p}}+\left(\int_{x_{i-1}}^t\Delta(\tau)^{\frac{1}{1-r}}d\left[-\varphi^{-\frac{r}{p(1-r)}}(\tau)\right]\right)^{\frac{1-r}{r}}\nonumber\\
		&\qquad \approx
		\Delta(t)^{\frac{1}{r}}\varphi(t)^{-\frac{1}{p}} +
		\Delta(x_{i-1})^{\frac{1}{r}} \left(\int_{x_{i-1}}^t d\left[-\varphi^{-\frac{r}{p(1-r)}}(\tau)\right]\right)^{\frac{1-r}{r}}\nonumber \\
		&\hspace{2cm}  +
		\left(\int_{x_{i-1}}^t\left(\int_{x_{i-1}}^{\tau}\delta\right)^{\frac{1}{1-r}}d\left[-\varphi^{-\frac{r}{p(1-r)}}(\tau)\right]\right)^{\frac{1-r}{r}} \nonumber\\
		&\qquad\lesssim
		\Delta(t)^{\frac{1}{r}}\varphi(t)^{-\frac{1}{p}}
		+
		\Delta(x_{i-1})^{\frac{1}{r}}\varphi(x_{i-1})^{-\frac{1}{p}} \nonumber \\
		&\hspace{2cm} +
		\left(\int_{x_{i-1}}^t\left(\int_{x_{i-1}}^{\tau}\delta\right)^{\frac{1}{1-r}}d\left[-\varphi^{-\frac{r}{p(1-r)}}(\tau)\right]\right)^{\frac{1-r}{r}}\nonumber\\
		&\qquad\lesssim
		\Delta(t)^{\frac{1}{r}}\varphi(t)^{-\frac{1}{p}}
		+
		\Delta(x_{i-1})^{\frac{1}{r}}\varphi(x_{i-1})^{-\frac{1}{p}} \nonumber \\
		&\hspace{2cm}+
		\left(\int_{x_{i-1}}^t\left(\int_{x_{i-1}}^{\tau}\delta\right)^{\frac{r}{1-r}}\delta(\tau)\varphi(\tau)^{-\frac{r}{p(1-r)}}\,d\tau\right)^{\frac{1-r}{r}}.
	\end{align}
	Then \eqref{EQ:antid_ii_4} gives \begin{align*}
		\I
		&\approx
		\sup_{N+1\leq k\leq M-1}\left(\int_{x_k}^L\Delta^{-\frac{q}{r}}w\right)^{\frac{1}{q}}
		\sup_{N+1\leq i\leq k}\left(\int_{x_{i-1}}^{x_i}\Delta^{\frac{r}{1-r}}\delta\varphi^{-\frac{r}{p(1-r)}}\right)^{\frac{1-r}{r}}
		\\
		& \approx
		\left(\int_{x_{N+1}}^L\Delta^{-\frac{q}{r}}w\right)^{\frac{1}{q}}
		\left(\int_{x_{N}}^{x_{N+1}}\Delta^{\frac{r}{1-r}}\delta\varphi^{-\frac{r}{p(1-r)}}\right)^{\frac{1-r}{r}} \\
		& \quad +
		\left(\int_{x_{N+2}}^L\Delta^{-\frac{q}{r}}w\right)^{\frac{1}{q}}
		\left(\int_{x_{N}}^{x_{N+1}}\Delta^{\frac{r}{1-r}}\delta\varphi^{-\frac{r}{p(1-r)}}\right)^{\frac{1-r}{r}}\\
		&  \quad + \sup_{N+2\leq k\leq M-1}\left(\int_{x_k}^L\Delta^{-\frac{q}{r}}w\right)^{\frac{1}{q}}
		\sup_{N+2\leq i\leq k}\left(\int_{x_{i-1}}^{x_i}\Delta^{\frac{r}{1-r}}\delta\varphi^{-\frac{r}{p(1-r)}}\right)^{\frac{1-r}{r}}.
	\end{align*}
	Next, using \eqref{EQ:antid_ii_5} with $t=x_i$, we obtain
	\begin{align*}
		\I
		& \lesssim
		\left(\int_{x_{N+1}}^L\Delta^{-\frac{q}{r}}w\right)^{\frac{1}{q}}
		\left(\int_{x_{N}}^{x_{N+1}}\Delta^{\frac{r}{1-r}}\delta\varphi^{-\frac{r}{p(1-r)}}\right)^{\frac{1-r}{r}} \\
		& \quad +
		\left(\int_{x_{N+2}}^L\Delta^{-\frac{q}{r}}w\right)^{\frac{1}{q}}
		\left(\int_{x_{N}}^{x_{N+1}}\Delta^{\frac{r}{1-r}}\delta\varphi^{-\frac{r}{p(1-r)}}\right)^{\frac{1-r}{r}} \\
		&\quad +\sup_{N+2\leq k\leq M-1}\left(\int_{x_k}^L\Delta^{-\frac{q}{r}}w\right)^{\frac{1}{q}}
		\sup_{N+2\leq i\leq k}\Delta(x_i)^{\frac{1}{r}}\varphi(x_i)^{-\frac{1}{p}}
		\\
		&\qquad +
		\sup_{N+2\leq k\leq M-1}\left(\int_{x_k}^L\Delta^{-\frac{q}{r}}w\right)^{\frac{1}{q}}
		\sup_{N+2\leq i\leq k}\Delta(x_{i-1})^{\frac{1}{r}}\varphi(x_{i-1})^{-\frac{1}{p}}
		\\
		&\qquad +
		\sup_{N+2\leq k\leq M-1}\left(\int_{x_k}^L\Delta^{-\frac{q}{r}}w\right)^{\frac{1}{q}}\\
		& \hspace{2cm}\times
		\sup_{N+2\leq i\leq k}\left(\int_{x_{i-1}}^{x_i}\left(\int_{x_{i-1}}^{\tau}\delta\right)^{\frac{r}{1-r}}\delta(\tau)\varphi(\tau)^{-\frac{r}{p(1-r)}}\,d\tau\right)^{\frac{1-r}{r}}\\
		& \approx
		\sup_{N+2\leq k\leq M-1}\left(\int_{x_k}^L\Delta^{-\frac{q}{r}}w\right)^{\frac{1}{q}}
		\sup_{N+2\leq i\leq k}\Delta(x_i)^{\frac{1}{r}}\varphi(x_i)^{-\frac{1}{p}}
		\\
		&\qquad  +
		\sup_{N+2\leq k\leq M-1}\left(\int_{x_k}^L\Delta^{-\frac{q}{r}}w\right)^{\frac{1}{q}}
		\sup_{N+2\leq i\leq k}\Delta(x_{i-1})^{\frac{1}{r}}\varphi(x_{i-1})^{-\frac{1}{p}}
		\\
		&\qquad +
		\sup_{N+1\leq k\leq M-1}\left(\int_{x_k}^L\Delta^{-\frac{q}{r}}w\right)^{\frac{1}{q}}  \\
		& \hspace{2cm} \times
		\sup_{N+1\leq i\leq k}\left(\int_{x_{i-1}}^{x_i}\left(\int_{x_{i-1}}^{\tau}\delta\right)^{\frac{r}{1-r}}\delta(\tau)\varphi(\tau)^{-\frac{r}{p(1-r)}}\,d\tau\right)^{\frac{1-r}{r}}.
	\end{align*}
	Reindexing $i\mapsto i+1$ in the second term  and interchanging the suprema in the last term, we get
	\begin{align*}
		\I
		&\lesssim
		\sup_{N+1\leq k\leq M-1}\left(\int_{x_k}^L\Delta^{-\frac{q}{r}}w\right)^{\frac{1}{q}}
		\sup_{N+1\leq i\leq k}\Delta(x_i)^{\frac{1}{r}}\varphi(x_i)^{-\frac{1}{p}}
		+
		C_{3,2}.
	\end{align*}
	Using \eqref{EQ:strongly_decreasing_sup_sum},and then interchanging suprema combined with \eqref{C21<C31} and \eqref{EQ:C31lessC32}, we arrive at
	\begin{align}\label{EQ:antid_ii_I}
		\I
		&\lesssim \sup_{N+1\leq k\leq M-1}\left(\int_{x_k}^L\Delta^{-\frac{q}{r}}w\right)^{\frac{1}{q}}
		\sup_{N+1\leq i\leq k}\left(\int_{x_{i-1}}^{x_i}\delta\right)^{\frac{1}{r}}\varphi(x_i)^{-\frac{1}{p}}
		+
		C_{3,2} \nonumber \\
		&= C_{2,1} + C_{3,2} \nonumber\\	
		&\lesssim C_{3,2}.
	\end{align}
	
	For future reference, combining \eqref{1} and \eqref{EQ:antid_ii_I}, we have showed 
	\begin{equation} \label{8}
		\sup_{N+1\leq k\leq M-1}\left(\int_{x_k}^L\Delta^{-\frac{q}{r}}w\right)^{\frac{1}{q}} \esup_{s\in(0,x_k)}U(s)\varphi(s)^{-\frac1p}\left(\int_s^{x_k}\Delta^{\frac{r}{1-r}}\delta U^{-\frac{r}{1-r}}\right)^{\frac{1-r}{r}} \lesssim C_{3,2}.
	\end{equation}
	
	As for $\II$, observe that, applying \eqref{EQ:antid_ii_5} with $i=k$, we have
	\begin{align}
		&\sup_{N+1\leq k\leq M}\esup_{t\in(x_{k-1},x_k)} \left(\int_t^{x_{k}}\Delta^{-\frac{q}{r}}w\right)^{\frac{1}{q}}
		\left(\int_{x_{k-1}}^t\Delta^{\frac{r}{1-r}}\delta\varphi^{-\frac{r}{p(1-r)}}\, \right)^{\frac{1-r}{r}}\nonumber\\
		& \lesssim \sup_{N+1\leq k\leq M}\esup_{t\in(x_{k-1},x_k)} \left(\int_t^{x_{k}}\Delta^{-\frac{q}{r}}w\right)^{\frac{1}{q}} \nonumber \\
		& \hspace{4cm} \times \left(\int_{x_{k-1}}^t\left(\int_{x_{k-1}}^{\tau}\delta\right)^{\frac{r}{1-r}}\delta(\tau)\varphi(\tau)^{-\frac{r}{p(1-r)}}\,d\tau\right)^{\frac{1-r}{r}}\nonumber\\
		&\qquad\quad + \sup_{N+2\leq k\leq M} \left(\int_{x_{k-1}}^{x_{k}}\Delta^{-\frac{q}{r}}w\right)^{\frac{1}{q}} \Delta(x_{k-1})^{\frac{1}{r}}\varphi(x_{k-1})^{-\frac{1}{p}}\nonumber\\
		&\qquad \quad + \sup_{N+2\leq k\leq M}\esup_{t\in(x_{k-1},x_k)} \left(\int_t^{x_{k}}\Delta^{-\frac{q}{r}}w\right)^{\frac{1}{q}} \Delta(t)^{\frac{1}{r}}\varphi(t)^{-\frac{1}{p}}\nonumber\\
		& = C_{1,3} + \sup_{N+2\leq k\leq M} \left(\int_{x_{k-1}}^{x_{k}}\Delta^{-\frac{q}{r}}w\right)^{\frac{1}{q}} \Delta(x_{k-1})^{\frac{1}{r}}\varphi(x_{k-1})^{-\frac{1}{p}}\nonumber\\
		&\qquad \quad + \sup_{N+2\leq k\leq M}\esup_{t\in(x_{k-1},x_k)} \left(\int_t^{x_{k}}\Delta^{-\frac{q}{r}}w\right)^{\frac{1}{q}} \Delta(t)^{\frac{1}{r}}\varphi(t)^{-\frac{1}{p}}.\nonumber
	\end{align}
	Applying \eqref{int-cut-D} once again, we have
	\begin{align*}
		&\sup_{N+1\leq k\leq M}\esup_{t\in(x_{k-1},x_k)} \left(\int_t^{x_{k}}\Delta^{-\frac{q}{r}}w\right)^{\frac{1}{q}}
		\left(\int_{x_{k-1}}^t\Delta^{\frac{r}{1-r}}\delta\varphi^{-\frac{r}{p(1-r)}}\, \right)^{\frac{1-r}{r}}\\
		& \qquad  \lesssim C_{1,3} + \sup_{N+2\leq k\leq M} \left(\int_{x_{k-1}}^{x_{k}}\Delta^{-\frac{q}{r}}w\right)^{\frac{1}{q}} \Delta(x_{k-1})^{\frac{1}{r}}\varphi(x_{k-1})^{-\frac{1}{p}}\\
		&\qquad \quad + \sup_{N+2\leq k\leq M}\esup_{t\in(x_{k-1},x_k)} \left(\int_t^{x_{k}}\Delta^{-\frac{q}{r}}w\right)^{\frac{1}{q}} \left(\int_{x_{k-1}}^t\delta\right)^{\frac{1}{r}}\varphi(t)^{-\frac{1}{p}}\\
		& \qquad  \lesssim C_{1,3} +
		\sup_{N+2\leq k\leq M} \left(\int_{x_{k-1}}^L\Delta^{-\frac{q}{r}}w\right)^{\frac{1}{q}} \Delta(x_{k-1})^{\frac{1}{r}}\varphi(x_{k-1})^{-\frac{1}{p}}\nonumber\\
		&\quad\qquad+
		\sup_{N+2\leq k\leq M}\esup_{t\in(x_{k-1},x_k)} \left(\int_t^{x_{k}}\Delta^{-\frac{q}{r}}w\right)^{\frac{1}{q}} \left(\int_{x_{k-1}}^t\delta\right)^{\frac{1}{r}}\varphi(t)^{-\frac{1}{p}}.
	\end{align*}
	Applying  \eqref{EQ:strongly_decreasing_sup_sum} in the second term, \eqref{sup-int} in the third term and using \eqref{C21<C31} together with \eqref{EQ:C31lessC32}, we have 
	\begin{align}\label{2}
		&\sup_{N+1\leq k\leq M}\esup_{t\in(x_{k-1},x_k)} \left(\int_t^{x_{k}}\Delta^{-\frac{q}{r}}w\right)^{\frac{1}{q}}
		\left(\int_{x_{k-1}}^t\Delta^{\frac{r}{1-r}}\delta\varphi^{-\frac{r}{p(1-r)}}\, \right)^{\frac{1-r}{r}}\nonumber\\
		& \qquad \lesssim C_{1,3} +
		\sup_{N+1\leq k\leq M-1} \left(\int_{x_{k}}^L\Delta^{-\frac{q}{r}}w\right)^{\frac{1}{q}} \left(\int_{x_{k-1}}^{x_k}\delta\right)^{\frac{1}{r}}\varphi(x_k)^{-\frac{1}{p}}\nonumber\\
		&\qquad \quad+
		\sup_{N+2\leq k\leq M}\esup_{t\in(x_{k-1},x_k)} \left(\int_t^{x_{k}}\Delta^{-\frac{q}{r}}w\right)^{\frac{1}{q}} \nonumber \\
		& \hspace{4cm} \times\left(\int_{x_{k-1}}^t \left(\int_{x_{k-1}}^\tau\delta\right)^{\frac{r}{1-r}}\delta(\tau)\varphi(\tau)^{-\frac{r}{p(1-r)}}d\tau\right)^{\frac{1-r}{r}}
		\nonumber\\
		&\qquad \leq 2C_{1,3} +C_{2,1} \nonumber\\
		& \qquad \lesssim C_{1,3} + C_{3,2}.
	\end{align}
	Next, thanks to $U\varphi^{-\frac1p}$ being nondecreasing and \eqref{2}, we have
	\begin{align*}
		\II
		\leq
		\sup_{N+1\leq k\leq M}\esup_{t\in(x_{k-1},x_k)} \left(\int_t^{x_{k}}\Delta^{-\frac{q}{r}}w\right)^{\frac{1}{q}}
		\left(\int_{x_{k-1}}^t\Delta^{\frac{r}{1-r}}\delta\varphi^{-\frac{r}{p(1-r)}}\, \right)^{\frac{1-r}{r}}
		\lesssim C_{1,3} + C_{3,2}.
	\end{align*}
	Thus, we have arrived at
	\begin{align}\label{EQ:antid_ii_B3}
		B_3\lesssim\I+\II\lesssim C_{1,3} + C_{3,2}.
	\end{align}
	Combining \eqref{EQ:antid_ii_B12} and \eqref{EQ:antid_ii_B3}, we have
	\begin{align*}
		B_1+B_2+B_3\lesssim C_{1,2} + C_{1,3} + C_{3,2} + C_{4,1},
	\end{align*}
	which together with \eqref{EQ:antid_ii_B123} yields
	\begin{align*}
		C \approx B_1+B_2+B_3.
	\end{align*}
	
	\rm{(iii)} By Theorem~\ref{thm:main_discretization}, we have
	\begin{equation*}
		C\approx C_{1,1}+C_{1,2}+C_{3,3}+C_{4,1}.
	\end{equation*}
	As for the desired upper estimate on $C$, it is sufficient to show that $C_{3,3}\lesssim B_2 + B_4$ owing to \eqref{EQ:C11lesB2}, \eqref{EQ:C12lesB1}, \eqref{EQ:C41lesB1}. To this end, one has
	\begin{align}\label{E:C3.3-1}
		C_{3,3} & \approx \sup_{N+1\le k\le M-1}\left(\int_{x_k}^{L}\Delta^{-\frac{q}{r}}w\right)^{\iq}\esup_{t\in(x_{k-1},x_k)}\left(\int_{x_{k-1}}^{t}\delta\right)^{\frac{1}{r}}\varphi(t)^{-\frac{1}{p}} \nonumber
		\\
		& \quad + \sup_{N+2\le k\le M-1}\left(\int_{x_k}^{L}\Delta^{-\frac{q}{r}}w\right)^{\iq} \left(\sum_{i=N+2}^{k-1}\esup_{t\in(x_{i-1},x_i)}\left(\int_{x_{i-1}}^{t}\delta\right)^{\frac{p}{p-r}}
		\varphi(t)^{-\frac{r}{p-r}}\right)^{\frac{p-r}{pr}} \nonumber
		\\
		& \ls B_2  + \sup_{N+2\le k\le M-1}\left(\int_{x_k}^{L}\Delta^{-\frac{q}{r}}w\right)^{\iq} \nonumber\\
		& \hspace{3cm} \times\left(\sum_{i=N+2}^{k-1}\esup_{t\in(x_{i-1},x_i)}\left(\int_{x_{i-1}}^{t}\delta\right)^{\frac{p}{p-r}}
		\varphi(t)^{-\frac{r}{p-r}}\right)^{\frac{p-r}{pr}}.
	\end{align}
	
	Observe that for $N+2 \le k \le M-1$, we have by using \eqref{phi-Z1} and \eqref{Uphi-Z2}
	\begin{align*}
		&\sum_{i=N+2}^{k-1}\esup_{t\in(x_{i-1},x_i)}\left(\int_{x_{i-1}}^{t}\delta\right)^{\frac{p}{p-r}}\varphi(t)^{-\frac{r}{p-r}}
		\le \sum_{i=N+2}^{k-1}\esup_{t\in(x_{i-1},x_i)}\Delta(t)^{\frac{p}{p-r}}\varphi(t)^{-\frac{r}{p-r}}
		\\
		&\approx \sum_{\substack{i\in\mathcal{Z}_1\\N+2\le i\le k-1}}\varphi(x_i)^{-\frac{r}{p-r}}\Delta(x_i)^{\frac{p}{p-r}}  \\
		&\quad +
		\sum_{\substack{i\in\mathcal{Z}_2\\N+2\le i\le k-1}}\varphi(x_{i-1})^{-\frac{r}{p-r}} U(x_{i-1})^{\frac{pr}{p-r}}\esup_{t\in(x_{i-1},x_i)}\Delta(t)^{\frac{p}{p-r}}U(t)^{-\frac{pr}{p-r}}
		\\
		& \le \sum_{\substack{i\in\mathcal{Z}_1\\N+2\le i\le k-1}}\varphi(x_i)^{-\frac{r}{p-r}}U(x_{i})^{\frac{pr}{p-r}}\esup_{t\in(x_{i},x_{i+1})}\Delta(t)^{\frac{p}{p-r}}U(t)^{-\frac{pr}{p-r}}
		\\
		&\quad+\sum_{\substack{i\in\mathcal{Z}_2\\N+2\le i\le k-1}}\varphi(x_{i-1})^{-\frac{r}{p-r}} U(x_{i-1})^{\frac{pr}{p-r}}\esup_{t\in(x_{i-1},x_i)}\Delta(t)^{\frac{p}{p-r}}U(t)^{-\frac{pr}{p-r}}
		\\
		&  \lesssim \sum_{i=N+1}^{k-1}\varphi(x_i)^{-\frac{r}{p-r}}U(x_{i})^{\frac{pr}{p-r}}\esup_{t\in(x_{i},x_k)}\Delta(t)^{\frac{p}{p-r}}U(t)^{-\frac{pr}{p-r}} 
		\\
		&  \quad + \sum_{i=N+2}^{k-1}\varphi(x_{i-1})^{-\frac{r}{p-r}}U(x_{i-1})^{\frac{pr}{p-r}}\esup_{t\in(x_{i-1},x_{k})}\Delta(t)^{\frac{p}{p-r}}U(t)^{-\frac{pr}{p-r}} \\
		& \lesssim \sum_{i=N+1}^{k-1}\varphi(x_i)^{-\frac{r}{p-r}}U(x_{i})^{\frac{pr}{p-r}}\esup_{t\in(x_{i},x_k)}\Delta(t)^{\frac{p}{p-r}}U(t)^{-\frac{pr}{p-r}}
		\\
		& \lesssim \int_{0}^{x_k}\sigma(t)U(t)^{\frac{pr}{p-r}}\esup_{\tau\in(t,x_k)}\Delta(\tau)^{\frac{p}{p-r}}U(\tau)^{-\frac{pr}{p-r}}\,dt,
	\end{align*}
	where we used \eqref{R1} in the last inequality.
	Inserting this to~\eqref{E:C3.3-1}, we obtain
	\begin{align*}
		&C_{3,3} \ls B_2 + \sup_{N+2\le k\le M-1}
		\left(\int_{x_k}^{L}\Delta^{-\frac{q}{r}}w\right)^{\iq} \\
		& \hspace{3cm} \times \left(\int_{0}^{x_k}\sigma(t) U(t)^{\frac{pr}{p-r}}\esup_{\tau\in(t,x_k)}\Delta(\tau)^{\frac{p}{p-r}}U(\tau)^{-\frac{pr}{p-r}}\,dt\right)^{\frac{p-r}{pr}}
		\\
		& \le B_2 + \sup_{N+2\le k\le M-1}\sup_{x\in(x_k,x_{k+1})}\left(\int_{x}^{L}\Delta^{-\frac{q}{r}}w\right)^{\iq} \\
		& \hspace{3cm} \times
		\left(\int_0^x\sigma(t) U(t)^{\frac{pr}{p-r}}\esup_{\tau\in(t,x)}\Delta(\tau)^{\frac{p}{p-r}}U(\tau)^{-\frac{pr}{p-r}}\,dt\right)^{\frac{p-r}{pr}}
		\\
		& \le B_2 + B_4.
	\end{align*}
	It follows that $C_{1,1}+C_{1,2}+C_{3,3}+C_{4,1} \ls B_1+B_2+B_4$.
	
	As for establishing the opposite inequality,  first observe that
	\begin{equation}\label{C-31<C-33}
		C_{3,1} \leq C_{3,3}.
	\end{equation}
	Owing to \eqref{EQ:B1+B2UpperBoundCase1} and \eqref{C-31<C-33}, we have
	\begin{equation*}
		B_1+B_2 \ls C_{1,1} + C_{1,2} + C_{3,1} + C_{4,1} \le C_{1,1} + C_{1,2} + C_{3,3} + C_{4,1}.
	\end{equation*}
	Therefore, it is sufficient to show that $B_4\lesssim C_{3,3} +  C_{1,1}$. To this end, first using \eqref{int-cut-Dw}, then decomposing the integral $\int_{0}^{t}$ into the sum $\int_{0}^{x_{k-1}}+\int_{x_{k-1}}^{t}$, we have
	\begin{align*}
		B_4    &= \sup_{N+1\le k\le M} \esup_{t\in(x_{k-1},x_k)} \left(\int_t^L\Delta^{-\frac{q}{r}}w\right)^{\frac{1}{q}} \\
		& \hspace{3cm} \times
		\left(\int_{0}^{t}\sigma(s)U(s)^{\frac{pr}{p-r}}\esup_{\tau\in(s, t)}\Delta(\tau)^{\frac{p}{p-r}}U(\tau)^{-\frac{pr}{p-r}}\, ds\right)^{\frac{p-r}{pr}}
		\\
		&\approx \sup_{N+1\le k\le M} \esup_{t\in(x_{k-1},x_k)} \left(\int_t^{x_k}\Delta^{-\frac{q}{r}}w\right)^{\frac{1}{q}} \\
		& \hspace{3cm} \times
		\left(\int_{0}^{t}\sigma(s)U(s)^{\frac{pr}{p-r}}\esup_{\tau\in(s, t)}\Delta(\tau)^{\frac{p}{p-r}}U(\tau)^{-\frac{pr}{p-r}}\, ds\right)^{\frac{p-r}{pr}}
		\\
		& + \sup_{N+1\le k\le M-1} \left(\int_{x_k}^{L}\Delta^{-\frac{q}{r}}w\right)^{\frac{1}{q}} \\
		& \hspace{3cm} \times
		\left(\int_{0}^{x_k}\sigma(s)U(s)^{\frac{pr}{p-r}}\esup_{\tau\in(s, x_k)}\Delta(\tau)^{\frac{p}{p-r}}U(\tau)^{-\frac{pr}{p-r}}\, ds\right)^{\frac{p-r}{pr}}\\
		&\approx \sup_{N+2\le k\le M} \esup_{t\in(x_{k-1},x_k)}  \left(\int_t^{x_k}\Delta^{-\frac{q}{r}}w\right)^{\frac{1}{q}}\\
		& \hspace{3cm} \times
		\left(\int_{0}^{x_{k-1}}\sigma(s)U(s)^{\frac{pr}{p-r}}\esup_{\tau\in(s, t)}\Delta(\tau)^{\frac{p}{p-r}}U(\tau)^{-\frac{pr}{p-r}}\, ds\right)^{\frac{p-r}{pr}}
		\\
		& + \sup_{N+1\le k\le M} \esup_{t\in(x_{k-1},x_k)}\left(\int_{t}^{x_k}\Delta^{-\frac{q}{r}}w\right)^{\frac{1}{q}} \\
		& \hspace{3cm} \times
		\left(\int_{x_{k-1}}^t\sigma(s)U(s)^{\frac{pr}{p-r}}\esup_{\tau\in(s, t)}\Delta(\tau)^{\frac{p}{p-r}}U(\tau)^{-\frac{pr}{p-r}}\, ds\right)^{\frac{p-r}{pr}}
		\\
		& + \sup_{N+1\le k\le M-1} \left(\int_{x_k}^{L}\Delta^{-\frac{q}{r}}w\right)^{\frac{1}{q}} \\
		& \hspace{3cm} \times
		\left(\int_{0}^{x_k}\sigma(s)U(s)^{\frac{pr}{p-r}}\esup_{\tau\in(s, x_k)}\Delta(\tau)^{\frac{p}{p-r}}U(\tau)^{-\frac{pr}{p-r}}\, ds\right)^{\frac{p-r}{pr}}.
	\end{align*}
	Since
	\begin{align}
		\esup_{\tau \in (s, t)}\Delta(\tau)^{\frac{p}{p-r}}&U(\tau)^{-\frac{pr}{p-r}} \notag\\
		&\approx \esup_{\tau \in (s, x_{k-1})}\Delta(\tau)^{\frac{p}{p-r}}U(\tau)^{-\frac{pr}{p-r}} + \esup_{\tau \in (x_{k-1}, t)}\Delta(\tau)^{\frac{p}{p-r}}U(\tau)^{-\frac{pr}{p-r}}
	\end{align}
	when $0 < s < x_{k-1}< t$, we get
	\begin{align*}
		B_4
		&\approx \sup_{N+2\le k\le M}  \left(\int_{x_{k-1}}^{x_k}\Delta^{-\frac{q}{r}}w\right)^{\frac{1}{q}} \\
		& \hspace{3cm} \times
		\left(\int_{0}^{x_{k-1}}\sigma(s)U(s)^{\frac{pr}{p-r}}\esup_{\tau \in (s, x_{k-1})}\Delta(\tau)^{\frac{p}{p-r}}U(\tau)^{-\frac{pr}{p-r}}\, ds\right)^{\frac{p-r}{pr}}
		\\
		&+ \sup_{N+2\le k\le M} \left(\int_{0}^{x_{k-1}}\sigma(s)U(s)^{\frac{pr}{p-r}}\, ds\right)^{\frac{p-r}{pr}} \\
		& \hspace{3cm} \times \esup_{t\in(x_{k-1},x_k)} \left(\int_t^{x_k}\Delta^{-\frac{q}{r}}w\right)^{\frac{1}{q}}
		\esup_{\tau\in(x_{k-1}, t)}\Delta(\tau)^{\frac{1}{r}}U(\tau)^{-1}
		\\
		& + \sup_{N+1\le k\le M} \esup_{t\in(x_{k-1},x_k)}\left(\int_{t}^{x_k}\Delta^{-\frac{q}{r}}w\right)^{\frac{1}{q}} \\
		& \hspace{3cm} \times
		\left(\int_{x_{k-1}}^t\sigma(s)U(s)^{\frac{pr}{p-r}}\esup_{\tau\in(s, t)}\Delta(\tau)^{\frac{p}{p-r}}U(\tau)^{-\frac{pr}{p-r}}\, ds\right)^{\frac{p-r}{pr}}
		\\
		& + \sup_{N+1\le k\le M-1} \left(\int_{x_k}^{L}\Delta^{-\frac{q}{r}}w\right)^{\frac{1}{q}} \\
		& \hspace{3cm} \times
		\left(\int_{0}^{x_k}\sigma(s)U(s)^{\frac{pr}{p-r}}\esup_{\tau\in(s, x_k)}\Delta(\tau)^{\frac{p}{p-r}}U(\tau)^{-\frac{pr}{p-r}}\, ds\right)^{\frac{p-r}{pr}}.
	\end{align*}
	Reindexing $k \mapsto k+1$ in the first term, we obtain
	\begin{align*}
		B_4 
		&\leq \sup_{N+1\le k\le M-1}  \left(\int_{x_{k}}^{L}\Delta^{-\frac{q}{r}}w\right)^{\frac{1}{q}}
		\left(\int_{0}^{x_{k}}\sigma(s)U(s)^{\frac{pr}{p-r}}\esup_{\tau\in(s, x_k)}\Delta(\tau)^{\frac{p}{p-r}}U(\tau)^{-\frac{pr}{p-r}}\, ds\right)^{\frac{p-r}{pr}}
		\\
		&+ \sup_{N+2\le k\le M} \left(\int_{0}^{x_{k-1}}\sigma(s)U(s)^{\frac{pr}{p-r}}\, ds\right)^{\frac{p-r}{pr}}\\
		& \hspace{3cm} \times\esup_{t\in(x_{k-1},x_k)} \left(\int_t^{x_k}\Delta^{-\frac{q}{r}}w\right)^{\frac{1}{q}}
		\esup_{\tau\in(x_{k-1}, t)}\Delta(\tau)^{\frac{1}{r}}U(\tau)^{-1}
		\\
		& + \sup_{N+1\le k\le M} \esup_{t\in(x_{k-1},x_k)}\left(\int_{t}^{x_k}\Delta^{-\frac{q}{r}}w\right)^{\frac{1}{q}} \\
		& \hspace{3cm} \times
		\left(\int_{x_{k-1}}^t\sigma(s)U(s)^{\frac{pr}{p-r}}\esup_{\tau\in(s, t)}\Delta(\tau)^{\frac{p}{p-r}}U(\tau)^{-\frac{pr}{p-r}}\, ds\right)^{\frac{p-r}{pr}}
		\\
		&=: \I + \II + \III.
	\end{align*}
	
	In view of~\eqref{R2}, we have
	\begin{align*}
		\I &\ls \sup_{N+1\le k \le M-1}\left(\int_{x_k}^{L}\Delta^{-\frac{q}{r}}w\right)^{\frac{1}{q}} \\
		& \hspace{2cm} \times
		\left(\sum_{i=N+1}^{k}\varphi(x_i)^{-\frac{r}{p-r}}U(x_i)^{\frac{pr}{p-r}}\esup_{\tau\in(x_i,x_k)}\Delta(\tau)^{\frac{p}{p-r}}U(\tau)^{-\frac{pr}{p-r}}\right)^{\frac{p-r}{pr}}
		\\
		&\quad + \left(\int_{x_{N+1}}^{L}\Delta^{-\frac{q}{r}}w\right)^{\frac{1}{q}}\esup_{t\in(0,x_{N+1})}\Delta(\tau)^{\frac{1}{r}}\varphi(t)^{-\frac{1}{p}}.
	\end{align*}
	Note that
	\begin{equation}
		\left(\int_{x_{N+1}}^{L}\Delta^{-\frac{q}{r}}w\right)^{\frac{1}{q}}\esup_{t\in(0,x_{N+1})}\Delta(t)^{\frac{1}{r}}\varphi(t)^{-\frac{1}{p}} \le C_{3,3}. \label{EQ:C33LBCase3}
	\end{equation}
	Then
	\begin{align*}
		\I &\ls C_{3,3} +  \sup_{N+1\le k \le M-1}\left(\int_{x_k}^{L}\Delta^{-\frac{q}{r}}w\right)^{\frac{1}{q}} \varphi(x_k)^{-\frac{1}{p}} \Delta(x_k)^{\frac{1}{r}}\\
		& \qquad + \sup_{N+2\le k \le M-1}\left(\int_{x_k}^{L}\Delta^{-\frac{q}{r}}w\right)^{\frac{1}{q}} \\
		& \hspace{2cm} \times
		\left(\sum_{i=N+1}^{k-1}\varphi(x_i)^{-\frac{r}{p-r}}U(x_i)^{\frac{pr}{p-r}}\esup_{\tau\in(x_i,x_k)}\Delta(\tau)^{\frac{p}{p-r}}U(\tau)^{-\frac{pr}{p-r}}\right)^{\frac{p-r}{pr}}\\
		&\leq C_{3,3} +  \sup_{N+1\le k \le M-1}\left(\int_{x_k}^{L}\Delta^{-\frac{q}{r}}w\right)^{\frac{1}{q}} \left(\sum_{i=N+1}^{k}\varphi(x_i)^{-\frac{r}{p-r}}\Delta(x_i)^{\frac{p}{p-r}}\right)^{\frac{p-r}{pr}}\\
		& \qquad + \sup_{N+2\le k \le M-1}\left(\int_{x_k}^{L}\Delta^{-\frac{q}{r}}w\right)^{\frac{1}{q}} \\
		& \hspace{2cm} \times
		\left(\sum_{i=N+1}^{k-1}\varphi(x_i)^{-\frac{r}{p-r}}U(x_i)^{\frac{pr}{p-r}}\esup_{\tau\in(x_i,x_k)}\Delta(\tau)^{\frac{p}{p-r}}U(\tau)^{-\frac{pr}{p-r}}\right)^{\frac{p-r}{pr}}.
	\end{align*}
	On the other hand, for  $N+2 \leq k \leq M$,
	\begin{align*}
		&\sum_{i=N+1}^{k-1}\varphi(x_i)^{-\frac{r}{p-r}}U(x_i)^{\frac{pr}{p-r}}\esup_{\tau\in(x_i,x_k)}\Delta(\tau)^{\frac{p}{p-r}}U(\tau)^{-\frac{pr}{p-r}}
		\\
		&= \sum_{i=N+1}^{k-1}\varphi(x_i)^{-\frac{r}{p-r}}U(x_i)^{\frac{pr}{p-r}}\sup_{i\le m\le k-1}\esup_{\tau\in(x_m,x_{m+1})}\Delta(\tau)^{\frac{p}{p-r}}U(\tau)^{-\frac{pr}{p-r}}.
	\end{align*}
	Since $\{U^p(x_i)/\varphi(x_i)\}_{i=N+1}^{k-1}$ is strongly increasing, using \eqref{EQ:strongly_increasing_sum_sup}
	\begin{align*}
		\sum_{i=N+1}^{k-1} & \varphi(x_i)^{-\frac{r}{p-r}}U(x_i)^{\frac{pr}{p-r}}\esup_{\tau\in(x_i,x_k)}\Delta(\tau)^{\frac{p}{p-r}}U(\tau)^{-\frac{pr}{p-r}}
		\\
		&\approx \sum_{i=N+1}^{k-1}\varphi(x_i)^{-\frac{r}{p-r}}U(x_i)^{\frac{pr}{p-r}}\esup_{\tau\in(x_i,x_{i+1})}\Delta(\tau)^{\frac{p}{p-r}}U(\tau)^{-\frac{pr}{p-r}}
		\\
		&\le \sum_{i=N+1}^{k-1}\esup_{\tau\in(x_i,x_{i+1})}\Delta(\tau)^{\frac{p}{p-r}}\varphi(\tau)^{-\frac{r}{p-r}}
	\end{align*}
	holds. Consequently,
	\begin{align*}
		\I & \ls  C_{3,3} + \sup_{N+1\le k\le M-1}\left(\int_{x_{k}}^{L}\Delta^{-\frac{q}{r}}w\right)^{\frac{1}{q}}
		\left(\sum_{i=N+1}^{k}\varphi(x_i)^{-\frac{r}{p-r}}\Delta(x_i)^{\frac{p}{p-r}}\right)^{\frac{p-r}{pr}}
		\\
		&\quad + \sup_{N+2\le k\le M-1}\left(\int_{x_{k}}^{L}\Delta^{-\frac{q}{r}}w\right)^{\frac{1}{q}}
		\left(\sum_{i=N+1}^{k-1}\esup_{\tau\in(x_i,x_{i+1})}\Delta(\tau)^{\frac{p}{p-r}}\varphi(\tau)^{-\frac{r}{p-r}}\right)^{\frac{p-r}{pr}}.
	\end{align*}
	Reindexing $i \mapsto i-1$ in the second term, we get
	\begin{align*}
		\I & \ls  C_{3,3} + \sup_{N+1\le k\le M-1}\left(\int_{x_{k}}^{L}\Delta^{-\frac{q}{r}}w\right)^{\frac{1}{q}}
		\left(\sum_{i=N+1}^{k}\varphi(x_i)^{-\frac{r}{p-r}}\Delta(x_i)^{\frac{p}{p-r}}\right)^{\frac{p-r}{pr}}
		\\
		&\quad + \sup_{N+2\le k\le M-1}\left(\int_{x_{k}}^{L}\Delta^{-\frac{q}{r}}w\right)^{\frac{1}{q}}
		\left(\sum_{i=N+2}^{k}\esup_{\tau\in(x_{i-1},x_{i})}\Delta(\tau)^{\frac{p}{p-r}}\varphi(\tau)^{-\frac{r}{p-r}}\right)^{\frac{p-r}{pr}}.    
	\end{align*}
	
	Now, applying \eqref{int-cut-D},  we obtain
	\begin{align*}
		\I & \ls  C_{3,3} + \sup_{N+1\le k\le M-1}\left(\int_{x_{k}}^{L}\Delta^{-\frac{q}{r}}w\right)^{\frac{1}{q}}
		\left(\sum_{i=N+1}^{k}\varphi(x_i)^{-\frac{r}{p-r}}\Delta(x_i)^{\frac{p}{p-r}}\right)^{\frac{p-r}{pr}}\\
		&\quad + \sup_{N+2\le k\le M-1}\left(\int_{x_{k}}^{L}\Delta^{-\frac{q}{r}}w\right)^{\frac{1}{q}}
		\left(\sum_{i=N+2}^{k} \Delta(x_{i-1})^{\frac{p}{p-r}}\varphi(x_{i-1})^{-\frac{r}{p-r}}\right)^{\frac{p-r}{pr}}\\
		&\quad + \sup_{N+2\le k\le M-1}\left(\int_{x_{k}}^{L}\Delta^{-\frac{q}{r}}w\right)^{\frac{1}{q}}
		\left(\sum_{i=N+2}^{k} \esup_{\tau\in(x_{i-1},x_{i})} \left(\int_{x_{i-1}}^{\tau}\delta\right)^{\frac{p}{p-r}}\varphi(\tau)^{-\frac{r}{p-r}}\right)^{\frac{p-r}{pr}}\\
		& \lesssim C_{3,3} + \sup_{N+1\le k\le M-1}\left(\int_{x_{k}}^{L}\Delta^{-\frac{q}{r}}w\right)^{\frac{1}{q}}
		\left(\sum_{i=N+1}^{k}\varphi(x_i)^{-\frac{r}{p-r}}\Delta(x_i)^{\frac{p}{p-r}}\right)^{\frac{p-r}{pr}}.        
	\end{align*}
	Since $\{\varphi(x_i)^{-\frac{r}{p-r}}\}_{i=N+1}^{M-1}$ is strongly decreasing, applying \eqref{EQ:strongly_decreasing_sum_sum} we have
	\begin{align*}
		\I & \lesssim C_{3,3} + \sup_{N+1\le k\le M-1} \left(\int_{x_{k}}^{L}\Delta^{-\frac{q}{r}}w\right)^{\frac{1}{q}}
		\left(\sum_{i=N+1}^{k}\varphi(x_i)^{-\frac{r}{p-r}} \bigg(\int_{x_{i-1}}^{x_i} \delta \bigg)^{\frac{p}{p-r}}\right)^{\frac{p-r}{pr}}\\
		&\lesssim C_{3,3}.
	\end{align*}
	
	Let us continue with the estimate of II. First of all, since $\{U^p(x_k)/\varphi(x_k)\}_{k=N+1}^{M-1}$ is strongly increasing, we have from Lemma in \cite{L:93}*{Lemma} that for every $N+2 \leq k \leq M$,
	\begin{equation} \label{N1}
		\sum_{i=N+1}^{k-1} U(x_i)^{\frac{pr}{p-r}} \varphi(x_i)^{-\frac{r}{p-r}}
		\lesssim \varphi(x_{k-1})^{-\frac{r}{p-r}}U(x_{k-1})^{\frac{pr}{p-r}}.
	\end{equation}
	
	On the other hand, clearly $U\in Q_{U}(0,L)$. Thus, applying case (iii) in Lemma~\ref{TH:antid_lemma1} to $h=U$, we have, for every $N+2\le k\le M$,
	\begin{align*}
		\int_{0}^{x_{k-1}}\sigma(s)U(s)^{\frac{pr}{p-r}}\,ds &\ls \sum_{i=N+1}^{k-1} U(x_i)^{\frac{pr}{p-r}} \varphi(x_i)^{-\frac{r}{p-r}} + \esup_{t\in(0,x_{N+1})}U(t)^{\frac{pr}{p-r}}\varphi(t)^{-\frac{r}{p-r}}.
	\end{align*}
	Then in view of \eqref{N1}, we obtain
	\begin{align}
		\int_{0}^{x_{k-1}}\sigma(s)U(s)^{\frac{pr}{p-r}}\,ds &\ls \varphi(x_{k-1})^{-\frac{r}{p-r}}U(x_{k-1})^{\frac{pr}{p-r}} + U(x_{N+1})^{\frac{pr}{p-r}}\varphi(x_{N+1})^{-\frac{r}{p-r}} \nonumber\\
		&\ls \varphi(x_{k-1})^{-\frac{r}{p-r}}U(x_{k-1})^{\frac{pr}{p-r}}. \label{EQ:antid_iii_1}
	\end{align}
	Thus, first using \eqref{EQ:antid_iii_1}, then applying \eqref{int-cut-D}, we have
	\begin{align*}
		\II &\ls \sup_{N+2\le k\le M}\varphi(x_{k-1})^{-\frac{1}{p}}U(x_{k-1})
		\esup_{t\in(x_{k-1},x_k)} \left(\int_t^{x_k}\Delta^{-\frac{q}{r}}w\right)^{\frac{1}{q}}
		\esup_{\tau\in(x_{k-1}, t)}\Delta(\tau)^{\frac{1}{r}}U(\tau)^{-1} \nonumber \\
		&\approx \sup_{N+2\le k\le M}\varphi(x_{k-1})^{-\frac{1}{p}}\Delta(x_{k-1})^{\frac1r}
		\left(\int_{x_{k-1}}^{x_k}\Delta^{-\frac{q}{r}}w\right)^{\frac{1}{q}}\nonumber \\
		&\quad + \sup_{N+2\le k\le M}\varphi(x_{k-1})^{-\frac{1}{p}}U(x_{k-1}) \nonumber\\
		& \hspace{2cm} \times
		\esup_{t\in(x_{k-1},x_k)} \left(\int_t^{x_k}\Delta^{-\frac{q}{r}}w\right)^{\frac{1}{q}}
		\esup_{\tau\in(x_{k-1},t)} \left(\int_{x_{k-1}}^{\tau}\delta\right)^{\frac{1}{r}}U(\tau)^{-1}.
	\end{align*}
	Reindexing $k\mapsto k+1$ and using the monotonicity of $U\varphi^{-\frac{1}{p}}$ gives
	\begin{align}
		\II &\ls  \sup_{N+1\le k\le M-1}  \varphi(x_k)^{-\frac{1}{p}} \Delta(x_k)^{\frac{1}{r}} \left(\int_{x_k}^L\Delta^{-\frac{q}{r}}w\right)^{\frac{1}{q}} \nonumber   \\
		&\quad + \sup_{N+1\le k\le M} \esup_{t\in(x_{k-1},x_k)} \left(\int_t^{x_k}\Delta^{-\frac{q}{r}}w\right)^{\frac{1}{q}}
		\esup_{\tau\in(x_{k-1},t)} \varphi(\tau)^{-\frac{1}{p}}\left(\int_{x_{k-1}}^{\tau}\delta\right)^{\frac{1}{r}}\nonumber \\
		& \approx C_{2,1}+C_{1,1} \label{3},
	\end{align}
	where we used \eqref{EQ:strongly_decreasing_sup_sum} in the last equivalence. Then using \eqref{C21<C31} and \eqref{C-31<C-33} we arrive at
	\begin{equation*}
		\II \lesssim C_{3,1} + C_{1,1} \leq C_{3,3} + C_{1,1}.
	\end{equation*}
	Finally,
	\begin{align*}
		\III     &\approx \esup_{t\in(0,x_{N+1})}\left(\int_{t}^{x_{N+1}}\Delta^{-\frac{q}{r}}w\right)^{\frac{1}{q}}
		\left(\int_{0}^{t}\sigma(s)U(s)^{\frac{pr}{p-r}}\esup_{\tau\in(s,t)}\Delta(\tau)^{\frac{p}{p-r}}U(\tau)^{-\frac{pr}{p-r}}\, ds\right)^{\frac{p-r}{pr}}
		\\
		&\quad + \sup_{\substack{N+2\le k\le M\\k\in\mathcal{Z}_1}} \esup_{t\in(x_{k-1},x_k)} \left(\int_t^{x_k}\Delta^{-\frac{q}{r}}w\right)^{\frac{1}{q}} \\
		& \hspace{4cm} \times
		\left(\int_{x_{k-1}}^t\sigma(s)U(s)^{\frac{pr}{p-r}}\esup_{\tau\in(s,t)}\Delta(\tau)^{\frac{p}{p-r}}U(\tau)^{-\frac{pr}{p-r}}\, ds\right)^{\frac{p-r}{pr}}
		\\
		&\quad +   \sup_{\substack{N+2\le k\le M\\k\in\mathcal{Z}_2}} \esup_{t\in(x_{k-1},x_k)} \left(\int_t^{x_k}\Delta^{-\frac{q}{r}}w\right)^{\frac{1}{q}} \\
		& \hspace{4cm} \times
		\left(\int_{x_{k-1}}^t\sigma(s)U(s)^{\frac{pr}{p-r}}\esup_{\tau\in(s,t)}\Delta(\tau)^{\frac{p}{p-r}}U(\tau)^{-\frac{pr}{p-r}}\, ds\right)^{\frac{p-r}{pr}}.
	\end{align*}
	Observe that for $h(s)= U(s)\esup_{\tau\in(s,t)}\Delta(\tau)^{\frac{1}{r}} U(\tau)^{-1} \in Q_U(0,t)$, if $k\in \mathcal{Z}_1$, by~\eqref{EQ:antid_lemma_1}, we have
	\begin{align*}
		&\int_{x_{k-1}}^t\sigma(s)U(s)^{\frac{pr}{p-r}}\esup_{\tau\in(s,t)}\Delta(\tau)^{\frac{p}{p-r}}U(\tau)^{-\frac{pr}{p-r}}\, ds
		\ls \varphi(t)^{-\frac{r}{p-r}}\Delta(t)^{\frac{p}{p-r}}
	\end{align*}
	also, if $k\in \mathcal{Z}_2$, we have by~\eqref{EQ:antid_lemma_2}, 
	\begin{align*}
		&\int_{x_{k-1}}^t\sigma(s)U(s)^{\frac{pr}{p-r}}\esup_{\tau\in(s,t)}\Delta(\tau)^{\frac{p}{p-r}}U(\tau)^{-\frac{pr}{p-r}}\, ds \\
		&\hspace{2cm}
		\ls \varphi(x_{k-1})^{-\frac{r}{p-r}}U(x_{k-1})^{\frac{pr}{p-r}}\esup_{\tau\in(x_{k-1},t)}\Delta(\tau)^{\frac{p}{p-r}}U(\tau)^{-\frac{pr}{p-r}}.
	\end{align*}
	Moreover, recall that, if $N > - \infty$, then $N+1\in\mathcal Z_2$ (see the proof of Lemma~\ref{TH:antid_lemma1}). Consequently, for every $t\in(0,x_{N+1}]$, by~ \eqref{EQ:antid_lemma_3}, we have
	\begin{align*}
		\int_{0}^t\sigma(s)U(s)^{\frac{pr}{p-r}}\esup_{\tau\in(s,t)}\Delta(\tau)^{\frac{p}{p-r}}U(\tau)^{-\frac{pr}{p-r}}\, ds  \lesssim 
		\esup_{\tau\in(0,t)} \Delta(\tau)^{\frac{p}{p-r}} \varphi(\tau)^{-\frac{r}{p-r}}.
	\end{align*}
	Then,
	\begin{align*}
		\III &\ls \esup_{t\in(0,x_{N+1})}\left(\int_{t}^{x_{N+1}}\Delta^{-\frac{q}{r}}w\right)^{\frac{1}{q}} \esup_{\tau\in(0,t)}\Delta(\tau)^{\frac{1}{r}}\varphi(\tau)^{-\frac{1}{p}}
		\\
		&\quad + \sup_{\substack{N+2\le k\le M\\k\in\mathcal{Z}_1}} \esup_{t\in(x_{k-1},x_k)} \left(\int_t^{x_k}\Delta^{-\frac{q}{r}}w\right)^{\frac{1}{q}}
		\varphi(t)^{-\frac{1}{p}}\Delta(t)^{\frac{1}{r}}
		\\
		&\quad +   \sup_{\substack{N+2\le k\le M\\k\in\mathcal{Z}_2}} \esup_{t\in(x_{k-1},x_k)} \left(\int_t^{x_k}\Delta^{-\frac{q}{r}}w\right)^{\frac{1}{q}}
		\varphi(x_{k-1})^{-\frac{1}{p}}U(x_{k-1}) \\
		& \hspace{2cm} \times\esup_{\tau\in(x_{k-1},t)}\Delta(\tau)^{\frac{1}{r}}U(\tau)^{-1}\\
		&\leq \esup_{t\in(0,x_{N+1})}\left(\int_{t}^{x_{N+1}}\Delta^{-\frac{q}{r}}w\right)^{\frac{1}{q}} \esup_{\tau\in(0,t)}\Delta(\tau)^{\frac{1}{r}}\varphi(\tau)^{-\frac{1}{p}}
		\\
		&\quad + \sup_{\substack{N+2\le k\le M\\k\in\mathcal{Z}_1}} \esup_{t\in(x_{k-1},x_k)} \left(\int_t^{x_k}\Delta^{-\frac{q}{r}}w\right)^{\frac{1}{q}}
		\varphi(t)^{-\frac{1}{p}}\Delta(t)^{\frac{1}{r}}
		\\
		&\quad +   \sup_{\substack{N+2\le k\le M\\k\in\mathcal{Z}_2}} \esup_{t\in(x_{k-1},x_k)} \left(\int_t^{x_k}\Delta^{-\frac{q}{r}}w\right)^{\frac{1}{q}}
		\esup_{\tau\in(x_{k-1},t)}\Delta(\tau)^{\frac{1}{r}} \varphi(\tau)^{-\frac{1}{p}}. 
	\end{align*}
	Since $N+1\in \mathcal{Z}_2$, we have
	\begin{align*}
		\III &\ls \sup_{\substack{N+2\le k\le M\\k\in\mathcal{Z}_1}} \esup_{t\in(x_{k-1},x_k)} \left(\int_t^{x_k}\Delta^{-\frac{q}{r}}w\right)^{\frac{1}{q}}
		\varphi(t)^{-\frac{1}{p}}\Delta(t)^{\frac{1}{r}} \nonumber
		\\
		&\quad +   \sup_{\substack{N+1\le k\le M\\k\in\mathcal{Z}_2}} \esup_{t\in(x_{k-1},x_k)} \left(\int_t^{x_k}\Delta^{-\frac{q}{r}}w\right)^{\frac{1}{q}}
		\esup_{\tau\in(x_{k-1},t)}\Delta(\tau)^{\frac{1}{r}}\varphi(\tau)^{-\frac{1}{p}}\nonumber
		\\
		&\le \sup_{N+1\le k\le M} \esup_{t\in(x_{k-1},x_k)} \left(\int_t^{x_k}\Delta^{-\frac{q}{r}}w\right)^{\frac{1}{q}}
		\esup_{\tau\in(x_{k-1},t)} \varphi(\tau)^{-\frac{1}{p}}\Delta(\tau)^{\frac{1}{r}}.
	\end{align*}
	Then by using \eqref{int-cut-D}, we get
	\begin{align*}
		\III &\ls   \sup_{N+1\le k\le M} \esup_{t\in(x_{k-1},x_k)} \left(\int_t^{x_k}\Delta^{-\frac{q}{r}}w\right)^{\frac{1}{q}}
		\esup_{\tau\in(x_{k-1},t)} \varphi(\tau)^{-\frac{1}{p}}\left(\int_{x_{k-1}}^{\tau}\delta\right)^{\frac{1}{r}}\nonumber
		\\
		&\quad +  \sup_{N+2\le k\le M}\varphi(x_{k-1})^{-\frac{1}{p}}\Delta(x_{k-1})^{\frac{1}{r}}\left(\int_{x_{k-1}}^{x_k}\Delta^{-\frac{q}{r}}w\right)^{\frac{1}{q}}\\
		& \le C_{1,1} +  \sup_{N+2\le k\le M}\varphi(x_{k-1})^{-\frac{1}{p}}\Delta(x_{k-1})^{\frac{1}{r}}\left(\int_{x_{k-1}}^{L}\Delta^{-\frac{q}{r}}w\right)^{\frac{1}{q}}.
	\end{align*}
	Finally, \eqref{C-31-estimate-new} combined with \eqref{C-31<C-33} yields
	\begin{align*}
		\III &\ls   C_{1,1}+C_{3,3}.
	\end{align*}  
	For future reference, note that we have shown that
	\begin{equation}\label{4}
		\sup_{N+1\le k\le M} \esup_{t\in(x_{k-1},x_k)} \left(\int_t^{x_k}\Delta^{-\frac{q}{r}}w\right)^{\frac{1}{q}}
		\esup_{\tau\in(x_{k-1},t)} \varphi(\tau)^{-\frac{1}{p}}\Delta(\tau)^{\frac{1}{r}}\ls C_{1,1}+C_{3,3}.
	\end{equation}
	Thus, we have obtained
	\begin{equation*}
		B_4 \ls \I+\II+\III \ls C_{1,1}+C_{3,3}.
	\end{equation*}
	Hence, putting all things together, we have
	\begin{equation*}
		C_{1,1}+C_{1,2}+C_{3,3}+C_{4,1} \approx B_1+B_2+B_4.
	\end{equation*}
	
	\rm{(iv)} By Theorem~\ref{thm:main_discretization}, we have
	\begin{equation*}
		C\approx C_{1,2}+C_{1,3}+C_{3,4}+C_{4,1}.
	\end{equation*}
	Moreover, thanks to \eqref{EQ:C12lesB1}, \eqref{EQ:C13lessB1+B3} and \eqref{EQ:C41lesB1}, to establish the desired upper bound on $C$, it is sufficient to show that $C_{3,4}\lesssim B_2 + B_ 3 + B_5$. To this end, note that
	\begin{equation}\label{E:C3,4-1}
		\begin{split}
			C_{3,4}
			&=\sup_{N+1\leq k \leq M-1} \bigg(\int_{x_k}^L \Delta^{-\frac{q}{r}} w \bigg)^{\frac{1}{q}} \\
			&\hspace{2cm}\times\bigg(\sum_{i=N+1}^k \bigg(\int_{x_{i-1}}^{x_i} \bigg(\int_{x_{i-1}}^t \delta\bigg)^{\frac{r}{1-r}} \delta(t) \varphi(t)^{-\frac{r}{p(1-r)}}dt \bigg)^{\frac{p(1-r)}{p-r}}  \bigg)^{\frac{p-r}{pr}}\\
			& \approx \bigg(\int_{x_{N+1}}^L \Delta^{-\frac{q}{r}} w \bigg)^{\frac{1}{q}}  \bigg(\int_{x_{N}}^{x_{N+1}} \bigg(\int_{x_{N}}^t \delta\bigg)^{\frac{r}{1-r}} \delta(t) \varphi(t)^{-\frac{r}{p(1-r)}}dt  \bigg)^{\frac{1-r}{r}}\\
			&+ \bigg(\int_{x_{N+2}}^L \Delta^{-\frac{q}{r}} w \bigg)^{\frac{1}{q}}  \bigg(\int_{x_{N}}^{x_{N+1}} \bigg(\int_{x_{N}}^t \delta\bigg)^{\frac{r}{1-r}} \delta(t) \varphi(t)^{-\frac{r}{p(1-r)}}dt \bigg)^{\frac{1-r}{r}}  \\
			& + \sup_{N+2\leq k \leq M-1} \bigg(\int_{x_k}^L \Delta^{-\frac{q}{r}} w \bigg)^{\frac{1}{q}} \bigg(\int_{x_{k-1}}^{x_k} \bigg(\int_{x_{k-1}}^t \delta\bigg)^{\frac{r}{1-r}} \delta(t) \varphi(t)^{-\frac{r}{p(1-r)}}dt \bigg)^{\frac{1-r}{r}}\\
			&+ \sup_{N+2\leq k \leq M-1} \bigg(\int_{x_k}^L \Delta^{-\frac{q}{r}} w \bigg)^{\frac{1}{q}} \\
			&\hspace{2cm} \times\bigg(\sum_{i=N+2}^{k-1} \bigg(\int_{x_{i-1}}^{x_i} \bigg(\int_{x_{i-1}}^t \delta\bigg)^{\frac{r}{1-r}} \delta(t) \varphi(t)^{-\frac{r}{p(1-r)}}dt \bigg)^{\frac{p(1-r)}{p-r}}  \bigg)^{\frac{p-r}{pr}}\\
			& \lesssim \sup_{N+1\le k\le M-1} \left(\int_{x_k}^{L}\Delta^{-\frac{q}{r}}w\right)^{\frac{1}{q}}
			\left(\int_{x_{k-1}}^{x_k}\left(\int_{x_{k-1}}^{t}\delta\right)^{\frac{r}{1-r}}\delta(t)\varphi(t)^{-\frac{r}{p(1-r)}}\,dt\right)^{\frac{1-r}{r}}
			\\
			& \quad + \sup_{N+2\le k\le M-1} \left(\int_{x_k}^{L}\Delta^{-\frac{q}{r}}w\right)^{\frac{1}{q}} \\
			& \hspace{2.5cm} \times
			\left(\sum_{i=N+2}^{k-1}\left(\int_{x_{i-1}}^{x_i}\left(\int_{x_{i-1}}^{t}\delta\right)^{\frac{r}{1-r}}\delta(t)\varphi(t)^{-\frac{r}{p(1-r)}}\,dt\right)^{\frac{p(1-r)}{p-r}}\right)^{\frac{p-r}{pr}}
			\\
			&= C_{3,2} + \sup_{N+2\le k\le M-1} \left(\int_{x_k}^{L}\Delta^{-\frac{q}{r}}w\right)^{\frac{1}{q}} \\
			& \hspace{2.5cm} \times
			\left(\sum_{i=N+2}^{k-1}\left(\int_{x_{i-1}}^{x_i}\left(\int_{x_{i-1}}^{t}\delta\right)^{\frac{r}{1-r}}\delta(t)\varphi(t)^{-\frac{r}{p(1-r)}}\,dt\right)^{\frac{p(1-r)}{p-r}}\right)^{\frac{p-r}{pr}}\\
			&\ls B_2+B_3+\sup_{N+2\le k\le M-1} \left(\int_{x_k}^{L}\Delta^{-\frac{q}{r}}w\right)^{\frac{1}{q}} \\
			& \hspace{2.5cm} \times
			\left(\sum_{i=N+2}^{k-1}\left(\int_{x_{i-1}}^{x_i}\left(\int_{x_{i-1}}^{t}\delta\right)^{\frac{r}{1-r}}\delta(t)\varphi(t)^{-\frac{r}{p(1-r)}}\,dt\right)^{\frac{p(1-r)}{p-r}}\right)^{\frac{p-r}{pr}}
		\end{split}
	\end{equation}
	where we used \eqref{EQ:C32lessB2+B3} in the last inequality.
	For any $k\in\Z$ satisfying $N+2\le k\le M-1$, we have by \eqref{phi-Z1} and \eqref{Uphi-Z2},
	\begin{align*}
		&\sum_{i=N+2}^{k-1}\left(\int_{x_{i-1}}^{x_i}\left(\int_{x_{i-1}}^{t}\delta\right)^{\frac{r}{1-r}}\delta(t)\varphi(t)^{-\frac{r}{p(1-r)}}\,dt\right)^{\frac{p(1-r)}{p-r}} \\
		& \quad \le \sum_{i=N+2}^{k-1}\left(\int_{x_{i-1}}^{x_i}\Delta^{\frac{r}{1-r}}\delta\varphi^{-\frac{r}{p(1-r)}}\right)^{\frac{p(1-r)}{p-r}}
		\\
		&\quad \approx \sum_{\substack{i\in\mathcal{Z}_1\\N+2\le i\le k-1}}\varphi(x_i)^{-\frac{r}{p-r}}\left(\int_{x_{i-1}}^{x_i}\Delta^{\frac{r}{1-r}}\delta\right)^{\frac{p(1-r)}{p-r}}
		\\
		&\qquad + \sum_{\substack{i\in\mathcal{Z}_2\\N+2\le i\le k-1}}\varphi(x_{i-1})^{-\frac{r}{p-r}}U(x_{i-1})^{\frac{pr}{p-r}}\left(\int_{x_{i-1}}^{x_i}\Delta^{\frac{r}{1-r}}\delta U^{-\frac{r}{1-r}}\right)^{\frac{p(1-r)}{p-r}}
		\\
		&\quad \le \sum_{\substack{i\in\mathcal{Z}_1\\N+2\le i\le k-1}}\varphi(x_i)^{-\frac{r}{p-r}}\Delta(x_i)^{\frac{p}{p-r}} \\
		& \qquad +
		\sum_{\substack{i\in\mathcal{Z}_2\\N+1\le i\le k-2}}\varphi(x_{i})^{-\frac{r}{p-r}}U(x_{i})^{\frac{pr}{p-r}}\left(\int_{x_{i}}^{x_{i+1}}\Delta^{\frac{r}{1-r}}\delta U^{-\frac{r}{1-r}}\right)^{\frac{p(1-r)}{p-r}}
		\\
		&\ls \sum_{i=N+1}^{k-1}\varphi(x_i)^{-\frac{r}{p-r}}
		\left[\Delta(x_i)^{\frac{1}{r}}+U(x_i)\left(\int_{x_{i}}^{x_{k}}\Delta^{\frac{r}{1-r}}\delta U^{-\frac{r}{1-r}}\right)^{\frac{1-r}{r}}\right]^{\frac{pr}{p-r}}
		\\
		& \approx \sum_{i=N+1}^{k-1}\varphi(x_i)^{-\frac{r}{p-r}}
		\left(\int_{0}^{x_k}\Delta(s)^{\frac{r}{1-r}}\delta(s)U(s)^{-\frac{r}{1-r}}\min\left\{U(s)^{\frac{r}{1-r}},U(x_i)^{\frac{r}{1-r}}\right\}\,ds\right)^{\frac{p(1-r)}{p-r}}
		\\
		&\ls \int_{0}^{x_k}\sigma(t)\left(\int_{0}^{x_k}\Delta(s)^{\frac{r}{1-r}}\delta(s)U(s)^{-\frac{r}{1-r}}\min\left\{U(s)^{\frac{r}{1-r}},U(t)^{\frac{r}{1-r}}\right\}\,ds\right)^{\frac{p(1-r)}{p-r}}\,dt,
	\end{align*}
	where we used \eqref{R3} in the last inequality. Plugging this into~\eqref{E:C3,4-1}, we obtain
	\begin{align}
		C_{3,4}&\ls B_2 + B_3 \nonumber\\
		& \quad + \sup_{N+2\le k\le M-1}\left(\int_{x_k}^{L}\Delta^{-\frac{q}{r}}w\right)^{\frac{1}{q}}
		\left(\int_{0}^{x_k}\sigma(t)\left(\int_{0}^{x_k}\Delta(s)^{\frac{r}{1-r}}\delta(s)U(s)^{-\frac{r}{1-r}} \right. \right. \nonumber\\
		& \hspace{5cm} \times \left.\left.\min\left\{U(s)^{\frac{r}{1-r}},U(t)^{\frac{r}{1-r}}\right\}\,ds\right)^{\frac{p(1-r)}{p-r}}\,dt\right)^{\frac{p-r}{pr}}
		\nonumber\\
		& \ls B_2 + B_3
		\nonumber\\
		& \quad + \sup_{N+2\le k\le M-1}\esup_{x\in(x_k,x_{k+1})}\left(\int_{x}^{L}\Delta^{-\frac{q}{r}}w\right)^{\frac{1}{q}} \left(\int_{0}^{x}\sigma(t)\left(\int_{0}^{x}\Delta(s)^{\frac{r}{1-r}}\delta(s)U(s)^{-\frac{r}{1-r}}\right. \right. \nonumber \\
		&\hspace{5cm}\times
		\left. \left. \min\left\{U(s)^{\frac{r}{1-r}},U(t)^{\frac{r}{1-r}}\right\}\,ds\right)^{\frac{p(1-r)}{p-r}}\,dt\right)^{\frac{p-r}{pr}}
		\nonumber\\
		& \ls B_2+B_3+B_5. \label{EQ:C34lessB2+B3+B5}
	\end{align}
	Altogether, we arrive at
	\begin{equation*}
		C_{1,2}+C_{1,3}+C_{3,4}+C_{4,1} \ls B_1+B_2+B_3+B_5.
	\end{equation*}
	
	As for the opposite inequality, note that
	\begin{equation}\label{C-32<C-34-new}
		C_{3,2} \leq C_{3,4}.
	\end{equation} 
	Then, owing to \eqref{EQ:B1+B2UpperBoundCase1} combined with \eqref{EQ:C11lessC13}, \eqref{EQ:C31lessC32} and \eqref{C-32<C-34-new}, we have
	\begin{equation*}
		B_1 + B_2 \lesssim C_{1,2} + C_{1,3} + C_{3,2} + C_{4,1} \leq C_{1,2} + C_{1,3} + C_{3,4} + C_{4,1}.
	\end{equation*}
	Moreover, thanks to \eqref{EQ:antid_ii_B3} and \eqref{C-32<C-34-new}, we have
	\begin{equation*}
		B_3  \lesssim C_{1,3}+C_{3,2} \leq  C_{1,3} + C_{3,4}.
	\end{equation*}
	Consequently,
	\begin{equation*}
		B_1+B_2+B_3 \ls  C_{1,2}+C_{1,3}+C_{3,4}+C_{4,1}.
	\end{equation*}
	
	As for $B_5$, we have
	\begin{align*}
		B_5 &\approx \sup_{N+1\le k\le M}\esup_{t\in(x_{k-1},x_{k})} \left(\int_{t}^{L}\Delta^{-\frac{q}{r}}w\right)^{\frac{1}{q}} \\
		& \hspace{2cm} \times
		\left(\int_{0}^{t}\sigma(s)\left[\Delta(s)^{\frac{1}{r}}+U(s)\left(\int_{s}^{t}\Delta^{\frac{r}{1-r}}\delta U^{-\frac{r}{1-r}}\right)^{\frac{1-r}{r}}\right]^{\frac{pr}{p-r}} \,ds\right)^{\frac{p-r}{pr}}.
	\end{align*}
	Then, similar to the previous cases decomposing the integral $\int_{t}^{L}$ into the sum $\int_{t}^{x_{k}}+\int_{x_{k}}^{L}$, we get
	\begin{align*}
		B_5& \approx
		\sup_{N+1\le k\le M}\esup_{t\in(x_{k-1},x_{k})} \left(\int_{t}^{x_k}\Delta^{-\frac{q}{r}}w\right)^{\frac{1}{q}} \\
		& \hspace{2cm} \times
		\left(\int_{0}^{t}\sigma(s)\left[\Delta(s)^{\frac{1}{r}}+U(s)\left(\int_{s}^{t}\Delta^{\frac{r}{1-r}}\delta U^{-\frac{r}{1-r}}\right)^{\frac{1-r}{r}}\right]^{\frac{pr}{p-r}} \,ds\right)^{\frac{p-r}{pr}}
		\\
		&\quad +
		\sup_{N+1\le k\le M-1}\left(\int_{x_k}^{L}\Delta^{-\frac{q}{r}}w\right)^{\frac{1}{q}} \\
		& \hspace{2cm} \times
		\left(\int_{0}^{x_k}\sigma(s)\left[\Delta(s)^{\frac{1}{r}}+U(s)\left(\int_{s}^{x_k}\Delta^{\frac{r}{1-r}}\delta U^{-\frac{r}{1-r}}\right)^{\frac{1-r}{r}}\right]^{\frac{pr}{p-r}} \,ds\right)^{\frac{p-r}{pr}}.
	\end{align*}
	Next, decomposing the integral $\int_0^{t}$ into the sum $\int_0^{x_{k-1}}+ \int_{x_{k-1}}^t$, we have
	\begin{align*}
		B_5& \approx
		\sup_{N+2\le k\le M}\esup_{t\in(x_{k-1},x_{k})} \left(\int_{t}^{x_k}\Delta^{-\frac{q}{r}}w\right)^{\frac{1}{q}} \\
		& \hspace{2cm} \times
		\left(\int_{0}^{x_{k-1}}\sigma(s)\left[\Delta(s)^{\frac{1}{r}}+U(s)\left(\int_{s}^{t}\Delta^{\frac{r}{1-r}}\delta U^{-\frac{r}{1-r}}\right)^{\frac{1-r}{r}}\right]^{\frac{pr}{p-r}} \,ds\right)^{\frac{p-r}{pr}}
		\\
		& +
		\sup_{N+1\le k\le M}\esup_{t\in(x_{k-1},x_{k})} \left(\int_{t}^{x_k}\Delta^{-\frac{q}{r}}w\right)^{\frac{1}{q}} \\
		& \hspace{2cm} \times
		\left(\int_{x_{k-1}}^{t}\sigma(s)\left[\Delta(s)^{\frac{1}{r}}+U(s)\left(\int_{s}^{t}\Delta^{\frac{r}{1-r}}\delta U^{-\frac{r}{1-r}}\right)^{\frac{1-r}{r}}\right]^{\frac{pr}{p-r}} \,ds\right)^{\frac{p-r}{pr}}
		\\
		& +
		\sup_{N+1\le k\le M-1} \left(\int_{x_k}^L\Delta^{-\frac{q}{r}}w\right)^{\frac{1}{q}} \\
		& \hspace{2cm} \times
		\left(\int_0^{x_{k}}\sigma(s)\left[\Delta(s)^{\frac{1}{r}}+U(s)\left(\int_{s}^{x_k}\Delta^{\frac{r}{1-r}}\delta U^{-\frac{r}{1-r}}\right)^{\frac{1-r}{r}}\right]^{\frac{pr}{p-r}} \,ds\right)^{\frac{p-r}{pr}}\\
		&  \approx
		\sup_{N+2\le k\le M} \left(\int_{x_{k-1}}^{x_k}\Delta^{-\frac{q}{r}}w\right)^{\frac{1}{q}}
		\left(\int_{0}^{x_{k-1}}\sigma(s)\Delta(s)^{\frac{p}{p-r}}\,ds\right)^{\frac{p-r}{pr}}\\
		& + \sup_{N+2\le k\le M}\esup_{t\in(x_{k-1},x_{k})} \left(\int_{t}^{x_k}\Delta^{-\frac{q}{r}}w\right)^{\frac{1}{q}} \\
		& \hspace{2cm} \times
		\left(\int_{0}^{x_{k-1}}\sigma(s) U(s)^{\frac{pr}{p-r}} \left(\int_{s}^{t}\Delta^{\frac{r}{1-r}}\delta U^{-\frac{r}{1-r}}\right)^{\frac{p(1-r)}{p-r}} \,ds\right)^{\frac{p-r}{pr}}
		\\
		& +
		\sup_{N+1\le k\le M}\esup_{t\in(x_{k-1},x_{k})} \left(\int_{t}^{x_k}\Delta^{-\frac{q}{r}}w\right)^{\frac{1}{q}} \\
		& \hspace{2cm} \times
		\left(\int_{x_{k-1}}^{t}\sigma(s)\left[\Delta(s)^{\frac{1}{r}}+U(s)\left(\int_{s}^{t}\Delta^{\frac{r}{1-r}}\delta U^{-\frac{r}{1-r}}\right)^{\frac{1-r}{r}}\right]^{\frac{pr}{p-r}} \,ds\right)^{\frac{p-r}{pr}}
		\\
		& +
		\sup_{N+1\le k\le M-1} \left(\int_{x_k}^L\Delta^{-\frac{q}{r}}w\right)^{\frac{1}{q}} \\
		& \hspace{2cm} \times
		\left(\int_0^{x_{k}}\sigma(s)\left[\Delta(s)^{\frac{1}{r}}+U(s)\left(\int_{s}^{x_k}\Delta^{\frac{r}{1-r}}\delta U^{-\frac{r}{1-r}}\right)^{\frac{1-r}{r}}\right]^{\frac{pr}{p-r}} \,ds\right)^{\frac{p-r}{pr}}.
	\end{align*}
	Now, in the second term, decomposing the integral $\int_s^t$ into sum $\int_s^{x_{k-1}} + \int_{x_{k-1}}^t$, we obtain 
	\begin{align*}
		B_5&\approx
		\sup_{N+2\le k\le M}\left(\int_{x_{k-1}}^{x_k}\Delta^{-\frac{q}{r}}w\right)^{\frac{1}{q}}
		\left(\int_{0}^{x_{k-1}}\sigma\Delta^{\frac{p}{p-r}}\right)^{\frac{p-r}{pr}}
		\\
		&\quad +
		\sup_{N+2\le k\le M} \left(\int_{x_{k-1}}^{x_k}\Delta^{-\frac{q}{r}}w\right)^{\frac{1}{q}} \\
		& \hspace{2cm} \times
		\left(\int_{0}^{x_{k-1}}\sigma(s)U(s)^{\frac{pr}{p-r}}
		\left(\int_s^{x_{k-1}}\Delta^{\frac{r}{1-r}}\delta U^{-\frac{r}{1-r}}\right)^{\frac{p(1-r)}{p-r}}\,ds\right)^{\frac{p-r}{pr}}
		\\
		&\quad +
		\sup_{N+2\le k\le M}\esup_{t\in(x_{k-1},x_{k})} \left(\int_{t}^{x_k}\Delta^{-\frac{q}{r}}w\right)^{\frac{1}{q}} \\
		& \hspace{2cm} \times
		\left(\int_0^{x_{k-1}}\sigma U^{\frac{pr}{p-r}}\right)^{\frac{p-r}{pr}}
		\left(\int_{x_{k-1}}^{t}\Delta^{\frac{r}{1-r}}\delta U^{-\frac{r}{1-r}}\right)^{\frac{1-r}{r}}
		\\
		&\quad +
		\sup_{N+1\le k\le M}\esup_{t\in(x_{k-1},x_{k})} \left(\int_{t}^{x_k}\Delta^{-\frac{q}{r}}w\right)^{\frac{1}{q}}\\
		& \hspace{2cm} \times
		\left(\int_{x_{k-1}}^{t}\sigma(s)\left[\Delta(s)^{\frac{1}{r}}+U(s)\left(\int_{s}^{t}\Delta^{\frac{r}{1-r}}\delta U^{-\frac{r}{1-r}}\right)^{\frac{1-r}{r}}\right]^{\frac{pr}{p-r}} \,ds\right)^{\frac{p-r}{pr}}
		\\
		&\quad +
		\sup_{N+1\le k\le M-1}\left(\int_{x_k}^{L}\Delta^{-\frac{q}{r}}w\right)^{\frac{1}{q}} \\
		& \hspace{2cm} \times
		\left(\int_{0}^{x_{k}}\sigma(s)\left[\Delta(s)^{\frac{1}{r}}+U(s)\left(\int_{s}^{x_k}\Delta^{\frac{r}{1-r}}\delta U^{-\frac{r}{1-r}}\right)^{\frac{1-r}{r}}\right]^{\frac{pr}{p-r}} \,ds\right)^{\frac{p-r}{pr}}.
	\end{align*}
	Observe that, reindexing $k\mapsto k+1$, 
	\begin{align}\label{New-B5}
		&\sup_{N+2\le k\le M}\left(\int_{x_{k-1}}^{x_k}\Delta^{-\frac{q}{r}}w\right)^{\frac{1}{q}}
		\left(\int_{0}^{x_{k-1}}\sigma\Delta^{\frac{p}{p-r}}\right)^{\frac{p-r}{pr}}
		\nonumber\\
		&\quad +
		\sup_{N+2\le k\le M} \left(\int_{x_{k-1}}^{x_k}\Delta^{-\frac{q}{r}}w\right)^{\frac{1}{q}} \nonumber\\
		& \hspace{2cm} \times
		\left(\int_{0}^{x_{k-1}}\sigma(s)U(s)^{\frac{pr}{p-r}}
		\left(\int_s^{x_{k-1}}\Delta^{\frac{r}{1-r}}\delta U^{-\frac{r}{1-r}}\right)^{\frac{p(1-r)}{p-r}}\,ds\right)^{\frac{p-r}{pr}}\nonumber
		\\
		& \approx \sup_{N+1\le k\le M-1}\left(\int_{x_k}^{x_{k+1}}\Delta^{-\frac{q}{r}}w\right)^{\frac{1}{q}}\nonumber \\
		& \hspace{2cm} \times
		\left(\int_{0}^{x_{k}}\sigma(s)\left[\Delta(s)^{\frac{1}{r}}+U(s)\left(\int_{s}^{x_k}\Delta^{\frac{r}{1-r}}\delta U^{-\frac{r}{1-r}}\right)^{\frac{1-r}{r}}\right]^{\frac{pr}{p-r}} \,ds\right)^{\frac{p-r}{pr}}.
	\end{align}
	Thus, we obtain
	\begin{align*}
		B_5&\lesssim \sup_{N+1\le k\le M-1}\left(\int_{x_{k}}^{L}\Delta^{-\frac{q}{r}}w\right)^{\frac{1}{q}} \\
		& \hspace{2cm} \times
		\left(\int_{0}^{x_{k}}\sigma(s)\left[\Delta(s)^{\frac{1}{r}}+U(s)\left(\int_{s}^{x_k}\Delta^{\frac{r}{1-r}}\delta U^{-\frac{r}{1-r}}\right)^{\frac{1-r}{r}}\right]^{\frac{pr}{p-r}} \,ds\right)^{\frac{p-r}{pr}}
		\\
		&\quad +
		\sup_{N+2\le k\le M} \left(\int_0^{x_{k-1}}\sigma U^{\frac{pr}{p-r}}\right)^{\frac{p-r}{pr}} \\
		& \hspace{2cm} \times\esup_{t\in(x_{k-1},x_{k})}\left(\int_{t}^{x_k}\Delta^{-\frac{q}{r}}w\right)^{\frac{1}{q}}
		\left(\int_{x_{k-1}}^{t}\Delta^{\frac{r}{1-r}}\delta U^{-\frac{r}{1-r}}\right)^{\frac{1-r}{r}}
		\\
		&\quad +
		\sup_{N+1\le k\le M}\esup_{t\in(x_{k-1},x_{k})}\left(\int_{t}^{x_k}\Delta^{-\frac{q}{r}}w\right)^{\frac{1}{q}} \\
		& \hspace{2cm} \times
		\left(\int_{x_{k-1}}^{t}\sigma(s)\left[\Delta(s)^{\frac{1}{r}}+U(s)\left(\int_{s}^{t}\Delta^{\frac{r}{1-r}}\delta U^{-\frac{r}{1-r}}\right)^{\frac{1-r}{r}}\right]^{\frac{pr}{p-r}} \,ds\right)^{\frac{p-r}{pr}}
		\\
		& =: \I+\II+\III.
	\end{align*}
	
	In view of~\eqref{R4}, we have
	\begin{align*}
		\I&  \approx\sup_{N+1\le k\le M-1}\left(\int_{x_{k}}^{L}\Delta^{-\frac{q}{r}}w\right)^{\frac{1}{q}} \bigg(\int_{0}^{x_{k}}\sigma(s)\\
		&\hspace{1cm}\times
		\bigg(\int_0^{x_k} \Delta(\tau)^{\frac{r}{1-r}}\delta(\tau) U(\tau)^{-\frac{r}{1-r}} \min\{U(s)^{\frac{r}{1-r}}, U(\tau)^{\frac{r}{1-r}}\}d\,\tau\bigg)^{\frac{p(1-r)}{p-r}}
		\,ds\bigg)^{\frac{p-r}{pr}}\\
		& \ls \sup_{N+1\le k\le M-1}\left(\int_{x_{k}}^{L}\Delta^{-\frac{q}{r}}w\right)^{\frac{1}{q}}
		\left(\sum_{i=N+1}^{k}\varphi(x_i)^{-\frac{r}{p-r}}\right.\\
		& \hspace{3cm}\times\left[\Delta(x_i)^{\frac{1}{r}}+ \left. U(x_i)\left(\int_{x_i}^{x_k}\Delta^{\frac{r}{1-r}}\delta U^{-\frac{r}{1-r}}\right)^{\frac{1-r}{r}}\right]^{\frac{pr}{p-r}} \right)^{\frac{p-r}{pr}}
		\\
		&\quad + \left(\int_{x_{N+1}}^{L}\Delta^{-\frac{q}{r}}w\right)^{\frac{1}{q}}\esup_{t\in(0,x_{N+1})}\varphi(t)^{-\frac{1}{p}}\Delta(t)^{\frac{1}{r}}
		\\
		&\quad + \left(\int_{x_{N+1}}^{L}\Delta^{-\frac{q}{r}}w\right)^{\frac{1}{q}}\esup_{t\in(0,x_{N+1})}\varphi(t)^{-\frac{1}{p}}U(t)\left(\int_{t}^{x_{N+1}}\Delta^{\frac{r}{1-r}}\delta U^{-\frac{r}{1-r}}\right)^{\frac{1-r}{r}}.
	\end{align*}
	Note that, owing to \eqref{EQ:C33LBCase3} and \eqref{sup-int},
	\begin{equation*}
		\left(\int_{x_{N+1}}^{L}\Delta^{-\frac{q}{r}}w\right)^{\frac{1}{q}}\esup_{t\in(0,x_{N+1})}\varphi(t)^{-\frac{1}{p}}\Delta(t)^{\frac{1}{r}} \ls C_{3,3}\ls C_{3,4}.
	\end{equation*}
	Moreover, since $U\varphi^{-\frac{1}{p}}$ is increasing, we have
	\begin{align*}
		& \left(\int_{x_{N+1}}^{L}\Delta^{-\frac{q}{r}}w\right)^{\frac{1}{q}} \esup_{t\in(0,x_{N+1})}\varphi(t)^{-\frac{1}{p}}U(t)\left(\int_{t}^{x_{N+1}}\Delta^{\frac{r}{1-r}}\delta U^{-\frac{r}{1-r}}\right)^{\frac{1-r}{r}}
		\\
		& \leq \left(\int_{x_{N+1}}^{L}\Delta^{-\frac{q}{r}}w\right)^{\frac{1}{q}}
		\left(\int_{0}^{x_{N+1}}\Delta^{\frac{r}{1-r}}\delta\varphi^{-\frac{r}{p(1-r)}}\right)^{\frac{1-r}{r}}
		\\
		& \leq C_{3,4}.
	\end{align*}
	Then, 
	\begin{align*}
		\I \ls C_{3,4} &+ \sup_{N+1\le k\le M-1}\bigg(\int_{x_{k}}^{L}\Delta^{-\frac{q}{r}}w\bigg)^{\frac{1}{q}}
		\bigg(\sum_{i=N+1}^{k}\varphi(x_i)^{-\frac{r}{p-r}}\Delta(x_i)^{\frac{p}{p-r}}\bigg)^{\frac{p-r}{pr}}
		\\
		& +\sup_{N+1\le k\le M-1}\bigg(\int_{x_{k}}^{L}\Delta^{-\frac{q}{r}}w\bigg)^{\frac{1}{q}} \\
		& \qquad \times
		\bigg(\sum_{i=N+1}^{k-1} \varphi(x_i)^{-\frac{r}{p-r}} U(x_i)^{\frac{pr}{p-r}}\bigg(\int_{x_i}^{x_k}\Delta^{\frac{r}{1-r}}\delta U^{-\frac{r}{1-r}}\bigg)^{\frac{p(1-r)}{p-r}} \bigg)^{\frac{p-r}{pr}}.
	\end{align*}
	Furthermore, since $\{U^p(x_i)/\varphi(x_i)\}_{i=N+1}^{k-1}$ is strongly increasing, using \eqref{EQ:strongly_increasing_sum_sum}, we have
	\begin{align*}
		&\sum_{i=N+1}^{k-1}\varphi(x_i)^{-\frac{r}{p-r}}U(x_i)^{\frac{pr}{p-r}}\left(\int_{x_i}^{x_k}\Delta^{\frac{r}{1-r}}\delta U^{-\frac{r}{1-r}}\right)^{\frac{p(1-r)}{p-r}}
		\\
		& = \sum_{i=N+1}^{k-1}\varphi(x_i)^{-\frac{r}{p-r}}U(x_i)^{\frac{pr}{p-r}}\left(\sum_{m=i}^{k-1}\int_{x_m}^{x_{m+1}}\Delta^{\frac{r}{1-r}}\delta U^{-\frac{r}{1-r}}\right)^{\frac{p(1-r)}{p-r}}
		\\
		& \approx \sum_{i=N+1}^{k-1}\varphi(x_i)^{-\frac{r}{p-r}}U(x_i)^{\frac{pr}{p-r}}\left(\int_{x_i}^{x_{i+1}}\Delta^{\frac{r}{1-r}}\delta U^{-\frac{r}{1-r}}\right)^{\frac{p(1-r)}{p-r}}.
	\end{align*}
	Next, using the monotonicity of $U^p/\varphi$, reindexing $i \mapsto i-1$, then applying \eqref{EQ:antid_ii_5} with $t=x_i$, we get
	\begin{align*}
		&\sum_{i=N+1}^{k-1}\varphi(x_i)^{-\frac{r}{p-r}}U(x_i)^{\frac{pr}{p-r}}\left(\int_{x_i}^{x_k}\Delta^{\frac{r}{1-r}}\delta U^{-\frac{r}{1-r}}\right)^{\frac{p(1-r)}{p-r}}
		\\
		& \leq \sum_{i=N+2}^{k}\left(\int_{x_{i-1}}^{x_{i}}\Delta^{\frac{r}{1-r}}\delta \varphi^{-\frac{r}{p(1-r)}}\right)^{\frac{p(1-r)}{p-r}}
		\\
		& \ls \sum_{i=N+2}^{k}\Delta(x_i)^{\frac{p}{p-r}}\varphi(x_i)^{-\frac{r}{p-r}} +
		\sum_{i=N+2}^{k}\Delta(x_{i-1})^{\frac{p}{p-r}}\varphi(x_{i-1})^{-\frac{r}{p-r}}
		\\
		& \quad + \sum_{i=N+2}^{k}\left(\int_{x_{i-1}}^{x_{i}}\left(\int_{x_{i-1}}^{\tau}\delta\right)^{\frac{r}{1-r}}\delta(\tau) \varphi(\tau)^{-\frac{r}{p(1-r)}}\,d\tau\right)^{\frac{p(1-r)}{p-r}}.
	\end{align*}
	Consequently, we have
	\begin{align*}
		\I &\ls C_{3,4} + \sup_{N+1\le k\le M-1}\left(\int_{x_{k}}^{L}\Delta^{-\frac{q}{r}}w\right)^{\frac{1}{q}}
		\left(\sum_{i=N+1}^{k} \varphi(x_i)^{-\frac{r}{p-r}} \Delta(x_i)^{\frac{p}{p-r}}\right)^{\frac{p-r}{pr}}
		\\
		& \quad + \sup_{N+1\le k\le M-1} \left(\int_{x_{k}}^{L}\Delta^{-\frac{q}{r}}w\right)^{\frac{1}{q}} \\
		& \hspace{3cm} \times
		\left(\sum_{i=N+2}^{k}\left(\int_{x_{i-1}}^{x_{i}}\left(\int_{x_{i-1}}^{\tau}\delta\right)^{\frac{r}{1-r}}\delta(\tau)\varphi(\tau)^{-\frac{r}{p(1-r)}}\,d\tau\right)^{\frac{p(1-r)}{p-r}}\right)^{\frac{p-r}{pr}}.
	\end{align*}
	Since $\{\varphi(x_i)^{-\frac{r}{p-r}}\}_{i=N+1}^{M-1}$ is strongly decreasing, applying \eqref{EQ:strongly_decreasing_sum_sum},
	\begin{align*}
		\I &  \ls C_{3,4} + \sup_{N+1\le k\le M-1}\left(\int_{x_{k}}^{L}\Delta^{-\frac{q}{r}}w\right)^{\frac{1}{q}}
		\left(\sum_{i=N+1}^{k} \varphi(x_i)^{-\frac{r}{p-r}} \bigg(\int_{x_{i-1}}^{x_i} \delta\bigg)^{\frac{p}{p-r}}\right)^{\frac{p-r}{pr}}
		\\
		& \approx C_{3,4} +C_{2,2}.
	\end{align*}
	
	Now, taking \eqref{C22<C33} and \eqref{C33<C34} into consideration we arrive at
	\begin{equation*}
		I \lesssim C_{3,4}.
	\end{equation*}
	Recall that
	\begin{equation*}
		\int_{0}^{x_{k-1}}\sigma U^{\frac{pr}{p-r}} \ls \varphi(x_{k-1})^{-\frac{r}{p-r}}U(x_{k-1})^{\frac{pr}{p-r}}
		\quad\text{for $k\in\Z, N+2\leq k \leq M$}
	\end{equation*}
	thanks to \eqref{EQ:antid_iii_1}. Next, applying \eqref{EQ:antid_iii_1}, exploiting the monotonicity of $U\varphi^{-\frac{1}{p}}$ and using \eqref{2}  and \eqref{C-32<C-34-new} in turn, we get,
	\begin{align*}
		\II &\ls \sup_{N+2\le k\le M} \varphi(x_{k-1})^{-\frac{1}{p}} U(x_{k-1})\esup_{t\in(x_{k-1},x_{k})}\left(\int_{t}^{x_k}\Delta^{-\frac{q}{r}}w\right)^{\frac{1}{q}}
		\left(\int_{x_{k-1}}^{t}\Delta^{\frac{r}{1-r}}\delta U^{-\frac{r}{1-r}}\right)^{\frac{1-r}{r}}
		\\
		& \leq \sup_{N+2\le k\le M} \esup_{t\in(x_{k-1},x_{k})} \left(\int_{t}^{x_k}\Delta^{-\frac{q}{r}}w\right)^{\frac{1}{q}}
		\left(\int_{x_{k-1}}^{t}\Delta^{\frac{r}{1-r}}\delta \varphi^{-\frac{r}{p(1-r)}}\right)^{\frac{1-r}{r}}
		\\
		& \ls C_{1,3} + C_{3,4}.
	\end{align*}
	For future reference note that, we have shown
	\begin{align}\label{III-new}
		\sup_{N+2\le k\le M} \esup_{t\in(x_{k-1},x_{k})} \left(\int_{t}^{x_k}\Delta^{-\frac{q}{r}}w\right)^{\frac{1}{q}}
		\left(\int_{x_{k-1}}^{t}\Delta^{\frac{r}{1-r}}\delta \varphi^{-\frac{r}{p(1-r)}}\right)^{\frac{1-r}{r}}
		\ls C_{1,3} + C_{3,4}.
	\end{align}
	To find a suitable upper estimate for $\III$, for every $t\in(0,L)$, set
	\begin{equation*}
		\widetilde{h}_t(s) = \Delta(s)^{\frac{1}{r}}+U(s)\left(\int_{s}^{t}\Delta^{\frac{r}{1-r}}\delta U^{-\frac{r}{1-r}}\right)^{\frac{1-r}{r}} \quad\text{for $s\in(0,t)$.}
	\end{equation*}
	Note that $\widetilde{h}_t\in Q_{U}(0,t)$. Finally,
	\begin{align*}
		\III &  = \sup_{N+1\le k\le M}\esup_{t\in(x_{k-1},x_{k})}\bigg(\int_{t}^{x_k}\Delta^{-\frac{q}{r}}w\bigg)^{\frac{1}{q}} \bigg(\int_{x_{k-1}}^{t}\sigma(s)\widetilde{h}_t(s)^{\frac{pr}{p-r}} \,ds\bigg)^{\frac{p-r}{pr}} \\
		& \approx \esup_{t\in(0,x_{N+1})}\left(\int_{t}^{x_{N+1}}\Delta^{-\frac{q}{r}}w\right)^{\frac{1}{q}}
		\left(\int_{0}^{t}\sigma(s)\th_t(s)^{\frac{pr}{p-r}} \,ds\right)^{\frac{p-r}{pr}}
		\\
		&\quad + \sup_{\substack{k\in\mathcal{Z}_1 \\ N+2\le k\le M}}\esup_{t\in(x_{k-1},x_{k})}\left(\int_{t}^{x_{k}}\Delta^{-\frac{q}{r}}w\right)^{\frac{1}{q}}
		\left(\int_{x_{k-1}}^{t}\sigma(s)\th_t(s)^{\frac{pr}{p-r}} \,ds\right)^{\frac{p-r}{pr}}
		\\
		&\quad + \sup_{\substack{k\in\mathcal{Z}_2 \\ N+2\le k\le M}}\esup_{t\in(x_{k-1},x_{k})}\left(\int_{t}^{x_{k}}\Delta^{-\frac{q}{r}}w\right)^{\frac{1}{q}}
		\left(\int_{x_{k-1}}^{t}\sigma(s)\th_t(s)^{\frac{pr}{p-r}} \,ds\right)^{\frac{p-r}{pr}}.
	\end{align*}
	If $k\in\mathcal{Z}_1$, $N+2\leq k\leq M$, then, thanks to \eqref{EQ:antid_lemma_1},
	\begin{equation}\label{result of 4.1}
		\int_{x_{k-1}}^{t}\sigma(s)\th_t(s)^{\frac{pr}{p-r}} \,ds \ls \varphi(t)^{-\frac{r}{p-r}}\Delta(t)^{\frac{p}{p-r}},
	\end{equation}
	while, if $k\in\mathcal{Z}_2$, $N+2\leq k\leq M$, then, owing to \eqref{EQ:antid_lemma_2} and the monotonicity of $U\varphi^{-\frac{1}{p}}$,
	\begin{align}\label{result of 4.2}
		&\int_{x_{k-1}}^{t}\sigma(s)\th_t(s)^{\frac{pr}{p-r}} \,ds  \nonumber\\
		& \qquad \ls \varphi(x_{k-1})^{-\frac{r}{p-r}}
		\left[\Delta(x_{k-1})^{\ir} + U(x_{k-1})\left(\int_{x_{k-1}}^{t}\Delta^{\frac{r}{1-r}}\delta U^{-\frac{r}{1-r}}\right)^{\frac{1-r}{r}}\right]^{\frac{pr}{p-r}}\nonumber\\
		& \qquad \leq \varphi(x_{k-1})^{-\frac{r}{p-r}}
		\Delta(x_{k-1})^{\frac{p}{p-r}} + \left(\int_{x_{k-1}}^{t}\Delta^{\frac{r}{1-r}}\delta \varphi^{-\frac{r}{p(1-r)}}\right)^{\frac{p(1-r)}{p-r}}.
	\end{align}
	Moreover, for every $t\in(0,x_{N+1})$, since $N+1 \in \mathcal{Z}_2$, we have by \eqref{EQ:antid_lemma_3},  \eqref{sup-int} with $k=N+1$ and the monotonicity of $U \varphi^{-\frac{1}{p}}$
	\begin{align}\label{result of 4.3}
		\int_{0}^{t}\sigma(s)\th_t(s)^\frac{pr}{p-r}\,ds &\ls \sup_{y\in (0,t)} \th_t(y)^\frac{pr}{p-r} \varphi(y)^{-\frac{r}{p-r}} \nonumber\\
		& = \sup_{y\in (0,t)} \bigg(\Delta(y)^{\frac{1}{r}}+U(y)\left(\int_{y}^{t}\Delta^{\frac{r}{1-r}}\delta U^{-\frac{r}{1-r}}\right)^{\frac{1-r}{r}}\bigg)^{\frac{pr}{p-r}}\varphi(y)^{-\frac{r}{p-r}} \nonumber\\
		& \lesssim \left(\int_{0}^{t}\Delta^{\frac{r}{1-r}}\delta \varphi^{-\frac{r}{p(1-r)}}\right)^{\frac{p(1-r)}{p-r}}.
	\end{align}
	
	Combining \eqref{result of 4.1}, \eqref{result of 4.2} and \eqref{result of 4.3}, we obtain
	\begin{align*}
		\III & \lesssim \esup_{t\in(0,x_{N+1})}\left(\int_{t}^{x_{N+1}}\Delta^{-\frac{q}{r}}w\right)^{\frac{1}{q}}
		\left(\int_{0}^{t}\Delta^{\frac{r}{1-r}}\delta \varphi^{-\frac{r}{p(1-r)}}\right)^{\frac{1-r}{r}}
		\\
		& \qquad + \sup_{\substack{k\in\mathcal{Z}_1 \\ N+2\le k\le M}}\esup_{t\in(x_{k-1},x_{k})}\left(\int_{t}^{x_{k}}\Delta^{-\frac{q}{r}}w\right)^{\frac{1}{q}}
		\varphi(t)^{-\frac{1}{p}}\Delta(t)^{\frac{1}{r}}
		\\
		& \qquad + \sup_{\substack{k\in\mathcal{Z}_2 \\ N+2\le k\le M}} \left(\int_{x_{k-1}}^{x_{k}}\Delta^{-\frac{q}{r}}w\right)^{\frac{1}{q}}\varphi(x_{k-1})^{-\frac{1}{p}}
		\Delta(x_{k-1})^{\frac{1}{r}} \\
		& \qquad + \sup_{\substack{k\in\mathcal{Z}_2 \\ N+2\le k\le M}} \esup_{t\in(x_{k-1},x_{k})} \left(\int_{t}^{x_{k}}\Delta^{-\frac{q}{r}}w\right)^{\frac{1}{q}}\left(\int_{x_{k-1}}^{t}\Delta^{\frac{r}{1-r}}\delta \varphi^{-\frac{r}{p(1-r)}}\right)^{\frac{1-r}{r}} \\
		& \leq \esup_{t\in(0,x_{N+1})}\left(\int_{t}^{x_{N+1}}\Delta^{-\frac{q}{r}}w\right)^{\frac{1}{q}}
		\left(\int_{0}^{t}\Delta^{\frac{r}{1-r}}\delta \varphi^{-\frac{r}{p(1-r)}}\right)^{\frac{1-r}{r}} \\
		& \qquad +\sup_{N+2\le k\le M} \esup_{t\in(x_{k-1},x_{k})}\left(\int_{t}^{x_{k}}\Delta^{-\frac{q}{r}}w\right)^{\frac{1}{q}}
		\varphi(t)^{-\frac{1}{p}}\Delta(t)^{\frac{1}{r}}
		\\
		& \qquad + \sup_{ N+2\le k\le M} \esup_{t\in(x_{k-1},x_{k})} \left(\int_{t}^{x_{k}}\Delta^{-\frac{q}{r}}w\right)^{\frac{1}{q}}\left(\int_{x_{k-1}}^{t}\Delta^{\frac{r}{1-r}}\delta \varphi^{-\frac{r}{p(1-r)}}\right)^{\frac{1-r}{r}}\\
		& \leq \esup_{t\in(0,x_{N+1})}\left(\int_{t}^{x_{N+1}}\Delta^{-\frac{q}{r}}w\right)^{\frac{1}{q}}
		\left(\int_{0}^{t}\Delta^{\frac{r}{1-r}}\delta \varphi^{-\frac{r}{p(1-r)}}\right)^{\frac{1-r}{r}} \\
		& \qquad + \sup_{N+2\le k\le M} \esup_{t\in(x_{k-1},x_{k})}\left(\int_{t}^{x_{k}}\Delta^{-\frac{q}{r}}w\right)^{\frac{1}{q}} \sup_{\tau\in (x_{k-1}, t)}
		\varphi(t)^{-\frac{1}{p}}\Delta(t)^{\frac{1}{r}}
		\\
		& \qquad + \sup_{ N+2\le k\le M} \esup_{t\in(x_{k-1},x_{k})} \left(\int_{t}^{x_{k}}\Delta^{-\frac{q}{r}}w\right)^{\frac{1}{q}}\left(\int_{x_{k-1}}^{t}\Delta^{\frac{r}{1-r}}\delta \varphi^{-\frac{r}{p(1-r)}}\right)^{\frac{1-r}{r}}.
	\end{align*}
	Consequently, using \eqref{4} and \eqref{III-new}, combined with \eqref{EQ:C11lessC13} and \eqref{C33<C34}, we get
	\begin{align*}
		\III \lesssim
		C_{1,3}+ C_{1,1}+  C_{3,3} + C_{3,4} \ls C_{1,3} + C_{3,3} + C_{3,4}\ls C_{1,3} + C_{3,4}.
	\end{align*}
	We finally arrived at
	\begin{equation}\label{EQ:B5lessC13+C34}
		B_5 \ls \I + \II + \III \ls C_{3,4} + C_{1,3}.
	\end{equation}
	
	Putting all things together, we have
	\begin{equation*}
		C_{1,2} + C_{1,3} + C_{3,4} + C_{4,1}
		\ls B_1 + B_2 + B_3 + B_5
		\ls C_{1,2} + C_{1,3} + C_{3,4} + C_{4,1}.
	\end{equation*}
	
	\rm{(v)} By Theorem~\ref{thm:main_discretization}, we have
	\begin{equation*}
		C\approx C_{1,4}+C_{1,5}+C_{3,1}+C_{4,1}.
	\end{equation*}
	We start by establishing the desired upper estimate. In view of \eqref{EQ:C31lesB2} and \eqref{EQ:C41lesB1}, it is sufficient to prove suitable upper bounds on $C_{1,4}$ and $C_{1,5}$.
	
	For every $t\in(x_{k-1}, x_{k}]$, we have by \eqref{phi-Z1}, if $k\in\mathcal{Z}_1$,
	\begin{align*}
		\esup_{s\in(x_{k-1},t)}\left(\int_{x_{k-1}}^s \delta\right)^{\frac{q}{r(1-q)}}\varphi(s)^{-\frac{q}{p(1-q)}}
		&\approx
		\varphi(x_k)^{-\frac{q}{p(1-q)}}\left(\int_{x_{k-1}}^t \delta\right)^{\frac{q}{r(1-q)}} \\
		&\leq
		\varphi(x_k)^{-\frac{q}{p(1-q)}}\Delta(t)^{\frac{q}{r(1-q)}},
	\end{align*}
	and, by  \eqref{Uphi-Z2}, if $k\in\mathcal{Z}_2$,
	\begin{align*}
		&\esup_{s\in(x_{k-1},t)}\left(\int_{x_{k-1}}^s \delta\right)^{\frac{q}{r(1-q)}}\varphi(s)^{-\frac{q}{p(1-q)}} \\
		& \qquad \approx
		\left(\frac{U(x_{k-1})}{\varphi(x_{k-1})^{\frac{1}{p}}}\right)^{\frac{q}{1-q}}\esup_{s\in(x_{k-1},t)}\left(\int_{x_{k-1}}^s \delta\right)^{\frac{q}{r(1-q)}}U(s)^{-\frac{q}{1-q}}.
	\end{align*}
	Then
	\begin{align*}
		C_{1,4}
		&  \lesssim
		\sup_{\substack{k\in\mathcal{Z}_1\\N+1\leq k \leq M}} \varphi(x_k)^{-\frac1p}\left( \int_{x_{k-1}}^{x_k} \left( \int_t^{x_k} \Delta^{-\frac{q}{r}} w \right)^\frac{q}{1-q}\Delta(t)^{-\frac{q}{r}} w(t) \Delta(t)^{\frac{q}{r(1-q)}}\,dt\right)^{\frac{1-q}{q}} \\
		&\quad  +
		\sup_{\substack{k\in\mathcal{Z}_2\\N+1\leq k \leq M}} \frac{U(x_{k-1})}{\varphi(x_{k-1})^{\frac1p}}\left( \int_{x_{k-1}}^{x_k} \left( \int_t^{x_k} \Delta^{-\frac{q}{r}} w \right)^\frac{q}{1-q}\Delta(t)^{-\frac{q}{r}} w(t)\right. \\
		&\qquad \left.  \times
		\esup_{s\in(x_{k-1},t)} \bigg(\int_{x_{k-1}}^s \delta\bigg)^{\frac{q}{r(1-q)}}U(s)^{-\frac{q}{1-q}}\,dt\right)^{\frac{1-q}{q}}\\
		&\lesssim
		\sup_{\substack{k\in\mathcal{Z}_1\\N+1\leq k \leq M}} \varphi(x_k)^{-\frac1p}\left( \int_{x_{k-1}}^{x_k} \left( \int_t^{x_k} \Delta^{-\frac{q}{r}} w \right)^\frac{q}{1-q}\Delta(t)^{-\frac{q}{r}} w(t) \Delta(t)^{\frac{q}{r(1-q)}}\,dt\right)^{\frac{1-q}{q}}\\
		&\quad+
		\sup_{\substack{k\in\mathcal{Z}_2\\N+1\leq k \leq M}} \frac{U(x_{k-1})}{\varphi(x_{k-1})^{\frac1p}}\left( \int_{x_{k-1}}^{x_k} \left( \int_t^{x_k} \Delta^{-\frac{q}{r}} w \right)^\frac{q}{1-q}\Delta(t)^{-\frac{q}{r}} w(t)\right. \\
		&\qquad\left.\times
		\esup_{s\in(x_{k-1},t)}\Delta(s)^{\frac{q}{r(1-q)}}U(s)^{-\frac{q}{1-q}}\,dt\right)^{\frac{1-q}{q}}.
	\end{align*}
	Monotonicity of $\Delta$ yields
	\begin{align*}
		C_{1,4}&\lesssim
		\sup_{N+1\leq k \leq M} \varphi(x_k)^{-\frac1p}\left( \int_{x_{k-1}}^{x_k} d\left[-\left(\int_t^{x_k}w\right)^{\frac1{1-q}}\right]\right)^{\frac{1-q}{q}} \\
		&\qquad +
		\sup_{N+1\leq k \leq M}\esup_{t\in(x_{k-1},x_k)}\frac{U(t)}{\varphi(t)^{\frac1p}}\\
		&\qquad \quad \times \left( \int_{t}^{L} \left( \int_s^{L} \Delta^{-\frac{q}{r}} w \right)^\frac{q}{1-q}\Delta(s)^{-\frac{q}{r}} w(s)
		\esup_{\tau\in(t,s)}\Delta(\tau)^{\frac{q}{r(1-q)}}U(\tau)^{-\frac{q}{1-q}}\,ds\right)^{\frac{1-q}{q}}\\
		&\leq
		B_1+B_6.
	\end{align*}
	On the other hand, in view of $\int_{x_{k-1}}^t\delta \leq \Delta(t)$,
	\begin{align*}
		C_{1,5}
		&\leq
		\sup_{N+1\leq k \leq M} \left( \int_{x_{k-1}}^{x_k} \left( \int_{x_{k-1}}^t w\right)^{\frac{q}{1-q}}w(t)\varphi(t)^{-\frac{q}{p(1-q)}}\,dt\right)^{\frac{1-q}{q}}\\
		& \leq \sup_{N+1\leq k \leq M} \left( \int_{x_{k-1}}^{x_k} W(t)^{\frac{q}{1-q}}w(t)\varphi(t)^{-\frac{q}{p(1-q)}}\,dt\right)^{\frac{1-q}{q}}.
	\end{align*}
	
	Furthermore, applying \eqref{KMT-Lemma3.5} with $p\mapsto \frac{1-q}{q}$, ${\widetilde{\varphi}}\mapsto U\varphi^{-\frac{1}{p}}$, $\varrho\mapsto U$ and $g\mapsto W^{\frac{q}{1-q}}w$, we have
	\begin{align}\label{5-new}
		&\sup_{N+1\leq k \leq M} \left( \int_{x_{k-1}}^{x_k} W(t)^{\frac{q}{1-q}}w(t)\varphi(t)^{-\frac{q}{p(1-q)}}\,dt\right)^{\frac{1-q}{q}}\nonumber\\
		&\quad \approx
		\esup_{t\in(0,L)}\varphi(t)^{-\frac{1}{p}}U(t)\left(\int_0^L\frac{W(s)^{\frac{q}{1-q}}w(s)}{U(t)^{\frac{q}{1-q}}+U(s)^{\frac{q}{1-q}}}\,ds\right)^{\frac{1-q}{q}}.
	\end{align}
	Thus using \eqref{5-new} and \eqref{min-equiv}
	\begin{align}
		C_{1,5}&\ls 
		\esup_{t\in(0,L)}\varphi(t)^{-\frac{1}{p}}U(t)\left(\int_0^L\frac{W(s)^{\frac{q}{1-q}}w(s)}{U(t)^{\frac{q}{1-q}}+U(s)^{\frac{q}{1-q}}}\,ds\right)^{\frac{1-q}{q}}\nonumber\\
		& \approx
		\esup_{t\in(0,L)}\varphi(t)^{-\frac{1}{p}} \bigg(\int_0^t W^{\frac{q}{1-q}} w\bigg)^{\frac{1-q}{q}}\nonumber\\
		&  \qquad +
		\esup_{t\in(0,L)}\varphi(t)^{-\frac{1}{p}}U(t)\left(\int_t^L W^{\frac{q}{1-q}}wU^{-\frac{q}{1-q}}\right)^{\frac{1-q}{q}}\nonumber\\
		&\approx
		\esup_{t\in(0,L)}\varphi(t)^{-\frac{1}{p}}W(t)^{\frac1q}
		+
		\esup_{t\in(0,L)}\varphi(t)^{-\frac{1}{p}}U(t)\left(\int_t^L W^{\frac{q}{1-q}}wU^{-\frac{q}{1-q}}\right)^{\frac{1-q}{q}}\nonumber\\
		&=
		B_1+B_7.\label{EQ:antid_v_*}
	\end{align}
	Putting all these upper estimates together, we have
	\begin{align*}
		C
		\approx
		C_{1,4}+C_{1,5}+C_{3,1}+C_{4,1}
		\lesssim
		B_1+B_2+B_6+B_7.
	\end{align*}
	
	For future reference note that \eqref{5-new} together with \eqref{EQ:antid_v_*} implies  
	\begin{align}\label{Es-B7-new}
		B_7 \lesssim \sup_{N+1\leq k \leq M} \left( \int_{x_{k-1}}^{x_k} W(t)^{\frac{q}{1-q}}w(t)\varphi(t)^{-\frac{q}{p(1-q)}}\,dt\right)^{\frac{1-q}{q}}.
	\end{align}
	
	Conversely, we already obtained in \eqref{EQ:antid_i_B1} that
	\begin{equation*}
		B_1
		\lesssim
		C_{4,1}+C_{1,2}+C_{3,1}.
	\end{equation*}
	Also, using the monotonicity of $\varphi$, we have 
	\begin{align*}
		C_{1,2}  & \approx
		\sup_{N+1\leq k \leq M} \esup_{t\in(x_{k-1}, x_k)} 
		\varphi(t)^{-\frac1p} \\
		& \hspace{2cm} \times \left(\int_{x_{k-1}}^t d\bigg[\bigg(\int_{x_{k-1}}^{\tau}\Delta(s)^{-\frac{q}{r}}w(s) \left(\int_{x_{k-1}}^s \delta\right)^{\frac{q}{r}}\, ds\bigg)^{\frac{1}{1-q}}\bigg]\right)^{\frac{1-q}{q}} \\
		&\approx
		\sup_{N+1\leq k \leq M} \esup_{t\in(x_{k-1}, x_k)} \left(\int_{x_{k-1}}^t\left(\int_{x_{k-1}}^{\tau}\Delta(s)^{-\frac{q}{r}}w(s) \left(\int_{x_{k-1}}^s \delta\right)^{\frac{q}{r}}\, ds\right)^{\frac{q}{1-q}}\right.\\
		&\qquad\times
		\left.\Delta(\tau)^{-\frac{q}{r}}w(\tau) \left(\int_{x_{k-1}}^{\tau} \delta\right)^{\frac{q}{r}}\, d\tau\right)^{\frac{1-q}{q}}\varphi(t)^{-\frac1p} \\
		&\leq C_{1,5}.
	\end{align*}
	Thus, we arrive at
	\begin{equation}
		B_1
		\lesssim
		C_{4,1}+C_{1,5}+C_{3,1}.\label{EQ:B1UBCase5}
	\end{equation}
	Next, by \eqref{EQ:antid_i_B2}, we know that
	\begin{align*}
		B_2 \lesssim C_{1,1}+C_{3,1}.
	\end{align*}
	Moreover, it is easy to see that
	\begin{align*}
		C_{1,1}
		&\approx
		\sup_{N+1\leq k \leq M} \esup_{t\in(x_{k-1}, x_k)} \left( \int_t^{x_k} \left(\int_s^{x_k} \Delta^{-\frac{q}{r}} w \right)^\frac{q}{1-q}\Delta(s)^{-\frac{q}{r}} w(s)\,ds \right)^\frac{1-q}{q} \\
		&\qquad\times
		\esup_{s\in(x_{k-1}, t)} \left( \int_{x_{k-1}}^s \delta \right)^\frac1{r} \varphi(s)^{-\frac1{p}}\\
		&\leq C_{1,4}.
	\end{align*}
	Combining these estimates gives
	\begin{align}\label{EQ:antid_v_B2}
		B_2 \lesssim
		C_{1,4}+C_{3,1}.
	\end{align}
	Next, similar to the proofs of the previous cases, first decomposing the integral $\int_t^L \dots$ and then the supremum $\sup_{\tau \in (t,s) }\dots$, we obtain
	\begin{align*}
		B_6 & \approx  \sup_{N+1 \leq k \leq M} \esup_{t\in(x_{k-1},x_k)}     U(t)\varphi(t)^{-\frac1p} \\ 
		& \qquad \times \Bigg(\int_t^L \esup_{\tau\in(t,s)}\Delta(\tau)^{\frac{q}{r(1-q)}}U(\tau)^{-\frac{q}{1-q}} d\Bigg[-\bigg(\int_s^L\Delta^{-\frac{q}{r}}w\bigg)^{\frac{1}{1-q}}\Bigg]\Bigg)^{\frac{1-q}{q}}\\
		& \approx  \sup_{N+1 \leq k \leq M} \esup_{t\in(x_{k-1},x_k)}     U(t)\varphi(t)^{-\frac1p} \\ 
		& \qquad \times \Bigg(\int_t^{x_k} \esup_{\tau\in(t,s)}\Delta(\tau)^{\frac{q}{r(1-q)}}U(\tau)^{-\frac{q}{1-q}} d\Bigg[-\bigg(\int_s^L\Delta^{-\frac{q}{r}}w\bigg)^{\frac{1}{1-q}}\Bigg]\Bigg)^{\frac{1-q}{q}}\\
		& +  \sup_{N+1 \leq k \leq M-1} \esup_{t\in(x_{k-1},x_k)}     U(t)\varphi(t)^{-\frac1p} \\ 
		&  \qquad \times \Bigg(\int_{x_k}^L \esup_{\tau\in(t,s)}\Delta(\tau)^{\frac{q}{r(1-q)}}U(\tau)^{-\frac{q}{1-q}} d\Bigg[-\bigg(\int_s^L\Delta^{-\frac{q}{r}}w\bigg)^{\frac{1}{1-q}}\Bigg]\Bigg)^{\frac{1-q}{q}}\\
		& \approx  \sup_{N+1 \leq k \leq M} \esup_{t\in(x_{k-1},x_k)}  U(t)\varphi(t)^{-\frac1p} \\ 
		&  \qquad \times \Bigg(\int_t^{x_k} \esup_{\tau\in(t,s)}\Delta(\tau)^{\frac{q}{r(1-q)}}U(\tau)^{-\frac{q}{1-q}} d\Bigg[-\bigg(\int_s^L\Delta^{-\frac{q}{r}}w\bigg)^{\frac{1}{1-q}}\Bigg]\Bigg)^{\frac{1-q}{q}}\\
		& +  \sup_{N+1 \leq k \leq M-1} \esup_{t\in(x_{k-1},x_k)}     U(t)\varphi(t)^{-\frac1p} \\ 
		&  \qquad \times \Bigg(\int_{x_k}^L \esup_{\tau\in(x_k,s)}\Delta(\tau)^{\frac{q}{r(1-q)}}U(\tau)^{-\frac{q}{1-q}} d\Bigg[-\bigg(\int_s^L\Delta^{-\frac{q}{r}}w\bigg)^{\frac{1}{1-q}}\Bigg]\Bigg)^{\frac{1-q}{q}}\\
		& +  \sup_{N+1 \leq k \leq M-1} \Bigg(\int_{x_k}^L  \Delta^{-\frac{q}{r}}w\bigg)^{\frac{1}{q}} \esup_{t\in(x_{k-1},x_k)}     U(t)\varphi(t)^{-\frac1p} \bigg(\esup_{\tau\in(t,x_k)} \Delta(\tau)^{\frac{1}{r}} U(\tau)^{-1}\bigg)\\
		&\approx
		\sup_{N+1\leq k \leq M} \esup_{t\in(x_{k-1}, x_k)}U(t)\varphi(t)^{-\frac1p} \\
		&\qquad \times \left(\int_t^{x_k}\esup_{\tau\in(t,s)}\Delta(\tau)^{\frac{q}{r(1-q)}}U(\tau)^{-\frac{q}{1-q}}d\left[-\left(\int_s^L\Delta^{-\frac{q}{r}}w\right)^{\frac{1}{1-q}}\right]\right)^{\frac{1-q}{q}}\\
		&\quad+
		\sup_{N+1\leq k \leq M-1} U(x_k)\varphi(x_k)^{-\frac1p}\\
		&\qquad\times
		\left(\int_{x_k}^L\left(\int_s^L\Delta^{-\frac{q}{r}}w\right)^{\frac{q}{1-q}}\Delta(s)^{-\frac{q}{r}}w(s)\esup_{\tau\in(x_k,s)}\Delta(\tau)^{\frac{q}{r(1-q)}}U(\tau)^{-\frac{q}{1-q}}\,ds\right)^{\frac{1-q}{q}}\\
		&\quad+
		\sup_{N+1\leq k \leq M-1} \left(\int_{x_k}^L\Delta^{-\frac{q}{r}}w\right)^{\frac{1}{q}} \esup_{t\in(x_{k-1}, x_k)}  U(t)\varphi(t)^{-\frac1p}\left(\esup_{\tau\in(t,x_k)}\Delta(\tau)^{\frac1r}U(\tau)^{-1}\right)\\
		&=:\I + \II + \III.
	\end{align*}
	Integrating by parts gives
	\begin{align*}
		\I & \lesssim 
		\sup_{N+1\leq k \leq M} \esup_{t\in(x_{k-1}, x_k)}\varphi(t)^{-\frac1p}\Delta(t)^{\frac1r}\left(\int_{t}^L\Delta^{-\frac{q}{r}}w\right)^{\frac{1}{q}}\\
		&\quad+ \sup_{N+1\leq k \leq M} \esup_{t\in(x_{k-1}, x_k)}U(t)\varphi(t)^{-\frac1p} \\
		& \qquad \quad  \times \left(\int_t^{x_k}\left(\int_s^L\Delta^{-\frac{q}{r}}w\right)^{\frac{1}{1-q}}d\left[\esup_{\tau\in(t,s)}\Delta(\tau)^{\frac{q}{r(1-q)}}U(\tau)^{-\frac{q}{1-q}}\right]\right)^{\frac{1-q}{q}}\\
		&=
		B_2 + \sup_{N+1\leq k \leq M} \esup_{t\in(x_{k-1}, x_k)}U(t)\varphi(t)^{-\frac1p} \\
		& \qquad \qquad  \times \left(\int_t^{x_k}\left(\int_s^L\Delta^{-\frac{q}{r}}w\right)^{\frac{1}{1-q}}d\left[\esup_{\tau\in(t,s)}\Delta(\tau)^{\frac{q}{r(1-q)}}U(\tau)^{-\frac{q}{1-q}}\right]\right)^{\frac{1-q}{q}}.
	\end{align*}
	Next, \eqref{int-cut-Dw} together with \eqref{EQ:antid_v_B2} yields,
	\begin{align*}
		\I  &\lesssim 
		C_{1,4} + C_{3,1} + \sup_{N+1\leq k \leq M} \esup_{t\in(x_{k-1}, x_k)}U(t)\varphi(t)^{-\frac1p} \\
		& \qquad \qquad  \times \left(\int_t^{x_k}\left(\int_s^{x_k}\Delta^{-\frac{q}{r}}w\right)^{\frac{1}{1-q}}d\left[\esup_{\tau\in(t,s)}\Delta(\tau)^{\frac{q}{r(1-q)}}U(\tau)^{-\frac{q}{1-q}}\right]\right)^{\frac{1-q}{q}}\\
		&+ \sup_{N+1\leq k \leq M} \left(\int_{x_k}^L\Delta^{-\frac{q}{r}}w\right)^{\frac{1}{q}} \esup_{t\in(x_{k-1}, x_k)} U(t)\varphi(t)^{-\frac1p} \\
		& \qquad \qquad  \times \left(\int_t^{x_k}d\left[\esup_{\tau\in(t,s)}\Delta(\tau)^{\frac{q}{r(1-q)}}U(\tau)^{-\frac{q}{1-q}}\right]\right)^{\frac{1-q}{q}}.
	\end{align*}
	Integrating by parts once again, we obtain
	\begin{align*}
		\I & \lesssim C_{1,4}+C_{3,1}
		+
		\sup_{N+1\leq k \leq M} \esup_{t\in(x_{k-1}, x_k)}U(t)\varphi(t)^{-\frac1p}\\
		&\qquad \quad \times
		\left(\int_t^{x_k}\left(\int_s^{x_k}\Delta^{-\frac{q}{r}}w\right)^{\frac{q}{1-q}}\Delta(s)^{-\frac{q}{r}}w(s)\esup_{\tau\in(t,s)}\Delta(\tau)^{\frac{q}{r(1-q)}}U(\tau)^{-\frac{q}{1-q}}\,ds\right)^{\frac{1-q}{q}}\\
		&\quad+
		\sup_{N+1\leq k \leq M-1} \left(\int_{x_k}^L\Delta^{-\frac{q}{r}}w\right)^{\frac{1}{q}}\esup_{t\in(x_{k-1},x_k)}U(t)\varphi(t)^{-\frac1p}\esup_{\tau\in(t,x_k)}\Delta(\tau)^{\frac{1}{r}}U(\tau)^{-1}.    
	\end{align*}
	Next, monotonicity of $U\varphi^{-\frac{1}{p}}$ gives
	\begin{align}\label{I-v-new}
		\I & \lesssim C_{1,4}+C_{3,1} \nonumber\\
		&\quad + \sup_{N+1\leq k \leq M}
		\left(\int_{x_{k-1}}^{x_k}\left(\int_s^{x_k}\Delta^{-\frac{q}{r}}w\right)^{\frac{q}{1-q}}\Delta(s)^{-\frac{q}{r}}w(s) \right. \nonumber\\
		& \qquad \quad \left. \times \esup_{\tau\in(x_{k-1},s)}\Delta(\tau)^{\frac{q}{r(1-q)}}\varphi(\tau)^{-\frac{q}{p(1-q)}}\,ds\right)^{\frac{1-q}{q}}\nonumber\\
		&\quad+
		\sup_{N+1\leq k \leq M-1} \left(\int_{x_k}^L\Delta^{-\frac{q}{r}}w\right)^{\frac{1}{q}}\esup_{\tau\in(x_{k-1},x_k)}\Delta(\tau)^{\frac{1}{r}}\varphi(\tau)^{-\frac{1}{p}}\nonumber\\
		& \lesssim C_{1,4}+C_{3,1} \nonumber\\
		&\quad + \sup_{N+1\leq k \leq M}
		\left(\int_{x_{k-1}}^{x_k}\left(\int_s^{x_k}\Delta^{-\frac{q}{r}}w\right)^{\frac{q}{1-q}}\Delta(s)^{-\frac{q}{r}}w(s) \right.\nonumber \\
		& \qquad \quad \left. \times \esup_{\tau\in(x_{k-1},s)}\Delta(\tau)^{\frac{q}{r(1-q)}}\varphi(\tau)^{-\frac{q}{p(1-q)}}\,ds\right)^{\frac{1-q}{q}}\nonumber\\
		&\quad+
		\sup_{N+1\leq k \leq M-1} \left(\int_{x_k}^L\Delta^{-\frac{q}{r}}w\right)^{\frac{1}{q}}\esup_{\tau\in(x_{k-1},x_k)}\Delta(\tau)^{\frac{1}{r}}\varphi(\tau)^{-\frac{1}{p}}.
	\end{align}
	Note that in view of \eqref{EQ:antid_v_B2}, we have
	\begin{equation}\label{I-v-new1}
		\sup_{N+1\leq k \leq M-1} \left(\int_{x_k}^L\Delta^{-\frac{q}{r}}w\right)^{\frac{1}{q}}\esup_{\tau\in(x_{k-1},x_k)}\Delta(\tau)^{\frac{1}{r}}\varphi(\tau)^{-\frac{1}{p}} \leq B_2 \lesssim C_{1,4} + C_{3,1},
	\end{equation}
	moreover, applying \eqref{int-cut-D} we have
	\begin{align}\label{5}
		&\sup_{N+1\leq k \leq M}\left(\int_{x_{k-1}}^{x_k}\left(\int_s^{x_k}\Delta^{-\frac{q}{r}}w\right)^{\frac{q}{1-q}}\Delta(s)^{-\frac{q}{r}}w(s) \right. \nonumber \\
		& \hspace{3cm} \left. \times \esup_{\tau\in(x_{k-1},s)}\Delta(\tau)^{\frac{q}{r(1-q)}}\varphi(\tau)^{-\frac{q}{p(1-q)}}\,ds\right)^{\frac{1-q}{q}} \nonumber\\
		& \approx 
		\left(\int_{x_{N}}^{x_{N+1}}\left(\int_s^{x_{N+1}}\Delta^{-\frac{q}{r}}w\right)^{\frac{q}{1-q}}\Delta(s)^{-\frac{q}{r}}w(s) \right. \nonumber \\
		& \hspace{3cm} \left. \times \esup_{\tau\in(x_{N},s)}\Delta(\tau)^{\frac{q}{r(1-q)}}\varphi(\tau)^{-\frac{q}{p(1-q)}}\,ds\right)^{\frac{1-q}{q}} \nonumber\\
		& \quad + \sup_{N+2\leq k \leq M}\left(\int_{x_{k-1}}^{x_k}\left(\int_s^{x_k}\Delta^{-\frac{q}{r}}w\right)^{\frac{q}{1-q}}\Delta(s)^{-\frac{q}{r}}w(s) \right. \nonumber \\
		& \hspace{3cm}\left. \times \esup_{\tau\in(x_{k-1},s)}\left(\int_{x_{k-1}}^{\tau} \delta\right)^{\frac{q}{r(1-q)}}\varphi(\tau)^{-\frac{q}{p(1-q)}}\,ds\right)^{\frac{1-q}{q}}\nonumber\\
		& \quad \quad + \sup_{N+2\leq k \leq M} \Delta(x_{k-1})^{\frac{1}{r}}\varphi(x_{k-1})^{-\frac{1}{p}} \nonumber \\
		&\hspace{3cm} \times\left(\int_{x_{k-1}}^{x_k}\left(\int_s^{x_k}\Delta^{-\frac{q}{r}}w\right)^{\frac{q}{1-q}}\Delta(s)^{-\frac{q}{r}}w(s)\,ds\right)^{\frac{1-q}{q}}\nonumber\\
		& \approx \sup_{N+1\leq k \leq M}\left(\int_{x_{k-1}}^{x_k}\left(\int_s^{x_k}\Delta^{-\frac{q}{r}}w\right)^{\frac{q}{1-q}}\Delta(s)^{-\frac{q}{r}}w(s) \right. \nonumber \\
		& \hspace{3cm}\left. \times \esup_{\tau\in(x_{k-1},s)}\left(\int_{x_{k-1}}^{\tau} \delta\right)^{\frac{q}{r(1-q)}}\varphi(\tau)^{-\frac{q}{p(1-q)}}\,ds\right)^{\frac{1-q}{q}}\nonumber\\
		& \quad + \sup_{N+2\leq k \leq M} \Delta(x_{k-1})^{\frac{1}{r}}\varphi(x_{k-1})^{-\frac{1}{p}} \nonumber \\
		&\hspace{3cm} \times\left(\int_{x_{k-1}}^{x_k}\left(\int_s^{x_k}\Delta^{-\frac{q}{r}}w\right)^{\frac{q}{1-q}}\Delta(s)^{-\frac{q}{r}}w(s)\,ds\right)^{\frac{1-q}{q}}\nonumber\\
		&\approx C_{1,4} + \sup_{N+2\leq k \leq M} \Delta(x_{k-1})^{\frac{1}{r}}\varphi(x_{k-1})^{-\frac{1}{p}}\left(\int_{x_{k-1}}^{x_k}\Delta^{-\frac{q}{r}}w\right)^{\frac{1}{q}}\nonumber\\
		&\lesssim C_{1,4} + C_{3,1},
	\end{align}
	where we used \eqref{C-31-estimate-new} in the last estimate. 
	Thus, plugging \eqref{I-v-new1} and \eqref{5} into \eqref{I-v-new}, we arrive at 
	\begin{equation}
		I \lesssim C_{1,4} + C_{3,1}.
	\end{equation}
	
	We shall now deal with $\II$. Note that
	\begin{align*}
		\II  & \approx   \sup_{N+1\leq k \leq M-1} U(x_k)\varphi(x_k)^{-\frac1p} \\
		& \hspace{1cm} \times
		\left(\sum_{i=k}^{M-1}\int_{x_i}^{x_{i+1}}\esup_{\tau\in(x_k,s)}\Delta(\tau)^{\frac{q}{r(1-q)}}U(\tau)^{-\frac{q}{1-q}} d\left[-\left(\int_s^L\Delta^{-\frac{q}{r}}w\right)^{\frac{1}{1-q}}\right]\right)^{\frac{1-q}{q}}\\
		& \approx  \sup_{N+1\leq k \leq M-1} U(x_k)\varphi(x_k)^{-\frac1p} \\
		& \hspace{1cm} \times
		\left(\int_{x_k}^{x_{k+1}} \esup_{\tau\in(x_k,s)}\Delta(\tau)^{\frac{q}{r(1-q)}}U(\tau)^{-\frac{q}{1-q}} d\left[-\left(\int_s^L\Delta^{-\frac{q}{r}}w\right)^{\frac{1}{1-q}}\right]\right)^{\frac{1-q}{q}}\\
		&  + \sup_{N+1\leq k \leq M-2} U(x_k)\varphi(x_k)^{-\frac1p} \\
		& \hspace{1cm}  \times
		\left(\sum_{i=k+1}^{M-1} \int_{x_i}^{x_{i+1}} \esup_{\tau\in(x_k,s)}\Delta(\tau)^{\frac{q}{r(1-q)}}U(\tau)^{-\frac{q}{1-q}}d\left[-\left(\int_s^L\Delta^{-\frac{q}{r}}w\right)^{\frac{1}{1-q}}\right]\right)^{\frac{1-q}{q}}.
	\end{align*}
	Since
	\begin{align*}
		\esup_{\tau\in(x_k,s)}\Delta(\tau)^{\frac{q}{r(1-q)}}U(\tau)^{-\frac{q}{1-q}}& \approx \esup_{\tau\in(x_k,x_i)}\Delta(\tau)^{\frac{q}{r(1-q)}}U(\tau)^{-\frac{q}{1-q}} \\
		& \hspace{2cm} + \esup_{\tau\in(x_i,s)}\Delta(\tau)^{\frac{q}{r(1-q)}}U(\tau)^{-\frac{q}{1-q}},
	\end{align*}
	when $i>k$ and $s\in (x_i, x_{i+1})$, we have
	\begin{align*}
		\II & \approx \sup_{N+1\leq k \leq M-1} U(x_k)\varphi(x_k)^{-\frac1p} \\
		& \hspace{1cm} \times
		\left(\int_{x_k}^{x_{k+1}} \esup_{\tau\in(x_k,s)}\Delta(\tau)^{\frac{q}{r(1-q)}}U(\tau)^{-\frac{q}{1-q}} d\left[-\left(\int_s^L\Delta^{-\frac{q}{r}}w\right)^{\frac{1}{1-q}}\right]\right)^{\frac{1-q}{q}}\\
		& + \sup_{N+1\leq k \leq M-2} U(x_k)\varphi(x_k)^{-\frac1p} \\
		& \hspace{1cm} \times
		\left(\sum_{i=k+1}^{M-1}\esup_{\tau\in(x_k,x_i)}\Delta(\tau)^{\frac{q}{r(1-q)}}U(\tau)^{-\frac{q}{1-q}}\int_{x_i}^{x_{i+1}}d\left[-\left(\int_s^L\Delta^{-\frac{q}{r}}w\right)^{\frac{1}{1-q}}\right]\right)^{\frac{1-q}{q}}\\
		&\quad+
		\sup_{N+1\leq k \leq M-2} U(x_k)\varphi(x_k)^{-\frac1p} \\
		& \hspace{1cm} \times
		\left(\sum_{i=k+1}^{M-1}\int_{x_i}^{x_{i+1}}\esup_{\tau\in(x_i,s)}\Delta(\tau)^{\frac{q}{r(1-q)}}U(\tau)^{-\frac{q}{1-q}} d\left[-\left(\int_s^L\Delta^{-\frac{q}{r}}w\right)^{\frac{1}{1-q}}\right]\right)^{\frac{1-q}{q}}\\
		& \approx\sup_{N+1\leq k \leq M-2} U(x_k)\varphi(x_k)^{-\frac1p} \\
		& \hspace{1cm} \times
		\left(\sum_{i=k+1}^{M-1}\esup_{\tau\in(x_k,x_i)}\Delta(\tau)^{\frac{q}{r(1-q)}}U(\tau)^{-\frac{q}{1-q}}\int_{x_i}^{x_{i+1}}d\left[-\left(\int_s^L\Delta^{-\frac{q}{r}}w\right)^{\frac{1}{1-q}}\right]\right)^{\frac{1-q}{q}}\\
		&\quad+
		\sup_{N+1\leq k \leq M-1} U(x_k)\varphi(x_k)^{-\frac1p} \\
		& \hspace{1cm} \times
		\left(\sum_{i=k}^{M-1}\int_{x_i}^{x_{i+1}}\esup_{\tau\in(x_i,s)}\Delta(\tau)^{\frac{q}{r(1-q)}}U(\tau)^{-\frac{q}{1-q}} d\left[-\left(\int_s^L\Delta^{-\frac{q}{r}}w\right)^{\frac{1}{1-q}}\right]\right)^{\frac{1-q}{q}}\\
		& =: \II_1 + \II_2.
	\end{align*}
	Observe that
	\begin{align*}
		\II_1
		&\lesssim
		\sup_{N+1\leq k \leq M-2} U(x_k)\varphi(x_k)^{-\frac1p} \left(\sum_{i=k+1}^{M-1}\left(\sum_{m=k+1}^i\esup_{\tau\in(x_{m-1},x_{m})}\Delta(\tau)^{\frac{q}{r(1-q)}}U(\tau)^{-\frac{q}{1-q}}\right) \right.\\
		&\hspace{2cm} \times \left.\int_{x_i}^{x_{i+1}}d\left[-\left(\int_s^L\Delta^{-\frac{q}{r}}w\right)^{\frac{1}{1-q}}\right]\right)^{\frac{1-q}{q}}.
	\end{align*}
	Changing the order of summation, then since $\{U^p(x_i)/\varphi(x_i)\}_{i=N+1}^{k-1}$ is strongly increasing, using \eqref{EQ:strongly_increasing_sup_sum} 
	\begin{align*}
		\II_1
		&\lesssim \sup_{N+1\leq k \leq M-2} U(x_k)\varphi(x_k)^{-\frac1p}
		\left(\sum_{m=k+1}^{M-1}\esup_{\tau\in(x_{m-1},x_{m})}\Delta(\tau)^{\frac{q}{r(1-q)}}U(\tau)^{-\frac{q}{1-q}} \right. \\
		& \hspace{2cm}\left. \times \int_{x_m}^{L}d\left[-\left(\int_s^L\Delta^{-\frac{q}{r}}w\right)^{\frac{1}{1-q}}\right]\right)^{\frac{1-q}{q}}\\
		&\approx
		\sup_{N+1\leq k \leq M-2} \left(\int_{x_{k+1}}^L\Delta^{-\frac{q}{r}}w\right)^{\frac{1}{q}}  U(x_k)\varphi(x_k)^{-\frac1p}\esup_{\tau\in(x_k,x_{k+1})}\Delta(\tau)^{\frac{1}{r}}U(\tau)^{-1}.
	\end{align*}
	Finally monotonicity of $U\varphi^{-\frac{1}{p}}$ and \eqref{EQ:antid_v_B2} yields, 
	\begin{align*}
		\II_1
		&\lesssim \sup_{N+1\leq k \leq M-2} \left(\int_{x_{k+1}}^L\Delta^{-\frac{q}{r}}w\right)^{\frac{1}{q}}  \esup_{\tau\in(x_k,x_{k+1})} \Delta(\tau)^{\frac{1}{r}}\varphi(\tau)^{-\frac1p}\\
		&\leq B_2\\
		& \lesssim C_{1,4}+C_{3,1}.
	\end{align*}
	On the other hand, applying  \eqref{EQ:strongly_increasing_sup_sum} and integrating by parts, we obtain
	\begin{align*}
		\II_2
		&\approx   \sup_{N+1\leq k \leq M-1} U(x_k)\varphi(x_k)^{-\frac1p} \\
		& \quad \qquad \times
		\left(\int_{x_k}^{x_{k+1}}\esup_{\tau\in(x_k,s)}\Delta(\tau)^{\frac{q}{r(1-q)}}U(\tau)^{-\frac{q}{1-q}}d\left[-\left(\int_s^L\Delta^{-\frac{q}{r}}w\right)^{\frac{1}{1-q}}\right]\right)^{\frac{1-q}{q}}\\
		& \lesssim \sup_{N+1\leq k \leq M-1} \varphi(x_k)^{-\frac1p}
		\Delta(x_k)^{\frac{1}{r}}\left(\int_{x_k}^L\Delta^{-\frac{q}{r}}w\right)^{\frac{1}{q}}\\
		&\quad + \sup_{N+1\leq k \leq M-1} U(x_k)\varphi(x_k)^{-\frac1p} \\
		& \quad \qquad \times
		\left(\int_{x_k}^{x_{k+1}} \left(\int_s^L\Delta^{-\frac{q}{r}}w\right)^{\frac{1}{1-q}}d\left[- \esup_{\tau\in(x_k,s)}\Delta(\tau)^{\frac{q}{r(1-q)}}U(\tau)^{-\frac{q}{1-q}}\right]\right)^{\frac{1-q}{q}}.
	\end{align*}
	Decomposing the integral $\int_s^L$ into sums $\int_s^{x_{k+1}} + \int_{x_{k+1}}^L$, then integrating by parts again, we get
	\begin{align*}
		\II_2& \lesssim \sup_{N+1\leq k \leq M-1} \varphi(x_k)^{-\frac1p}
		\Delta(x_k)^{\frac{1}{r}}\left(\int_{x_k}^L\Delta^{-\frac{q}{r}}w\right)^{\frac{1}{q}}\\
		&\quad + \sup_{N+1\leq k \leq M-2} U(x_k)\varphi(x_k)^{-\frac1p}
		\left(\int_{x_{k+1}}^L\Delta^{-\frac{q}{r}}w\right)^{\frac{1}{q}}\esup_{\tau\in(x_k,x_{k+1})}\Delta(\tau)^{\frac{1}{r}}U(\tau)^{-1}\\
		&\quad + \sup_{N+1\leq k \leq M-1} U(x_k)\varphi(x_k)^{-\frac1p} \\
		& \quad \qquad \times
		\left(\int_{x_k}^{x_{k+1}} \left(\int_s^{x_{k+1}}\Delta^{-\frac{q}{r}}w\right)^{\frac{1}{1-q}}d\left[- \esup_{\tau\in(x_k,s)}\Delta(\tau)^{\frac{q}{r(1-q)}}U(\tau)^{-\frac{q}{1-q}}\right]\right)^{\frac{1-q}{q}}\\
		&\lesssim
		\sup_{N+1\leq k \leq M-1} \varphi(x_k)^{-\frac1p}
		\Delta(x_k)^{\frac{1}{r}}\left(\int_{x_k}^L\Delta^{-\frac{q}{r}}w\right)^{\frac{1}{q}}\\
		&\quad+
		\sup_{N+1\leq k \leq M-2} U(x_k)\varphi(x_k)^{-\frac1p}
		\left(\int_{x_{k+1}}^L\Delta^{-\frac{q}{r}}w\right)^{\frac{1}{q}}\esup_{\tau\in(x_k,x_{k+1})}\Delta(\tau)^{\frac{1}{r}}U(\tau)^{-1}\\
		&\quad+
		\sup_{N+1\leq k \leq M-1} U(x_k)\varphi(x_k)^{-\frac1p}
		\left(\int_{x_k}^{x_{k+1}}\left(\int_s^{x_{k+1}}\Delta^{-\frac{q}{r}}w\right)^{\frac{q}{1-q}}\Delta(s)^{-\frac{q}{r}}w(s) \right. \\
		& \quad \qquad \times \left. \esup_{\tau\in(x_k,s)}\Delta(\tau)^{\frac{q}{r(1-q)}}U(\tau)^{-\frac {q}{1-q}}\,ds\right)^{\frac{1-q}{q}}.
	\end{align*}
	Using the monotonicity of $U\varphi^{-\frac{1}{p}}$ in the second and third terms, we obtain
	\begin{align*}
		\II_2 & \lesssim B_2 + \sup_{N+1\leq k \leq M-1} 
		\left(\int_{x_k}^{x_{k+1}}\left(\int_s^{x_{k+1}}\Delta^{-\frac{q}{r}}w\right)^{\frac{q}{1-q}}\Delta(s)^{-\frac{q}{r}}w(s) \right. \\
		& \quad \qquad \times \left. \esup_{\tau\in(x_k,s)}\Delta(\tau)^{\frac{q}{r(1-q)}}\varphi(\tau)^{-\frac {q}{p(1-q)}}\,ds\right)^{\frac{1-q}{q}}.
	\end{align*}
	Thus,  \eqref{EQ:antid_v_B2} and \eqref{5} give $ \II_2 \lesssim C_{1,4}+C_{3,1}$. Consequently, we have $\II \lesssim C_{1,4}+C_{3,1} $. 
	
	Furthermore, using the monotonicity of $U\varphi^{-\frac{1}{p}}$ and applying  \eqref{EQ:antid_v_B2} once again, we obtain
	\begin{align*}
		\III\leq\sup_{N+1\leq k\leq M-1}\left(\int_{x_k}^L\Delta^{-\frac{q}{r}}w\right)^{\frac{1}{q}} \esup_{\tau\in(x_{k-1},x_k)}\Delta(\tau)^{\frac1r}\varphi(\tau)^{-\frac1p}\leq C_{1,4}+C_{3,1}.
	\end{align*}
	Altogether, we have
	\begin{align*}
		B_6
		\approx
		\I + \II + \III
		\lesssim
		C_{1,4}+C_{3,1}.
	\end{align*}
	
	Lastly, note that for $N+2\leq k\leq M$ and $y\in(x_{k-1}, L)$, 
	\begin{equation}\label{W-cut}
		W(t)=W(x_{k-1})+\int_{x_{k-1}}^t w, \,\, t\in (x_{k-1},y).
	\end{equation}
	Then for $N+2\leq k\leq M$ and $y\in(x_{k-1}, L)$ by integrating by parts, and \eqref{W-cut} we have
	\begin{align}\label{EQ:antid_v_**}
		&\left(\int_{x_{k-1}}^y W^{\frac{q}{1-q}}w\varphi^{-\frac{q}{p(1-q)}}\right)^{\frac{1-q}{q}} \nonumber\\
		& \qquad \lesssim
		W(y)^{\frac1q}\varphi(y)^{-\frac1p}+\left(\int_{x_{k-1}}^y W(t)^{\frac{1}{1-q}}d\left[-\varphi(t)^{-\frac{q}{p(1-q)}}\right]\right)^{\frac{1-q}{q}} \nonumber\\
		&\qquad  \lesssim
		W(y)^{\frac1q}\varphi(y)^{-\frac1p}+W(x_{k-1})^{\frac1q}\varphi(x_{k-1})^{-\frac1p} \nonumber\\
		& \hspace{2cm}  +
		\left(\int_{x_{k-1}}^y\left(\int_{x_{k-1}}^tw\right)^{\frac{1}{1-q}} d\left[-\varphi(t)^{-\frac{q}{p(1-q)}}\right]\right)^{\frac{1-q}{q}}\nonumber\\
		&\qquad\lesssim
		W(y)^{\frac1q}\varphi(y)^{-\frac1p}+W(x_{k-1})^{\frac1q}\varphi(x_{k-1})^{-\frac1p}\nonumber\\
		& \hspace{2cm} +
		\left(\int_{x_{k-1}}^y\left(\int_{x_{k-1}}^tw\right)^{\frac{q}{1-q}}w(t)\varphi(t)^{-\frac{q}{p(1-q)}}\,dt\right)^{\frac{1-q}{q}}.
	\end{align}
	Finally, first using \eqref{Es-B7-new}, then applying \eqref{EQ:antid_v_**}  with $y=x_k$,  we have
	\begin{align*}
		B_7
		&\lesssim
		\sup_{N+1\leq k\leq M}\left(\int_{x_{k-1}}^{x_k}W^{\frac{q}{1-q}}w\varphi^{-\frac{q}{p(1-q)}}\right)^{\frac{1-q}q} \\
		&\approx
		\left(\int_{x_N}^{x_{N+1}}W^{\frac{q}{1-q}}w\varphi^{-\frac{q}{p(1-q)}}\right)^{\frac{1-q}q}
		+
		\sup_{N+2\leq k\leq M}\left(\int_{x_{k-1}}^{x_k}W^{\frac{q}{1-q}}w\varphi^{-\frac{q}{p(1-q)}}\right)^{\frac{1-q}q} \\
		&\lesssim
		\sup_{N+2\leq k\leq M}W(x_{k})^{\frac1q}\varphi(x_{k})^{-\frac1p}
		+
		\sup_{N+2\leq k\leq M}W(x_{k-1})^{\frac1q}\varphi(x_{k-1})^{-\frac1p} \\
		&\quad+
		\sup_{N+1\leq k\leq M}\left(\int_{x_{k-1}}^{x_k}\left(\int_{x_{k-1}}^tw\right)^{\frac{q}{1-q}}w(t)\varphi(t)^{-\frac{q}{p(1-q)}}\,dt\right)^{\frac{1-q}{q}}.
	\end{align*}
	Since $\{\varphi(x_k)^{-\frac{1}{p}}\}_{k=N+1}^{M-1}$ is strongly decreasing, it is clear that,  using \eqref{EQ:strongly_decreasing_sup_sum}, we have
	\begin{equation}\label{6}
		\sup_{N+2\leq k\leq M}W(x_{k-1})^{\frac1q}\varphi(x_{k-1})^{-\frac1p} \approx \sup_{N+2\leq k\leq M}\left(\int_{x_{k-2}}^{x_{k-1}}w\right)^{\frac1q}\varphi(x_{k-1})^{-\frac1p} = C_{4,1},
	\end{equation}
	and applying \eqref{EQ:strongly_decreasing_sup_sum}, together with the monotonicity of $\varphi$, 
	\begin{align}\label{7}
		\sup_{N+2\leq k\leq M}W(x_k)^{\frac1q}\varphi(x_k)^{-\frac1p} & \approx \sup_{N+2\leq k\leq M}\bigg(
		\int_{x_{k-1}}^{x_k} w \bigg)^{\frac1q}\varphi(x_k)^{-\frac1p} \nonumber\\
		&\lesssim \sup_{N+2\leq k\leq M}\left(\int_{x_{k-1}}^{x_k} \left(\int_{x_{k-1}}^t w\right)^{\frac{q}{1-q}} w(t) \varphi(t)^{-\frac{q}{p(1-q)}}\, dt\right)^{\frac{1-q}{q}}.
	\end{align}
	Then, in view of \eqref{6}, \eqref{7} and integrating by parts, we get
	\begin{align*}
		B_7  & \lesssim C_{4,1} + \sup_{N+1\leq k\leq M}\left(\int_{x_{k-1}}^{x_k}\left(\int_{x_{k-1}}^tw\right)^{\frac{q}{1-q}}w(t)\varphi(t)^{-\frac{q}{p(1-q)}}\,dt\right)^{\frac{1-q}{q}}\\
		&\lesssim
		C_{4,1}+ \sup_{N+1\leq k\leq M}\varphi(x_k)^{-\frac{1}{p}} \left(\int_{x_{k-1}}^{x_k}w\right)^{\frac{1}{q}} \\
		&\quad +\sup_{N+1\leq k\leq M}\left(\int_{x_{k-1}}^{x_k}\varphi(t)^{-\frac{q}{p(1-q)}}d\left[\left(\int_{x_{k-1}}^tw\right)^{\frac{1}{1-q}}\right]\right)^{\frac{1-q}{q}}\\
		&  \leq
		C_{4,1}+ \sup_{N+1\leq k\leq M} \sup_{t\in (x_{k-1}, x_k)} \varphi(t)^{-\frac{1}{p}} \left(\int_{x_{k-1}}^{t} w \right)^{\frac{1}{q}} \\
		&\quad  +\sup_{N+1\leq k\leq M}\left(\int_{x_{k-1}}^{x_k}\varphi(t)^{-\frac{q}{p(1-q)}}d\left[\left(\int_{x_{k-1}}^tw\right)^{\frac{1}{1-q}}\right]\right)^{\frac{1-q}{q}}.
	\end{align*}
	Using \eqref{EQ:C41+sthlesB1} and \eqref{EQ:B1UBCase5} we obtain
	\begin{align}
		B_7  &\lesssim
		B_1 +\sup_{N+1\leq k\leq M}\left(\int_{x_{k-1}}^{x_k}\varphi(t)^{-\frac{q}{p(1-q)}}d\left[\left(\int_{x_{k-1}}^tw\right)^{\frac{1}{1-q}}\right]\right)^{\frac{1-q}{q}} \nonumber \\
		& \lesssim C_{4,1} + C_{1,5} + C_{3,1}+\sup_{N+1\leq k\leq M}\left(\int_{x_{k-1}}^{x_k}\varphi(t)^{-\frac{q}{p(1-q)}}d\left[\left(\int_{x_{k-1}}^tw\right)^{\frac{1}{1-q}}\right]\right)^{\frac{1-q}{q}}.\label{EQ:B7half}
	\end{align}
	Observe that, integrating by parts yields, 
	\begin{align*}
		&\sup_{N+1\leq k\leq M}\left(\int_{x_{k-1}}^{x_k}\varphi(t)^{-\frac{q}{p(1-q)}}d\left[\left(\int_{x_{k-1}}^tw\right)^{\frac{1}{1-q}}\right]\right)^{\frac{1-q}{q}} \\
		&\approx \sup_{N+1\leq k\leq M}\varphi(x_k)^{-\frac{1}{p}} \left(\int_{x_{k-1}}^{x_k}w\right)^{\frac{1}{q}} \nonumber\\
		&\quad+ \sup_{N+1\leq k\leq M} \left(\int_{x_{k-1}}^{x_k}\left(\int_{x_{k-1}}^t w\right)^{\frac{1}{1-q}}d\left[-\varphi(t)^{-\frac{q}{p(1-q)}}\right]\right)^{\frac{1-q}{q}}\nonumber\\
		& \approx \sup_{N+1\leq k\leq M}\varphi(x_k)^{-\frac{1}{p}} \left(\int_{x_{k-1}}^{x_k} w \right)^{\frac{1}{q}}  \nonumber\\
		&\quad  + \left(\int_{x_{N}}^{x_{N+1}} \left(\int_{x_{N}}^t \bigg(\int_{x_{N}}^s \delta\bigg)^{\frac{q}{r}} \bigg(\int_{x_{N}}^s \delta\bigg)^{-\frac{q}{r}}w(s)ds\right)^{\frac{1}{1-q}}d\left[-\varphi(t)^{-\frac{q}{p(1-q)}}\right]\right)^{\frac{1-q}{q}}\nonumber \\
		&\quad  + \sup_{N+2\leq k\leq M} \left(\int_{x_{k-1}}^{x_k}\left(\int_{x_{k-1}}^t \Delta^{\frac{q}{r}} \Delta^{-\frac{q}{r}}w\right)^{\frac{1}{1-q}}d\left[-\varphi(t)^{-\frac{q}{p(1-q)}}\right]\right)^{\frac{1-q}{q}}.\nonumber 
	\end{align*}
	Next, applying \eqref{int-cut-D}, we have
	\begin{align*}
		&\sup_{N+1\leq k\leq M}\left(\int_{x_{k-1}}^{x_k}\varphi(t)^{-\frac{q}{p(1-q)}}d\left[\left(\int_{x_{k-1}}^tw\right)^{\frac{1}{1-q}}\right]\right)^{\frac{1-q}{q}} \\
		&\approx \sup_{N+1\leq k\leq M}\varphi(x_k)^{-\frac{1}{p}} \left(\int_{x_{k-1}}^{x_k}w\right)^{\frac{1}{q}} \nonumber\\
		&\quad + \left(\int_{x_{N}}^{x_{N+1}} \left(\int_{x_{N}}^t \bigg(\int_{x_{N}}^s \delta\bigg)^{\frac{q}{r}} \bigg(\int_{x_{N}}^s \delta\bigg)^{-\frac{q}{r}}w(s)ds\right)^{\frac{1}{1-q}}d\left[-\varphi(t)^{-\frac{q}{p(1-q)}}\right]\right)^{\frac{1-q}{q}}\nonumber \\
		&\quad+ \sup_{N+2\leq k\leq M}\Delta(x_{k-1})^{\frac1r}\left(\int_{x_{k-1}}^{x_k}\left(\int_{x_{k-1}}^t\Delta^{-\frac{q}{r}}w\right)^{\frac{1}{1-q}}d\left[-\varphi(t)^{-\frac{q}{p(1-q)}}\right]\right)^{\frac{1-q}{q}}\nonumber\\
		&\quad+
		\sup_{N+2\leq k\leq M}\left(\int_{x_{k-1}}^{x_k}\left(\int_{x_{k-1}}^t\Delta(s)^{-\frac{q}{r}}w(s)\left(\int_{x_{k-1}}^s\delta\right)^{\frac{q}{r}}\,ds\right)^{\frac{1}{1-q}}d\left[-\varphi(t)^{-\frac{q}{p(1-q)}}\right]\right)^{\frac{1-q}{q}} \nonumber\\
		&\approx \sup_{N+1\leq k\leq M}\varphi(x_k)^{-\frac{1}{p}} \left(\int_{x_{k-1}}^{x_k}w\right)^{\frac{1}{q}} \nonumber\\
		&\quad+ \sup_{N+2\leq k\leq M}\Delta(x_{k-1})^{\frac1r}\left(\int_{x_{k-1}}^{x_k}\left(\int_{x_{k-1}}^t\Delta^{-\frac{q}{r}}w\right)^{\frac{1}{1-q}}d\left[-\varphi(t)^{-\frac{q}{p(1-q)}}\right]\right)^{\frac{1-q}{q}}\nonumber\\
		&\quad+
		\sup_{N+1\leq k\leq M}\left(\int_{x_{k-1}}^{x_k}\left(\int_{x_{k-1}}^t\Delta(s)^{-\frac{q}{r}}w(s)\left(\int_{x_{k-1}}^s\delta\right)^{\frac{q}{r}}\,ds\right)^{\frac{1}{1-q}}d\left[-\varphi(t)^{-\frac{q}{p(1-q)}}\right]\right)^{\frac{1-q}{q}} \nonumber\\
		&\lesssim B_1 +
		\sup_{N+2\leq k\leq M}\Delta(x_{k-1})^{\frac1r}\left(\int_{x_{k-1}}^{x_k}\Delta^{-\frac{q}{r}}w\right)^{\frac{1}{q}}\varphi(x_{k-1})^{-\frac{1}{p}} \nonumber\\
		&\quad+
		\sup_{N+1\leq k\leq M}\left(\int_{x_{k-1}}^{x_k}\left(\int_{x_{k-1}}^t\Delta(s)^{-\frac{q}{r}}w(s)\left(\int_{x_{k-1}}^s\delta\right)^{\frac{q}{r}}\,ds\right)^{\frac{1}{1-q}}d\left[-\varphi(t)^{-\frac{q}{p(1-q)}}\right]\right)^{\frac{1-q}{q}}.
	\end{align*}
	Integrating by parts once more, we get
	\begin{align*}
		&\sup_{N+1\leq k\leq M}\left(\int_{x_{k-1}}^{x_k}\varphi(t)^{-\frac{q}{p(1-q)}}d\left[\left(\int_{x_{k-1}}^tw\right)^{\frac{1}{1-q}}\right]\right)^{\frac{1-q}{q}} \\
		&\lesssim B_1 +
		\sup_{N+2\leq k\leq M}\Delta(x_{k-1})^{\frac1r}\left(\int_{x_{k-1}}^{x_k}\Delta^{-\frac{q}{r}}w\right)^{\frac{1}{q}}\varphi(x_{k-1})^{-\frac{1}{p}} \nonumber\\
		&\quad+
		\sup_{N+1\leq k\leq M}\left(\int_{x_{k-1}}^{x_k}\left(\int_{x_{k-1}}^t\Delta(s)^{-\frac{q}{r}}w(s)\left(\int_{x_{k-1}}^s\delta\right)^{\frac{q}{r}}\,ds\right)^{\frac{q}{1-q}}\right. \nonumber\\
		&\qquad\times
		\left.\Delta(t)^{-\frac{q}{r}}w(t)\left(\int_{x_{k-1}}^t\delta\right)^{\frac{q}{r}}\varphi(t)^{-\frac{q}{p(1-q)}}\,dt\right)^{\frac{1-q}{q}} \nonumber\\
		&\leq  B_1 +  B_2 + C_{1,5} \nonumber \\& \lesssim C_{4,1} +  C_{1,4}+C_{3,1}+C_{1,5}.
	\end{align*}
	Note that we used \eqref{EQ:B1UBCase5} and \eqref{EQ:antid_v_B2} in the last inequality. Plugging the last estimate in \eqref{EQ:B7half}, we arrive at
	\begin{equation}\label{EQ:B8lessC41+C15+C31}
		B_7 \lesssim C_{1,4} + C_{3,1} + C_{1,5} + C_{4,1}.
	\end{equation}
	
	Putting all things together, we have
	\begin{align*}
		B_1+B_2+B_6+B_7
		\lesssim
		C_{4,1}+C_{1,5}+C_{3,1}+C_{1,4}
		\lesssim
		B_1+B_2+B_6+B_7;
	\end{align*}
	consequently
	\begin{align*}
		C
		\approx
		B_1+B_2+B_6+B_7.
	\end{align*}
	
	\rm{(vi)}
	By Theorem~\ref{thm:main_discretization}, we have
	\begin{equation*}
		C\approx C_{1,5} + C_{1,6} + C_{3,2} + C_{4,1}.
	\end{equation*}
	Using \eqref{EQ:antid_ii_1}, \eqref{nofor6} and the monotonicity of  $\Delta$ we have
	\begin{align}
		C_{1,6} 	&  \lesssim
		\sup_{N+1\leq k \leq M} \varphi(x_k)^{-\frac1p}\left( \int_{x_{k-1}}^{x_k} \left( \int_t^{x_{k}} \Delta^{-\frac{q}{r}} w \right)^\frac{q}{1-q}  \Delta(t)^{\frac{q^2}{r(1-q)}} w(t) dt \right)^{\frac{1-q}{q}}  \nonumber\\
		& \quad +
		\sup_{N+1\leq k \leq M} U(x_{k-1}) \varphi(x_{k-1})^{-\frac1p} \bigg( \int_{x_{k-1}}^{x_k} \bigg( \int_t^{x_{k}} \Delta^{-\frac{q}{r}} w \bigg)^\frac{q}{1-q} 
		\nonumber\\
		&\qquad  \times \, 
		\Delta(t)^{-\frac{q}{r}} w(t) \bigg( \int_{x_{k-1}}^{t} \bigg( \int_{x_{k-1}}^{s}  \delta(s) U(s)^{-\frac{r}{1-r}}\, ds \bigg)^{\frac{q(1-r)}{r(1-q)}} dt\bigg)^{\frac{1-q}{q}}
		\nonumber\\
		&\lesssim
		\sup_{N+1\leq k \leq M} \varphi(x_k)^{-\frac1p}\left( \int_{x_{k-1}}^{x_k} \left( \int_t^{x_{k}}  w \right)^\frac{q}{1-q}  w(t) dt \right)^{\frac{1-q}{q}}
		\nonumber\\
		&\quad+
		\sup_{N+1\leq k \leq M} \esup_{\tau\in(x_{k-1},x_k)}U(\tau)\varphi(\tau)^{-\frac1p}\left( \int_{\tau}^{x_k} \left( \int_t^{x_{k}} \Delta^{-\frac{q}{r}} w \right)^\frac{q}{1-q} \right.
		\nonumber\\
		&\qquad \times\left.
		\Delta(t)^{-\frac{q}{r}} w(t) \left(\int_{\tau}^t \Delta^\frac{r}{1-r} \delta U^{-\frac{r}{1-r}} \right)^{\frac{q(1-r)}{r(1-q)}} 	dt \right)^{\frac{1-q}{q}}
		\nonumber\\
		&\lesssim
		\sup_{N+1\leq k \leq M} \varphi(x_k)^{-\frac1p}W(x_k)^{\frac1q} \nonumber \\
		& \quad
		+
		\sup_{N+1\leq k \leq M} \esup_{\tau\in(x_{k-1},x_k)}U(\tau)\varphi(\tau)^{-\frac1p}\left( \int_{\tau}^{L} \left( \int_t^{L} \Delta^{-\frac{q}{r}} w \right)^\frac{q}{1-q} \right.
		\nonumber\\
		&\qquad \times\left.
		\Delta(t)^{-\frac{q}{r}} w(t) \left(\int_{\tau}^t \Delta^\frac{r}{1-r} \delta U^{-\frac{r}{1-r}} \right)^{\frac{q(1-r)}{r(1-q)}} 	dt \right)^{\frac{1-q}{q}} \nonumber\\
		&= B_1 + B_8. \label{EQ:C16lessB1+B9}
	\end{align}
	
	Thus,  \eqref{EQ:C41lesB1}, \eqref{EQ:C32lessB2+B3}, \eqref{EQ:antid_v_*} and \eqref{EQ:C16lessB1+B9} altogether gives
	\begin{align*}
		C_{1,5}+C_{1,6}+C_{3,2}+C_{4,1}\lesssim B_1+B_2+B_3+B_7+B_8.
	\end{align*}
	
	As for the opposite inequality, thanks to \eqref{EQ:B1UBCase5} and \eqref{EQ:antid_v_B2} combined with \eqref{EQ:C31lessC32} note that
	\begin{equation*}
		B_1 + B_2 \lesssim C_{1,4} + C_{1,5} + C_{3,1} + C_{4,1} \lesssim C_{1,4} + C_{1,5} + C_{3,2} + C_{4,1}.
	\end{equation*}
	On the other hand, applying \eqref{sup-int}, it is easy to see that 
	\begin{equation}\label{C14<C16}
		C_{1,4} \lesssim C_{1,6}.
	\end{equation} Then,
	\begin{equation}\label{EQ:B1+B2lessC15+C16+C32+C41}
		B_1 + B_2 \lesssim C_{1,6} + C_{1,5} + C_{3,2} + C_{4,1}.
	\end{equation}
	Using the same argument as in \eqref{sup-int} it can be easily shown that 
	\begin{equation}\label{C13<C16}
		C_{1,3} \lesssim C_{1,6}.    
	\end{equation}
	Consequently, owing to \eqref{EQ:antid_ii_B3}, observe that 
	\begin{equation}\label{EQ:B3lessC16+C32}
		B_3 \lesssim C_{1,3} + C_{3,2} \lesssim C_{1,6} + C_{3,2},
	\end{equation}
	and thanks to \eqref{EQ:B8lessC41+C15+C31} combined with \eqref{C14<C16} and \eqref{EQ:C31lessC32},
	\begin{equation}\label{EQ:B8lessC15+C32+C41}
		B_7 \lesssim C_{1,4} + C_{3,1} + C_{1,5} + C_{4,1} \lesssim C_{1,6} + C_{3,2} + C_{1,5} + C_{4,1}.
	\end{equation}

	On the other hand, decomposing the integral $\int_{t}^{L}$ into the sum $\int_{t}^{x_{k}}+\int_{x_{k}}^{L}$ yields
	\begin{align*}	
		B_{8}
		&\approx
		\sup_{N+1\leq k \leq M}\esup_{t\in(x_{k-1},x_{k})} U(t)\varphi(t)^{-\frac1p} \\
		& \qquad \times \left(\int_t^{x_k}\left(\int_t^s\Delta^{\frac{r}{1-r}}\delta U^{-\frac{r}{1-r}}\right)^{\frac{q(1-r)}{r(1-q)}}d\left[-\left(\int_s^L\Delta^{-\frac{q}{r}}w\right)^{\frac{1}{1-q}}\right]\right)^{\frac{1-q}{q}}
		\\
		&\quad+
		\sup_{N+1\leq k \leq M-1}\esup_{t\in(x_{k-1},x_{k})} U(t)\varphi(t)^{-\frac1p}
		\\
		&\qquad\times
		\left(\int_{x_k}^L\left(\int_s^L\Delta^{-\frac{q}{r}}w\right)^{\frac{q}{1-q}}\Delta(s)^{-\frac{q}{r}}w(s)\left(\int_t^s\Delta^{\frac{r}{1-r}}\delta U^{-\frac{r}{1-r}}\right)^{\frac{q(1-r)}{r(1-q)}}\,ds\right)^{\frac{1-q}{q}}.
	\end{align*}
	Now, integration by parts in the first term and decomposing the integral $\int_{t}^{s}$ into the sum $\int_{t}^{x_{k}}+\int_{x_{k}}^{s}$ in the second term, gives
	\begin{align}	
		B_{8}
		&\lesssim	
		\sup_{N+1\leq k \leq M}\esup_{t\in(x_{k-1},x_{k})} U(t)\varphi(t)^{-\frac1p} \notag\\
		& \qquad \times\left(\int_t^{x_k}	\left(\int_s^L\Delta^{-\frac{q}{r}}w\right)^{\frac{1}{1-q}}d\left[\left(\int_t^s\Delta^{\frac{r}{1-r}}\delta U^{-\frac{r}{1-r}}\right)^{\frac{q(1-r)}{r(1-q)}}\right]\right)^{\frac{1-q}{q}}
		\notag\\
		&\quad+
		\sup_{N+1\leq k \leq M-1} \left(\int_{x_k}^L \left(\int_s^L \Delta^{-\frac{q}{r}}w \right)^{\frac{q}{1-q}}  \Delta(s)^{-\frac{q}{r}} w(s)\, ds\right)^{\frac{1-q}{q}} \notag\\
		& \quad \quad \times \esup_{t\in(x_{k-1},x_{k})} U(t)\varphi(t)^{-\frac1p} \left(\int_t^{x_k}\Delta^{\frac{r}{1-r}}\delta U^{-\frac{r}{1-r}}\right)^{\frac{1-r}{r}}
		\notag\\
		&\quad+
		\sup_{N+1\leq k \leq M-1} U(x_k)\varphi(x_k)^{-\frac1p} \notag\\
		& \qquad \times
		\left(\int_{x_k}^L\left(\int_s^L\Delta^{-\frac{q}{r}}w\right)^{\frac{q}{1-q}}\Delta(s)^{-\frac{q}{r}}w(s)\left(\int_{x_k}^s\Delta^{\frac{r}{1-r}}\delta U^{-\frac{r}{1-r}}\right)^{\frac{q(1-r)}{r(1-q)}}\,ds\right)^{\frac{1-q}{q}}.\label{B8-new}
	\end{align}
	Observe that, decomposing the integral $\int_s^L$ into the sum $\int_s^{x_k} + \int_{x_k}^L$, yields
	\begin{align*}
		&\sup_{N+1\leq k \leq M}\esup_{t\in(x_{k-1},x_{k})} U(t)\varphi(t)^{-\frac1p} \\
		& \hspace{2cm} \times\left(\int_t^{x_k}	\left(\int_s^L\Delta^{-\frac{q}{r}}w\right)^{\frac{1}{1-q}}d\left[\left(\int_t^s\Delta^{\frac{r}{1-r}}\delta U^{-\frac{r}{1-r}}\right)^{\frac{q(1-r)}{r(1-q)}}\right]\right)^{\frac{1-q}{q}} \\
		& \approx \sup_{N+1\leq k \leq M} \esup_{t\in(x_{k-1},x_{k})} U(t)\varphi(t)^{-\frac1p} \\
		& \hspace{2cm} \times\left(\int_t^{x_k}	\left(\int_s^{x_k}\Delta^{-\frac{q}{r}}w\right)^{\frac{1}{1-q}}d\left[\left(\int_t^s\Delta^{\frac{r}{1-r}}\delta U^{-\frac{r}{1-r}}\right)^{\frac{q(1-r)}{r(1-q)}}\right]\right)^{\frac{1-q}{q}} \\
		& + \sup_{N+1\leq k \leq M} \left(\int_{x_k}^L\Delta^{-\frac{q}{r}}w\right)^{\frac{1}{1-q}} \esup_{t\in(x_{k-1},x_{k})} U(t)\varphi(t)^{-\frac1p}  \left(\int_t^{x_k}\Delta^{\frac{r}{1-r}}\delta U^{-\frac{r}{1-r}}\right)^{\frac{1-r}{r}}.
	\end{align*}
	On the other hand,
	\begin{align*}
		& \sup_{N+1\leq k \leq M-1} \left(\int_{x_k}^L \left(\int_s^L \Delta^{-\frac{q}{r}}w \right)^{\frac{q}{1-q}}  \Delta(s)^{-\frac{q}{r}} w(s)\, ds\right)^{\frac{1-q}{q}} \\
		& \hspace{3cm} \times \esup_{t\in(x_{k-1},x_{k})} U(t)\varphi(t)^{-\frac1p} \left(\int_t^{x_k}\Delta^{\frac{r}{1-r}}\delta U^{-\frac{r}{1-r}}\right)^{\frac{1-r}{r}}\\ 
		& \approx \sup_{N+1\leq k \leq M-1} \left(\int_{x_k}^L \Delta^{-\frac{q}{r}}w \right)^{\frac{1}{q}} \esup_{t\in(x_{k-1},x_{k})} U(t)\varphi(t)^{-\frac1p} \left(\int_t^{x_k}\Delta^{\frac{r}{1-r}}\delta U^{-\frac{r}{1-r}}\right)^{\frac{1-r}{r}}.
	\end{align*}
	Plugging the last two estimates in \eqref{B8-new}, we obtain
	\begin{align*}
		B_8 &\lesssim \sup_{N+1\leq k \leq M}\esup_{t\in(x_{k-1},x_{k})} U(t)\varphi(t)^{-\frac1p} \\
		& \qquad \times \left(\int_t^{x_k}	\left(\int_s^{x_k}\Delta^{-\frac{q}{r}}w\right)^{\frac{1}{1-q}}d\left[\left(\int_t^s\Delta^{\frac{r}{1-r}}\delta U^{-\frac{r}{1-r}}\right)^{\frac{q(1-r)}{r(1-q)}}\right]\right)^{\frac{1-q}{q}}
		\\
		&\quad+
		\sup_{N+1\leq k \leq M-1} \left(\int_{x_k}^L \Delta^{-\frac{q}{r}}w \right)^{\frac{1}{q}}\esup_{t\in(x_{k-1},x_{k})} U(t)\varphi(t)^{-\frac1p} \left(\int_t^{x_k}\Delta^{\frac{r}{1-r}}\delta U^{-\frac{r}{1-r}}\right)^{\frac{1-r}{r}}
		\\
		&\quad+
		\sup_{N+1\leq k \leq M-1} U(x_k)\varphi(x_k)^{-\frac1p} \\
		& \qquad \times
		\left(\int_{x_k}^L\left(\int_s^L\Delta^{-\frac{q}{r}}w\right)^{\frac{q}{1-q}}\Delta(s)^{-\frac{q}{r}}w(s)\left(\int_{x_k}^s\Delta^{\frac{r}{1-r}}\delta U^{-\frac{r}{1-r}}\right)^{\frac{q(1-r)}{r(1-q)}}\,ds\right)^{\frac{1-q}{q}} \\
		& =: \I + \II + \III.
	\end{align*}
	
	As for $\II$, using \eqref{8}, we obtain
	\begin{align*}
		\II
		\lesssim
		\sup_{N+1\leq k \leq M-1}\left(\int_{x_k}^L\Delta^{-\frac{q}{r}}w\right)^{\frac{1}{q}}\esup_{t\in(0,x_{k})} U(t)\varphi(t)^{-\frac1p}\left(\int_t^{x_k}\Delta^{\frac{r}{1-r}}\delta U^{-\frac{r}{1-r}}\right)^{\frac{1-r}{r}}
		\lesssim
		C_{3,2}.
	\end{align*}
	
	We shall now turn our attention to $\I$. Integrating by parts, the monotonicity of $U\varphi^{-\frac{1}{p}}$ and applying \eqref{EQ:antid_ii_5}, we have
	\begin{align*}
		\I
		&\approx
		\sup_{N+1\leq k \leq M}\esup_{t\in(x_{k-1},x_{k})} U(t)\varphi(t)^{-\frac1p}\left(\int_t^{x_k}	\left(\int_s^{x_k}\Delta^{-\frac{q}{r}}w\right)^{\frac{q}{1-q}}\Delta(s)^{-\frac{q}{r}}w(s)\right.
		\\
		&\hspace{3cm} \times\left.
		\left(\int_t^s\Delta^{\frac{r}{1-r}}\delta U^{-\frac{r}{1-r}}\right)^{\frac{q(1-r)}{r(1-q)}}\,ds\right)^{\frac{1-q}{q}}
		\\
		&\leq
		\sup_{N+1\leq k \leq M}\left(\int_{x_{k-1}}^{x_k}	\left(\int_s^{x_k}\Delta^{-\frac{q}{r}}w\right)^{\frac{q}{1-q}}\Delta(s)^{-\frac{q}{r}}w(s) \right. \\
		& \hspace{3cm} \left. \times \left(\int_{x_{k-1}}^s\Delta^{\frac{r}{1-r}}\delta \varphi^{-\frac{r}{p(1-r)}}\right)^{\frac{q(1-r)}{r(1-q)}}\,ds\right)^{\frac{1-q}{q}}
		\\
		&\approx
		\left(\int_{x_N}^{x_{N+1}}	\left(\int_s^{x_{N+1}}\Delta^{-\frac{q}{r}}w\right)^{\frac{q}{1-q}}\Delta(s)^{-\frac{q}{r}}w(s)\left(\int_{x_N}^s\Delta^{\frac{r}{1-r}}\delta \varphi^{-\frac{r}{p(1-r)}}\right)^{\frac{q(1-r)}{r(1-q)}}\,ds\right)^{\frac{1-q}{q}}
		\\
		&\quad+
		\sup_{N+2\leq k \leq M}\left(\int_{x_{k-1}}^{x_k}	\left(\int_s^{x_k}\Delta^{-\frac{q}{r}}w\right)^{\frac{q}{1-q}}\Delta(s)^{-\frac{q}{r}}w(s) \right. \\
		& \hspace{2cm} \left. \times \left(\int_{x_{k-1}}^s\Delta^{\frac{r}{1-r}}\delta \varphi^{-\frac{r}{p(1-r)}}\right)^{\frac{q(1-r)}{r(1-q)}}\,ds\right)^{\frac{1-q}{q}}
		\\
		&\lesssim
		\left(\int_{x_N}^{x_{N+1}}	\left(\int_s^{x_{N+1}}\Delta^{-\frac{q}{r}}w\right)^{\frac{q}{1-q}}\Delta(s)^{-\frac{q}{r}}w(s)\left(\int_{x_N}^s\Delta^{\frac{r}{1-r}}\delta \varphi^{-\frac{r}{p(1-r)}}\right)^{\frac{q(1-r)}{r(1-q)}}\,ds\right)^{\frac{1-q}{q}}
		\\
		& \quad +
		\sup_{N+2\leq k \leq M}\left(\int_{x_{k-1}}^{x_k}	\left(\int_s^{x_k}\Delta^{-\frac{q}{r}}w\right)^{\frac{q}{1-q}}\Delta(s)^{-\frac{q}{r}}w(s) \right. \\
		& \hspace{2cm} \left. \times \Delta(s)^{\frac{q}{r(1-q)}}\varphi(s)^{-\frac{q}{p(1-q)}}\,ds\right)^{\frac{1-q}{q}}
		\\
		&\quad+
		\sup_{N+2\leq k \leq M}\varphi(x_{k-1})^{-\frac1p}\Delta(x_{k-1})^{\frac1r}\left(\int_{x_{k-1}}^{x_k}	\left(\int_s^{x_k}\Delta^{-\frac{q}{r}}w\right)^{\frac{q}{1-q}}\Delta(s)^{-\frac{q}{r}}w(s)\,ds\right)^{\frac{1-q}{q}}
		\\
		&\quad+
		\sup_{N+2\leq k \leq M}\left(\int_{x_{k-1}}^{x_k}	\left(\int_s^{x_k}\Delta^{-\frac{q}{r}}w\right)^{\frac{q}{1-q}}\Delta(s)^{-\frac{q}{r}}w(s)\right.
		\\
		&\hspace{2cm} \times\left.
		\left(\int_{x_{k-1}}^s\left(\int_{x_{k-1}}^{\tau}\delta\right)^{\frac{r}{1-r}}\delta(\tau) \varphi(\tau)^{-\frac{r}{p(1-r)}}\,d\tau\right)^{\frac{q(1-r)}{r(1-q)}}\,ds\right)^{\frac{1-q}{q}}\\
		&\approx
		\sup_{N+2\leq k \leq M}\left(\int_{x_{k-1}}^{x_k}	\Delta(s)^{\frac{q}{r(1-q)}}\varphi(s)^{-\frac{q}{p(1-q)}}d\left[-\left(\int_s^{x_k}\Delta^{-\frac{q}{r}}w\right)^{\frac{1}{1-q}}\right]\right)^{\frac{1-q}{q}}
		\\
		&\quad+
		\sup_{N+2\leq k \leq M}\varphi(x_{k-1})^{-\frac1p}\Delta(x_{k-1})^{\frac1r}\left(\int_{x_{k-1}}^{x_k}	\left(\int_s^{x_k}\Delta^{-\frac{q}{r}}w\right)^{\frac{q}{1-q}}\Delta(s)^{-\frac{q}{r}}w(s)\,ds\right)^{\frac{1-q}{q}}
		\\
		&\quad+
		\sup_{N+1\leq k \leq M}\left(\int_{x_{k-1}}^{x_k}	\left(\int_s^{x_k}\Delta^{-\frac{q}{r}}w\right)^{\frac{q}{1-q}}\Delta(s)^{-\frac{q}{r}}w(s)\right.
		\\
		&\hspace{2cm} \times\left.
		\left(\int_{x_{k-1}}^s\left(\int_{x_{k-1}}^{\tau}\delta\right)^{\frac{r}{1-r}}\delta(\tau) \varphi(\tau)^{-\frac{r}{p(1-r)}}\,d\tau\right)^{\frac{q(1-r)}{r(1-q)}}\,ds\right)^{\frac{1-q}{q}}\\
		&\lesssim
		C_{1,6}+
		\sup_{N+2\leq k \leq M}\left(\int_{x_{k-1}}^{x_k}	\Delta(s)^{\frac{q}{r(1-q)}}\varphi(s)^{-\frac{q}{p(1-q)}}d\left[-\left(\int_s^{x_k}\Delta^{-\frac{q}{r}}w\right)^{\frac{1}{1-q}}\right]\right)^{\frac{1-q}{q}}
		\\
		&\quad+
		\sup_{N+2\leq k \leq M}\varphi(x_{k-1})^{-\frac1p}\Delta(x_{k-1})^{\frac1r}\left(\int_{x_{k-1}}^{x_k}	\Delta^{-\frac{q}{r}}w\right)^{\frac{1}{q}}.
	\end{align*}
	Next, \eqref{int-cut-D} and the monotonicity of $\varphi$ yields
	\begin{align*}
		\I
		&\lesssim C_{1,6}+
		\sup_{N+2\leq k \leq M}\left(\int_{x_{k-1}}^{x_k}	\left(\int_{x_{k-1}}^{s}\delta\right)^{\frac{q}{r(1-q)}}\varphi(s)^{-\frac{q}{p(1-q)}} d\left[-\left(\int_s^{x_k}\Delta^{-\frac{q}{r}}w\right)^{\frac{1}{1-q}}\right]\right)^{\frac{1-q}{q}}
		\\
		&\quad+
		\sup_{N+2\leq k \leq M} \Delta(x_{k-1})^{\frac{1}{r}} \varphi(x_{k-1})^{-\frac{1}{p}} \left(\int_{x_{k-1}}^{x_k}	\Delta^{-\frac{q}{r}}w\right)^{\frac{1}{q}}
		\\
		&\approx
		C_{1,6}+
		\sup_{N+2\leq k \leq M}\left(\int_{x_{k-1}}^{x_k}	\left(\int_{x_{k-1}}^{s}\left(\int_{x_{k-1}}^{\tau}\delta\right)^{\frac{r}{1-r}}\delta(\tau)\,d\tau\right)^{\frac{q(1-r)}{r(1-q)}}
		\right.\\
		&\hspace{2cm}\times\left.\varphi(s)^{-\frac{q}{p(1-q)}}\left(\int_s^{x_k}\Delta^{-\frac{q}{r}}w\right)^{\frac{q}{1-q}}\Delta(s)^{-\frac{q}{r}}w(s)\,ds\right)^{\frac{1-q}{q}}
		\\
		&\quad+
		\sup_{N+2\leq k \leq M}\Delta(x_{k-1})^{\frac{1}{r}}\varphi(x_{k-1})^{-\frac{1}{p}}\left(\int_{x_{k-1}}^{x_k}	\Delta^{-\frac{q}{r}}w\right)^{\frac{1}{q}}
		\\
		&\leq
		2C_{1,6}+
		\sup_{N+2\leq k \leq M}\Delta(x_{k-1})^{\frac{1}{r}}\varphi(x_{k-1})^{-\frac{1}{p}}\left(\int_{x_{k-1}}^{x_k}	\Delta^{-\frac{q}{r}}w\right)^{\frac{1}{q}}\\
		&\leq
		2C_{1,6}+
		B_2\\
		&\lesssim 	C_{1,6}+C_{1,5}+C_{3,2}+C_{4,1},
	\end{align*}
	where we used the fact that $B_2 \lesssim C_{1,6} + C_{1,5} + C_{3,2} + C_{4,1}$ thanks to \eqref{EQ:B1+B2lessC15+C16+C32+C41} in the last inequality.
	
	For future reference, note that we also showed that
	\begin{align}\label{EQ:antid_vi_*}
		&\sup_{N+2\leq k \leq M}\bigg(\int_{x_{k-1}}^{x_k}	\bigg(\int_s^{x_k}\Delta^{-\frac{q}{r}}w\bigg)^{	\frac{q}{1-q}}\Delta(s)^{-\frac{q}{r}}w(s)\bigg(\int_{x_{k-1}}^s\Delta^{\frac{r}{1-r}}\delta \varphi^{-\frac{r}{p(1-r)}}\bigg)^{\frac{q(1-r)}{r(1-q)}}\,ds\bigg)^{\frac{1-q}{q}}
		\nonumber\\
		&\qquad\lesssim C_{1,6}+C_{1,5}+C_{3,2}+C_{4,1}.
	\end{align}
	
	Finally, as for $\III$, we have
	\begin{align*}
		\III
		&\approx
		\sup_{N+1\leq k \leq M-1}U(x_k)\varphi(x_k)^{-\frac1p} \\
		& \qquad \times
		\left(\sum_{i=k}^{M-1}\int_{x_i}^{x_{i+1}}\left(\int_{x_k}^s\Delta^{\frac{r}{1-r}}\delta U^{-\frac{r}{1-r}}\right)^{\frac{q(1-r)}{r(1-q)}}d\left[-\left(\int_s^L\Delta^{-\frac{q}{r}}w\right)^{\frac{1}{1-q}}\right]\right)^{\frac{1-q}{q}}.
	\end{align*}
	Decomposing the integral $\int_{x_k}^s$ into sum $\int_{x_k}^{x_i} + \int_{x_i}^s$, we get
	\begin{align*}
		\III
		&\approx
		\sup_{N+1\leq k \leq M-1}U(x_k)\varphi(x_k)^{-\frac1p}\\
		& \qquad \times
		\left(\sum_{i=k}^{M-1}\int_{x_i}^{x_{i+1}}\left(\int_{x_i}^s\Delta^{\frac{r}{1-r}}\delta U^{-\frac{r}{1-r}}\right)^{\frac{q(1-r)}{r(1-q)}}d\left[-\left(\int_s^L\Delta^{-\frac{q}{r}}w\right)^{\frac{1}{1-q}}\right]\right)^{\frac{1-q}{q}}\\
		&\quad+
		\sup_{N+1\leq k \leq M-2}U(x_k)\varphi(x_k)^{-\frac1p} \\
		& \qquad \times
		\left(\sum_{i=k+1}^{M-1}\left(\int_{x_k}^{x_i}\Delta^{\frac{r}{1-r}}\delta U^{-\frac{r}{1-r}}\right)^{\frac{q(1-r)}{r(1-q)}}\int_{x_i}^{x_{i+1}}d\left[-\left(\int_s^L\Delta^{-\frac{q}{r}}w\right)^{\frac{1}{1-q}}\right]\right)^{\frac{1-q}{q}}
		\\
		&=:\operatorname{III_1}+	\operatorname{III_2}.
	\end{align*}
	Since $\{U(x_k)/\varphi(x_k)^{\frac{1}{p}}\}_{k=N+1}^{M-1}$ is strongly increasing, using \eqref{EQ:strongly_increasing_sup_sum}, then integration by parts yields
	\begin{align*}
		\operatorname{III_1}
		&\approx
		\sup_{N+1\leq k \leq M-1}U(x_k)\varphi(x_k)^{-\frac1p} \\
		& \qquad \times
		\left(\int_{x_k}^{x_{k+1}}\left(\int_{x_k}^s\Delta^{\frac{r}{1-r}}\delta U^{-\frac{r}{1-r}}\right)^{\frac{q(1-r)}{r(1-q)}}d\left[-\left(\int_s^L\Delta^{-\frac{q}{r}}w\right)^{\frac{1}{1-q}}\right]\right)^{\frac{1-q}{q}}
		\\
		&\lesssim
		\sup_{N+1\leq k \leq M-1}U(x_k)\varphi(x_k)^{-\frac1p} \\
		& \qquad \times
		\left(\int_{x_k}^{x_{k+1}}\left(\int_s^L\Delta^{-\frac{q}{r}}w\right)^{\frac{1}{1-q}}d\left[\left(\int_{x_k}^s\Delta^{\frac{r}{1-r}}\delta U^{-\frac{r}{1-r}}\right)^{\frac{q(1-r)}{r(1-q)}}\right]\right)^{\frac{1-q}{q}}.
	\end{align*}
	Next, as in previous cases, decomposing the integral $\int_s^L$, we obtain
	\begin{align*}
		\operatorname{III_1}
		&\lesssim
		\sup_{N+1\leq k \leq M-1}U(x_k)\varphi(x_k)^{-\frac1p} \\
		& \qquad \times
		\left(\int_{x_k}^{x_{k+1}}\left(\int_s^{x_{k+1}}\Delta^{-\frac{q}{r}}w\right)^{\frac{1}{1-q}}d\left[\left(\int_{x_k}^s\Delta^{\frac{r}{1-r}}\delta U^{-\frac{r}{1-r}}\right)^{\frac{q(1-r)}{r(1-q)}}\right]\right)^{\frac{1-q}{q}}
		\\
		&\quad+
		\sup_{N+1\leq k \leq M-2}U(x_k)\varphi(x_k)^{-\frac1p}
		\left(\int_{x_{k+1}}^L\Delta^{-\frac{q}{r}}w\right)^{\frac{1}{q}}\left(\int_{x_k}^{x_{k+1}}\Delta^{\frac{r}{1-r}}\delta U^{-\frac{r}{1-r}}\right)^{\frac{1-r}{r}}.
	\end{align*}
	Integration by parts and the monotonicity of $U\varphi^{-\frac{1}{p}}$ gives
	\begin{align*}
		\operatorname{III_1} 
		&\lesssim
		\sup_{N+1\leq k \leq M-1}U(x_k)\varphi(x_k)^{-\frac1p}\\
		& \qquad \times
		\left(\int_{x_k}^{x_{k+1}}\left(\int_{x_k}^s\Delta^{\frac{r}{1-r}}\delta U^{-\frac{r}{1-r}}\right)^{\frac{q(1-r)}{r(1-q)}}\left(\int_s^{x_{k+1}}\Delta^{-\frac{q}{r}}w\right)^{\frac{q}{1-q}}\Delta(s)^{-\frac{q}{r}}w(s)\,ds\right)^{\frac{1-q}{q}}
		\\
		&\quad+
		\sup_{N+1\leq k \leq M-2}\left(\int_{x_{k+1}}^L\Delta^{-\frac{q}{r}}w\right)^{\frac{1}{q}}
		\esup_{t\in(0,x_k)}U(t)\varphi(t)^{-\frac1p}
		\left(\int_{t}^{x_{k+1}}\Delta^{\frac{r}{1-r}}\delta U^{-\frac{r}{1-r}}\right)^{\frac{1-r}{r}}
		\\
		&\leq
		\sup_{N+2\leq k \leq M}
		\left(\int_{x_{k-1}}^{x_{k}}\left(\int_{x_{k-1}}^s\Delta^{\frac{r}{1-r}}\delta \varphi^{-\frac{r}{p(1-r)}}\right)^{\frac{q(1-r)}{r(1-q)}} \right. \\
		& \qquad \left. \times \left(\int_s^{x_{k}}\Delta^{-\frac{q}{r}}w\right)^{\frac{q}{1-q}}\Delta(s)^{-\frac{q}{r}}w(s)\,ds\right)^{\frac{1-q}{q}}
		\\
		&\quad+
		\sup_{N+2\leq k \leq M-1}\left(\int_{x_{k}}^L\Delta^{-\frac{q}{r}}w\right)^{\frac{1}{q}}
		\esup_{t\in(0,x_k)}U(t)\varphi(t)^{-\frac1p}
		\left(\int_{t}^{x_{k}}\Delta^{\frac{r}{1-r}}\delta U^{-\frac{r}{1-r}}\right)^{\frac{1-r}{r}}.
	\end{align*}
	Finally, applying \eqref{EQ:antid_vi_*} and \eqref{8}, we obtain
	\begin{align*}
		\operatorname{III_1} 
		&\lesssim
		C_{1,6}+C_{1,5} + C_{4,1}+C_{3,2}.
	\end{align*}
	Also, 
	\begin{align*}
		\operatorname{III_2}
		&=
		\sup_{N+1\leq k \leq M-2}U(x_k)\varphi(x_k)^{-\frac1p} \left(\sum_{i=k+1}^{M-1}\left(\sum_{j=k+1}^i\int_{x_{j-1}}^{x_j}\Delta^{\frac{r}{1-r}}\delta U^{-\frac{r}{1-r}}\right)^{\frac{q(1-r)}{r(1-q)}} \right.
		\\
		&\qquad \quad \left.\times
		\int_{x_i}^{x_{i+1}}d\left[-\left(\int_s^L\Delta^{-\frac{q}{r}}w\right)^{\frac{1}{1-q}}\right]\right)^{\frac{1-q}{q}}.
	\end{align*}
	If $q\leq r$, then $\frac{q(1-r)}{r(1-q)}\leq 1$, which gives
	\begin{align*}
		\operatorname{III_2}
		&\leq
		\sup_{N+1\leq k \leq M-2}U(x_k)\varphi(x_k)^{-\frac1p} \left(\sum_{i=k+1}^{M-1} \sum_{j=k+1}^i \left(\int_{x_{j-1}}^{x_j}\Delta^{\frac{r}{1-r}}\delta U^{-\frac{r}{1-r}}\right)^{\frac{q(1-r)}{r(1-q)}} \right.
		\\
		&\qquad \quad \left.\times
		\int_{x_i}^{x_{i+1}}d\left[-\left(\int_s^L\Delta^{-\frac{q}{r}}w\right)^{\frac{1}{1-q}}\right]\right)^{\frac{1-q}{q}}
	\end{align*}
	and so, changing the order of summation then reindexing $j-1 \mapsto j$ we have,
	\begin{align*}
		\operatorname{III_2}
		&\leq \sup_{N+1\leq k \leq M-2}U(x_k)\varphi(x_k)^{-\frac1p} \\
		& \qquad \times
		\left(\sum_{j=k+1}^{M-1}\left(\int_{x_{j-1}}^{x_j}\Delta^{\frac{r}{1-r}}\delta U^{-\frac{r}{1-r}}\right)^{\frac{q(1-r)}{r(1-q)}}\int_{x_j}^{L}d\left[-\left(\int_s^L\Delta^{-\frac{q}{r}}w\right)^{\frac{1}{1-q}}\right]\right)^{\frac{1-q}{q}}\\
		& = \sup_{N+1\leq k \leq M-2}U(x_k)\varphi(x_k)^{-\frac1p} \\
		& \qquad \times
		\left(\sum_{j=k}^{M-2}\left(\int_{x_{j}}^{x_{j+1}}\Delta^{\frac{r}{1-r}}\delta U^{-\frac{r}{1-r}}\right)^{\frac{q(1-r)}{r(1-q)}}\int_{x_{j+1}}^{L}d\left[-\left(\int_s^L\Delta^{-\frac{q}{r}}w\right)^{\frac{1}{1-q}}\right]\right)^{\frac{1-q}{q}}.
	\end{align*}	
	Since $U\varphi^{-\frac{1}{p}}$ is strongly increasing, applying \eqref{EQ:strongly_increasing_sup_sum} and  \eqref{8} again, we have
	\begin{align*}
		\operatorname{III_2}
		&\lesssim
		\sup_{N+1\leq k \leq M-2}U(x_k)\varphi(x_k)^{-\frac1p}
		\left(\int_{x_{k}}^{x_{k+1}}\Delta^{\frac{r}{1-r}}\delta U^{-\frac{r}{1-r}}\right)^{\frac{1-r}{r}}\left(\int_{x_{k+1}}^{L}\Delta^{-\frac{q}{r}}w\right)^{\frac{1}{q}}	
		\\
		&\leq
		\sup_{N+1\leq k \leq M-2}\left(\int_{x_{k+1}}^{L}\Delta^{-\frac{q}{r}}w\right)^{\frac{1}{q}}\esup_{t\in (0,x_k)}	U(t)\varphi(t)^{-\frac1p}
		\left(\int_{t}^{x_{k+1}}\Delta^{\frac{r}{1-r}}\delta U^{-\frac{r}{1-r}}\right)^{\frac{1-r}{r}}
		\\	
		&\lesssim C_{3,2}.	
	\end{align*}	
	
	On the other hand, if $r<q$, then $\frac{q(1-r)}{r(1-q)}> 1$ and we can apply the Minkowski inequality. Using that and the similar arguments as above, we arrive at
	\begin{align*}
		\operatorname{III_2}
		&\leq	
		\sup_{N+1\leq k \leq M-2}U(x_k)\varphi(x_k)^{-\frac1p}
		\left(\sum_{j=k+1}^{M-1} \left(\int_{x_{j-1}}^{x_j}\Delta^{\frac{r}{1-r}}\delta U^{-\frac{r}{1-r}}\right)\right.\\
		&\qquad\times\left.\left(\sum_{i=j}^{M-1}
		\int_{x_i}^{x_{i+1}}d\left[-\left(\int_s^L\Delta^{-\frac{q}{r}}w\right)^{\frac{1}{1-q}}\right]\right)^{\frac{r(1-q)}{q(1-r)}}\right)^{\frac{1-r}{r}}
		\\
		&=	
		\sup_{N+2\leq k \leq M-2}U(x_k)\varphi(x_k)^{-\frac1p} \\
		& \qquad \times
		\left(\sum_{j=k+1}^{M-1}\left(\int_{x_{j-1}}^{x_j}\Delta^{\frac{r}{1-r}}\delta U^{-\frac{r}{1-r}}\right)
		\left(\int_{x_j}^{L}\Delta^{-\frac{q}{r}}w\right)^{\frac{r}{q(1-r)}}\right)^{\frac{1-r}{r}}
		\\
		&\approx
		\sup_{N+2\leq k \leq M-2}U(x_k)\varphi(x_k)^{-\frac1p}
		\left(\int_{x_{k}}^{x_{k+1}}\Delta^{\frac{r}{1-r}}\delta U^{-\frac{r}{1-r}}\right)^{\frac{1-r}{r}}
		\left(\int_{x_{k+1}}^{L}\Delta^{-\frac{q}{r}}w\right)^{\frac{1}{q}}
		\\
		&\lesssim C_{3,2}.
	\end{align*}
	The estimates above yields
	\begin{align*}
		\III\lesssim C_{1,6}+C_{1,5} +C_{4,1}+C_{3,2};
	\end{align*}
	hence, altogether,
	\begin{equation}\label{EQ:B9lessC16+C32+C41}
		B_8\lesssim C_{1,6}+C_{1,5}+C_{4,1}+C_{3,2}.
	\end{equation}
	Finally, putting all things together, we obtain
	\begin{align*}
		B_1+B_2+B_3+B_7+B_8\lesssim C_{1,5}+C_{1,6}+C_{3,2}+C_{4,1}\lesssim B_1+B_2+B_3+B_7+B_8.
	\end{align*}
	
	\rm{(vii)} By Theorem~\ref{thm:main_discretization}, we have
	\begin{equation*}
		C\approx C_{1,5} + C_{1,6} + C_{3,4} + C_{4,1}.
	\end{equation*}
	
	Now, \eqref{EQ:antid_v_*}, \eqref{EQ:C16lessB1+B9}, \eqref{EQ:C34lessB2+B3+B5} and \eqref{EQ:C41lesB1} combined all together yield the desired upper estimate on $C_{1,5} + C_{1,6} + C_{3,4} + C_{4,1}$.
	
	Finally, we obtain the opposite inequality by combining \eqref{EQ:B1+B2lessC15+C16+C32+C41}, \eqref{EQ:B3lessC16+C32}, \eqref{EQ:B5lessC13+C34}, \eqref{EQ:B8lessC15+C32+C41} and \eqref{EQ:B9lessC16+C32+C41} upon using the fact that $C_{3,2}\leq C_{3,4}$ from \eqref{C-32<C-34-new} and $C_{1,3}\lesssim C_{1,6}$ from  \eqref{C13<C16}.
\end{proof}

\end{document}